\documentclass{amsart}
\usepackage{amsmath, amsthm, amsfonts}
\usepackage{a4wide}
\usepackage{amsmath, array}
\usepackage{url}
\usepackage{enumitem}

\usepackage{amssymb}
\usepackage{graphicx}
\usepackage{yhmath}

\usepackage{epsfig,graphics}
\usepackage[all,knot,poly,color]{xy}
\usepackage{multicol}
\usepackage{array}
\usepackage{makecell}

\newtheorem{thm}{Theorem}[section]
\newtheorem{cor}[thm]{Corollary}

\newtheorem{lem}[thm]{Lemma}
\newtheorem{prop}[thm]{Proposition}
\theoremstyle{definition}

\newtheorem{rem}[thm]{Remark}
\newtheorem{co}[thm]{Conjecture}
\newtheorem{ex}[thm]{Example}

\def\ZZ{\mathbb{Z}}
\def\NN{\mathbb{N}}
\def\CC{\mathbb{C}}

\newcommand{\gra}{{\alpha}} \newcommand{\grb}{{\beta}} \newcommand{\grg}{{\gamma}} \newcommand{\grd}{{\delta}}
 \newcommand{\grz}{{\zeta}}  \newcommand{\gru}{{\theta}}
\newcommand{\gri}{{\iota}}   
   \newcommand{\grp}{{\pi}}
 \newcommand{\grs}{{\sigma}}  
\newcommand{\grf}{{\phi}}  \newcommand{\grc}{{\psi}}

\newcommand{\grF}{{\Phi}}  \newcommand{\grC}{{\Psi}} 
\newcommand{\arw}{\rightarrow} 
\newcommand{\ann}{\Leftrightarrow} 

\title{The BMR freeness conjecture for the tetrahedral and octahedral families}
\author{Eirini Chavli}
\begin{document}
\maketitle
\begin{abstract}
	\noindent
	We prove the validity of the freeness conjecture of Brou\'e,  Malle  and  Rouquier  for  the generic Hecke  algebras  associated to the exceptional complex reflection groups of rank 2 belonging to the tetrahedral and octahedral families, and we give a description of the basis similar
	to the classical case of the finite Coxeter groups.
\end{abstract}

\section{Introduction}
\indent

Between 1994 and 1998, M. Brou\'e, G. Malle, and R. Rouquier 
generalized in a natural way the definition of the Iwahori-Hecke algebra to   arbitrary complex reflection groups (see \cite{bmr}). Attempting to also generalize the properties of the Coxeter case, they stated a number of conjectures concerning the Hecke algebras, which haven't been proven yet. Even without being proven, those conjectures have been used by a number of papers in the last two decades as assumptions, and are still being used in various subjects, such as representation theory of finite reductive groups, Cherednik algebras, and usual braid groups (more details about these conjectures and their applications can be found in \cite{marinreport}).

One specific example of importance, regarding those yet unsolved conjectures, is the so-called freeness conjecture. In 1998, M. Brou\'e, G. Malle and R. Rouquier conjectured that the generic Hecke algebra $H$ associated to a complex reflection group $W$ is a free module of rank $|W|$ over its ring of definition $R$. They also proved that it is enough to show that $H$ is generated as $R$-module by $|W|$ elements. 

The freeness conjecture is fundamental in the world of generic Hecke algebras;
the validity of it, even in its weak version (which states that $H$ is finitely generated as $R$-module), implies that by extending the scalars to an algebraic closure  $F$ of the field of fractions of $R$, the algebra $H\otimes_R F$ becomes isomorphic to the group algebra $FW$ (see \cite{marincubic} and \cite{marinG26}).  
G. Malle assumed the validity of the conjecture and used it to prove that the characters of $H$ take their values in a specific field (see \cite{malle1}). Moreover, he and J. Michel also used this conjecture to provide matrix models for the representations of $H$, whenever we can compute them; these matrices for the generators of $H$ have entries in the field generated by the corresponding character values (see \cite{mallem}). Moreover, once the freeness conjecture is proved, our better knowledge of the generic Hecke algebras could allow the possibility of using various computer algorithms on the structure constants for the multiplication, in order to thoroughly improve our understanding in each case (see for example \S 8 in \cite{mallem} about
the determination of a canonical trace).  

The freeness conjecture has also many applications, apart from the ones connected to the properties of the generic Hecke algebra itself. Provided that the freeness conjecture is true, the category of representations of $H$ is related to a category of representations of a Cherednik algebra (see \cite{Ocat}). Another application is about the algebras connected to cubic
invariants,  including  the  Kaufman  polynomial  and  the  Links-Gould  polynomial. These algebras are quotients of the generic Hecke algebra associated to the exceptional groups $G_4$, $G_{25}$ and $G_{32}$. I. Marin used the validity of the conjecture of these cases and he proved that the generic algebra $K_n(\gra, \grb)$ introduced by P. Bellingeri and L. Funar in \cite{bfunar} is zero for
$n\geq 5$ (see theorem 1.4 in \cite{marincubic}). Furthermore, in \cite{chavli} we used the freeness conjecture for the cases of the exceptional groups $G_4$, $G_8$ and $G_{16}$  to recover and explain a classification due to I. Tuba and H. Wenzl  for the irreducible representations of the braid group on 3 strands of dimension at most 5 (see \cite{tuba}).

Any complex reflection group can be decomposed as a direct product of the so-called irreducible ones (which means
that, considering them as subgroups of the general linear group $GL(V)$, where $V$ is a finite dimensional complex vector space, they act irreducibly on $V$). The irreducible complex reflection
groups were classified by G. C. Shephard and J. A. Todd (see \cite{shephard}); they
belong either to the infinite family $G(de, e, n)$ depending on 3 positive
integer parameters, or to the 34 exceptional groups, which are numbered
from 4 to 37 and are known as $G_4,\dots, G_{37}$, in the Shephard and Todd classification.

The freeness
conjecture, which we also call the BMR freeness conjecture, is known to be true for the finite Coxeter groups (see for example \cite{geck}, lemma 4.4.3.), and also for the infinite series by Ariki and Koike (see \cite{ariki} and \cite{arikii}). Considering the exceptional cases, one may divide them into two families; the family that includes the exceptional groups $G_4, \dots, G_{22}$, which are of rank 2, and the family that includes the rest of them, which are of rank at least 3 and at most 8. Among the second family we encounter 6 finite Coxeter groups for which we know the validity of the conjecture: the groups $G_{23}$, $G_{28}$,
$G_{30}$, $G_{35}$, $G_{36}$ and $G_{37}$. Thus, it remains to prove the conjecture for 28 cases: the exceptional groups of rank 2 and the exceptional groups $G_{24}$, $G_{25}$, $G_{26}$ $G_{27}$,
$G_{29}$, $G_{31}$, $G_{32}$, $G_{33}$ and $G_{34}$.

Until recently, it was
widely believed that the BMR freeness conjecture had been verified for most of
the exceptional groups of rank 2. However, there were flaws and
gaps in the proofs, as I. Marin indicated a few years ago (for more details see the introduction of \cite{marinG26}). In the following years, his own research and his joint work with G. Pfeiffer concluded that the exceptional complex reflection groups for which there is a complete proof for the freeness conjecture  are the groups $G_4$ (this case has also been proved in \cite{brouem} and independently in \cite{funar1995}), $G_{12}$,  $G_{22}$, $G_{23}$,  \dots, $G_{37}$ (see  \cite{marincubic}, \cite{marinG26} and \cite{marinpfeiffer}). Moreover, in \cite{chavli} we proved the cases of $G_8$ and $G_{16}$, completing the proof for the validity of the BMR conjecture for the case of the exceptional groups, whose associated complex braid group is an Artin group.

The remaining cases are almost all the exceptional groups of rank 2. Recent work by I. Losev, and the result of P. Etingof and E. Rains of the validity of the weak version of the BMR freeness conjecture for the exceptional groups of rank 2, 
implies the BMR conjecture for these groups in
characteristic zero (for more details one may refer to \cite{etingof2016}).  However, this result cannot be used to prove the strong version of the conjecture. Moreover, even in characteristic zero, we cannot provide a basis of the Hecke algebra consisting of braid group elements (see \cite{etingof2016}, remark 2.4.3).

The exceptional groups of rank 2 are divided into three families: the tetrahedral, octahedral and icosahedral family. The main goal of this paper is to prove the conjecture for the first two families (including also the cases of $G_4$, $G_8$ and $G_{12}$), by providing a basis consisting of braid group elements, a result that also holds for the finite Coxeter groups, the infinite family and the  exceptional groups $G_{16}$, $G_{22}$, $G_{23}, \dots, G_{37}$. This particular basis not only provides the proof of the BMR freeness conjecture for these two families, but also gives a nice description of it, similar to  the classical case of the finite Coxeter groups. 

The BMR freeness conjecture is still open for the groups $G_{17}, G_{18} \dots, G_{21}$, which are 5 of the 7 exceptional groups belonging to the icosahedral family.  Since these groups are large we are not sure if one can provide computer-free proofs, as for the other exceptional groups of rank 2. 
However, there are strong indications that with continued research, and possibly with the development of computer algorithms, one can prove these final cases.
 
 After this work was complete, I. Marin provided a proof for $G_{20}$ and $G_{21}$, using a different method than the one we explain here (see \cite{marin2017}). This method allowed him to automate completely the calculations. It seems so far that this technique cannot be applied to automatize the calculations of this paper, nor to provide the 3 remaining cases. We are optimistic, however, that the methodology we explain here, combined with a computer approach, could lead to a proof of the conjecture for these final cases.
\\ \\
\textbf{Acknowledgments:} This research is based on my Ph.D. dissertation (see \cite{chavlithesis}) and one can find an announcement of these results in \cite{chavli2}. I want to thank I. Marin for his support during this research and for suggesting the proof provided for lemma \ref{ll}, which is simpler than the initial one. I also thank A. Esterle and S. Koenig for a careful reading of this paper.
\section{Preliminaries}
\subsection{Generic Hecke algebras}
Let $W$ be a complex reflection group on a finite dimensional $\CC$-vector space $V$. We say that $W$ is of rank $n$, where $n$ denotes the dimension of $V$. We let $\mathcal{R}$ denote the set of pseudo-reflections of $W$, $\mathcal{A}=$\{ker$(s-1)\;|\;s\in \mathcal{R}\}$ the hyperplane arrangement associated to $\mathcal{R}$, and $X=V\setminus \cup \mathcal{A}$ the corresponding hyperplane complement. We assume that $\mathcal{A}$ is essential, meaning that $\cap_{H\in \mathcal{A}}H=\{0\}$. By Steinberg's theorem (see \cite{steinberg}) we have that the action of $W$ on $X$ is free. Therefore, it defines a Galois covering $X\rightarrow X/W$, which gives rise to the following exact sequence, for every $x\in X$:
$$1\rightarrow \grp_1(X,x)\rightarrow \grp_1(X/W, \underline{x})\rightarrow W\rightarrow 1,$$
where $\underline{x}$ denotes the image of $x$  under the canonical surjection $ X\rightarrow X/W$.
M. Brou\'e, G. Malle and R. Rouquier defined the \emph{complex braid group} associated to $W$ as $B:=\grp_1(X/W, \underline{x})$. Moreover, they associated to every $s\in \mathcal{R}$ homotopy classes in $B$, that we call \emph{braided reflections}  (for more details one may refer to \cite{bmr}). 

A pseudo-reflection $s$ is called \emph{distinguished} if its only nontrivial eigenvalue on $V$ equals $e^{-2\grp \gri k/e_s}$, where  $\gri$ denotes a chosen imaginary unit (a solution of the equation $x^2=-1$) and $e_s$ denotes the order of $s$ in $W$. Let $S$ denote the set of the distinguished pseudo-reflections of $W$. For each $s\in S$ we choose a set of $e_s$ indeterminates $u_{s,1},\dots, u_{s,e_s}$, such that $u_{s,i}=u_{t,i}$ if $s$ and $t$ are conjugate in $W$. We denote by $R$ the Laurent polynomial ring $\ZZ[u_{s,i},u_{s,i}^{-1}]$. The \emph{generic Hecke algebra} $H$ associated to $W$ with parameters $u_{s,1},\dots, u_{s,e_s}$ is the quotient of the group algebra $RB$ of $B$ by the ideal generated by the elements of the form 
	\begin{equation}
	(\grs-u_{s,1})(\grs-u_{s,2})\dots (\grs-u_{s,e_s}),
	\label{Hecker}
	\end{equation}
	where $s$ runs over the conjugacy classes of $S$ and $\grs$ over the set of braided reflections associated to the pseudo-reflection $s$. 
It is enough to choose one relation of the form described in (\ref{Hecker}) per conjugacy class, since  the corresponding braided reflections are conjugate in $B$. 

When $W$ is a finite Coxeter group (also called a real reflection group), the generic Hecke algebra associated to it is known as the \emph{Iwahori-Hecke algebra}.

 Let $\grf: R\rightarrow \CC$ be the specialization morphism defined as
		$u_{s,k}\mapsto e^{-2\grp \gri k/e_s}$, where $1\leq k\leq e_c$ and $\gri$ denotes the chosen imaginary unit that defines $s$. Therefore, $H\otimes_{\grf}\CC=\CC B/(\grs^{e_s}-1)=\CC \big(B/(\grs^{e_s}-1)\big)=\CC W$, meaning that $H$ is a deformation of the group algebra of $W$.
		
		We have the following conjecture due to M. Brou\'e, G. Malle and R. Rouquier (see \cite{bmr}). This conjecture is known to be true in the real case i.e. for the Iwahori Hecke algebras (see for example \cite{geck}, lemma 4.4.3.). 
		
		\begin{co} (The BMR freeness conjecture) The generic Hecke algebra $H$ is a free module over $R$ of rank $|W|$.
		\end{co}
		 The next proposition (theorem 4.24 in \cite{bmr} or proposition 2.4(1) in \cite{marinG26}) states that in order to prove the validity of the BMR conjecture, it is enough to find a spanning set of $H$ over $R$ of $|W|$ elements.
		\begin{prop}
			If $H$ is generated as $R$-module by $|W|$ elements, then it is a free module over $R$ of rank $|W|$.
			\label{BMR PROP}
		\end{prop}
		We know that every complex reflection group is a direct product of \emph{irreducible complex reflection groups} (see, for example, proposition 1.27 in \cite{lehrer}). As a result, the proof of the BMR freeness conjecture reduces to the irreducible case. Due to  Shephard-Todd classification (see \cite{shephard}), any irreducible complex reflection group belongs either to the infinite family $G(de,e,n)$ or to 
				the 34 exceptional groups denoted as $G_4, \dots, G_{37}$. Thanks to S. Ariki and S. Ariki and K. Koike (see \cite{ariki} and \cite{arikii}) we have the validity of the conjecture for the  infinite family $G(de,e,n)$. Moreover, we  know the validity of the conjecture for the groups  $G_{23}$, $G_{28}$,
				$G_{30}$, $G_{35}$, $G_{36}$ and $G_{37}$, since these groups are finite Coxeter groups. 
				
				Among the 28 remaining cases, we encounter 6 groups whose associated complex braid group is an Artin group; the groups $G_4$, $G_8$ and $G_{16}$, which are related to the Artin  group of Coxeter type $A_2$, and the groups $G_{25}$, $G_{26}$ and $G_{32}$, which are related to the Artin group of Coxeter type $A_3$, $B_3$ and $A_4$, respectively.   
				The next theorem summarizes the results found in \cite{chavli}, \cite{marincubic} and \cite{marinG26}.
		\begin{thm}
			The BMR freeness conjecture holds for the exceptional groups, whose associated braid group is an Artin group.
			\label{braid}
			\end{thm}
			Exploring the rest of the cases, we notice that we encounter 9 groups generated by reflections (i.e. pseudo-reflections of order 2): These groups are the exceptional groups $G_{12}$, $G_{13}$, $G_{22}$, which are of rank 2, and the exceptional groups $G_{24}$, $G_{27}$,
			$G_{29}$, $G_{31}$, $G_{33}$ and $G_{34}$ of rank at least 3 and at most 6. These exceptional groups are also known as the \emph{2-exceptional groups}. The next theorem is due to 
			I. Marin and G. Pfeiffer (see \cite{marinpfeiffer}), who proved the BMR freeness conjecture for all the 2-exceptional groups apart from the case of $G_{13}$.
		\begin{thm}
			The BMR freeness conjecture holds for all the 2-exceptional groups with a single reflection class.
			\label{2case}
		\end{thm}
			To sum up, the BMR freeness conjecture is still open for the exceptional groups $G_5$, $G_6$, $G_7$, $G_9$, $G_{10}$, $G_{11}$, $G_{13}$, $G_{14}$, $G_{15}$, $G_{17}$, $G_{18}, \dots, G_{21}$. These  groups cover almost all the exceptional groups of rank 2.  
			The rest of this paper is devoted to the proof of 9 of these 14 remaining cases, including also an alternative proof for the cases of $G_4$, $G_8$ and $G_{12}$. Moreover, after this work was complete, I. Marin proved two more cases of the BMR freeness conjecture (see \cite{marin2017}).

		\subsection{The exceptional groups of rank 2}Let $W$ be an exceptional irreducible complex reflection group of rank 2. Using the Shephard-Todd notation, this means that $W$ is one of the groups $G_4, G_5, \dots, G_{22}$. We know that these groups fall into 3 families, according to whether the group $\overline{W}:=W/Z(W)$ is the tetrahedral, octahedral or icosahedral group (for more details one may refer to Chapter 6 of \cite{lehrer}); the first family, known as \emph{the tetrahedral family}, includes the groups $G_4,\dots, G_7$, the second one, known as the \emph{octahedral family} includes the groups $G_8,\dots, G_{15}$ and the last one, known as the \emph{icosahedral family}, includes the rest of them, which are the groups $G_{16},\dots, G_{22}$. 
		
		In each family, there is a maximal group of order $|\overline{W}|^2$ and all the other groups are its subgroups. These are the groups $G_7$, $G_{11}$ and $G_{19}$. Moreover, the group $\overline{W}$ is the subgroup of a finite Coxeter group $C$ of rank 3 (of type $A_3$, $B_3$ and $H_3$ for the tetrahedral, octahedral and icosahedral family, respectively), consisting of the elements of even Coxeter length.

		We know that for every exceptional group of rank 2 we have  a Coxeter-like presentation; that is a presentation of the form 
		$$\langle s\in S\;|\; \{v_i=w_i\}_{i\in I} , \{s^{e_s}=1\}_{s\in S}\rangle,$$
		where $S$ is a finite set of distinguished  reflections and $I$ is a finite set of relations such that, for each $i\in I$,  $v_i$ and $w_i$ are positive words with the same length in elements of $S$. We also know that for the associated complex braid group $B$ we have an Artin-like presentation; that is a presentation of the form 
		$$\langle \mathbf{s}\in \mathbf{S}\;|\; \mathbf{v_i}=\mathbf{w_i} \rangle_{i\in I},$$
		where $\mathbf{S}$ is a finite set of distinguished braided reflections and $I$ is a finite set of relations such that, for each $i\in I$,  $\mathbf{v_i}$ and $\mathbf{w_i}$ are positive words in elements of $\mathbf{S}$.
		We call these presentations \emph{the BMR presentations}, due to M. Brou\'e, G. Malle and R. Rouquier.
		
		In 2006 P. Etingof and E. Rains gave different presentations of $W$ and $B$, based on the BMR presentations associated to the maximal groups $G_7$, $G_{11}$ and $G_{19}$ (see \textsection 6.1 of \cite{ERrank2}). We call these presentations \emph{the ER presentations}.
		In tables \ref{t3} and \ref{t2} of Appendix \ref{ap} we give the two representations for every $W$ and $B$, as well as the isomorphisms between the BMR and ER presentations. Notice that for the maximal groups, the ER presentations coincide with the BMR presentations.

		\subsection{Deformed Coxeter group algebras} Let $W$ be an exceptional group of rank 2 and let $C$ be a finite Coxeter group of type either $A_3$, $B_3$ or $H_3$ with Coxeter system $y_1,y_2, y_3$ and Coxeter matrix $(m_{ij})$. We set $\tilde \ZZ:=\ZZ\left[e^{\frac{2\pi i}{m_{ij}}}\right]$. In \textsection 2 of \cite{ERcoxeter}, P. Etingof and E. Rains defined an $\tilde\ZZ$-algebra, which they call $A(C)$,  presented as follows:
		\begin{itemize}[leftmargin=*]
			\item \underline{Generators}: $Y_1, Y_2, Y_3$, $t_{ij,k}$, where $i,j\in\{1,2,3\}$, $i\not=j$ and $k\in\ZZ/m_{ij}\ZZ$.
			\item \underline{Relations}: $Y_i^2=1$, $t_{ij,k}^{-1}=t_{ji,-k}$, $\prod\limits_{k=1}^{m_{ij}}(Y_iY_j-t_{ij,k})=0$, $t_{ij,k}Y_r=Y_rt_{ij,k}$, $t_{ij,k}t_{i'j',k'}=t_{i'j',k'}t_{ij,k}$.
		\end{itemize}
		This construction of $A(C)$ is more general and can be done also for any Coxeter group, not necessarily finite.
		Let $R^C=\tilde \ZZ\left[t_{ij,k}^{\pm}\right]=\tilde \ZZ\left[t_{ij,k}\right]$.
		The algebra $A(C)$ is naturally an $R^C$-algebra. The sub-$R^C$-algebra $A_+(C)$ generated by $Y_iY_j$, $i\not=j$ can be presented as follows:
		\begin{itemize}[leftmargin=*]
			\item \underline{Generators}: $A_{ij}:=Y_iY_j$, where $i,j\in\{1,2,3\}$, $i\not=j$.
			\item \underline{Relations}: $A_{ij}^{-1}=A_{ji}$, $\prod\limits_{k=1}^{m_{ij}}(A_{ij}-t_{ij,k})=0$, $A_{ij}A_{jl}A_{li}=1$, for $\#\{i,j,l\}=3$.
		\end{itemize}
		\begin{lem}
			The relation $\prod\limits_{k=1}^{m_{ij}}(A_{ij}-t_{ij,k})=0$, $i\not =j$ implies that $\prod\limits_{k=1}^{m_{ij}}(A_{ji}-t_{ji,k})=0$.
			\label{ll}
				\end{lem}
				\begin{proof} By definition, $A_{ij}$ and $t_{ij,k}$ are invertible.
				Therefore, 	
					$$\prod\limits_{k=1}^{m_{ij}}(A_{ij}-t_{ij,k})=0\ann \prod\limits_{k=1}^{m_{ij}}A_{ij}t_{ij,k}(t_{ij,k}^{-1}-A_{ij}^{-1})=0\ann\prod\limits_{k=1}^{m_{ij}}(t_{ij,k}^{-1}-A_{ij}^{-1})=0,$$
					the second equivalence since $A_{ij}$ commutes with each term of the product. Now, since $A_{ij}^{-1}=A_{ji}$ and $t_{ij,k}^{-1}=t_{ji,-k}$ the last equality above reads
					$$\prod\limits_{k=1}^{m_{ij}}(A_{ji}-t_{ji,-k})=0\ann  \prod\limits_{k=1}^{m_{ij}}(A_{ji}-t_{ji,m_{ij}-k})=0,$$ where the equivalence results from the fact that $k\in\ZZ/m_{ij}\ZZ$. Taking into account that the terms in the last product commute pairwise, our final equality is as stated in the lemma.
			\end{proof}
	
		\begin{lem}Let $C$ be a finite Coxeter group of type either $A_3$, $B_3$ or $H_3$. We can present the $R^C$ algebra $A_+(C)$ as follows:$$\left\langle
			\begin{array}{l|cl}
			&(A_{13}-t_{13,1})( A_{13}-t_{13,2})=0&\\
			A_{13},  A_{32},  A_{21}&(A_{32}-t_{32,1})( A_{32}-t_{32,2})( A_{32}-t_{32,3})=0,&  A_{13} A_{32} A_{21}=1\\
			&( A_{21}-t_{21,1})( A_{21}-t_{21,2})\dots( A_{21}-t_{21,m})=0&
			\end{array}\right\rangle,
			$$
			where $m$ is 3, 4 or 5 for each type, respectively.
		\label{presentation}
		\end{lem}
		\begin{proof}
		The Coxeter matrix is of the form $\left( \begin{array}{ccc}
			1 & m & 2 \\
			m & 1 & 3 \\
			2 & 3 & 1 \end{array} \right)$, where $m$ is 3, 4 or 5 for each type, respectively.
		 By definition, the algebra $A_+(C)$ is generated by the elements $A_{ij}$, $i\not=j$. The result follows from the fact that $A_{ij}=A_{ji}^{-1}$ and  from lemma \ref{ll}.
			\end{proof}
		If $w$ is a word in letters $y_i$ we let $T_w$ denote the corresponding word in $Y_i$, an element of 
		$A(C)$. For every $x\in \overline{W}$ let us choose a reduced word $w_x$ that represents $x$ in $\overline{W}$. We notice that $T_{w_x}$ is an element in $A_+(C)$, since $w_x$ is reduced and $\overline{W}$ contains the elements of $C$ of even Coxeter length. 
		\begin{ex}
			Let $W$ be one of the exceptional groups belonging to the octahedral family, meaning that $\overline{W}$ is the octahedral group. As we mentioned before, the octahedral group is the subgroup of the finite Coxeter group of type $B_3$, consisting of elements of even Coxeter length. For the reduced word  $w_x=y_1y_2y_1y_3$ we have that $T_{w_x}=A_{12}A_{13}$.
		\qed \end{ex}
		The following theorem is theorem 2.3(ii) in \cite{ERcoxeter}.
		\begin{thm}
			The algebra $A_+(C)$ is generated as $R^C$-module by the elements $T_{w_x}$, $x\in \overline{W}$.
			\label{thmER}
		\end{thm}
		\section{The BMR freeness conjecture for the first two families}
		\subsection{The connection between the algebras $H$ and $A_+(C)$}
		\label{s}
		Let $W$ be an exceptional group of rank 2 with associated complex braid group $B$ and generic Hecke algebra $H$, defined over $R$.
		Following the notations of \textsection 2.2 of \cite{marinG26}, we set 
		$R_{\tilde{\ZZ}}:=R\otimes_{\ZZ}\tilde{\ZZ}$ and $H_{\tilde{\ZZ}}:=H\otimes_{R}R_{\tilde{\ZZ}}$.  We denote by $\tilde u_{s,i}$ the images of $u_{s,i}$ inside $R_{\tilde{\ZZ}}$ and by $\mathcal{S}$ a set of representatives of the conjugacy classes of $S$.
		By definition, $H_{\tilde{\ZZ}}$ is the quotient of the group algebra $R_{\tilde{\ZZ}}B$ of $B$ by the ideal generated by  $P_{s}(\grs)$, where $s$ runs over $\mathcal{S}$, $\grs$ over the set of  braided reflections associated to $s$ and $P_{s}\left[X\right]$ are the monic polynomials $(X-\tilde u_{s,1})\dots(X-\tilde u_{s,e_s})$ inside $R_{\tilde{\ZZ}}\left[X\right]$. Notice that, if $s$ and $t$ are conjugate in $W$, the polynomials $P_s(X)$ and $P_t(X)$ coincide. 
		
		Let $Z(B)$ denote the center of $B$ 
		 and let $z \in Z(B)$. We  set $\bar B:=B/\langle z\rangle$ and $R_{ \tilde{\ZZ}}^+:=R_{\tilde{\ZZ} }\left[x,x^{-1}\right]$. For every $b\in B$ we denote by $\bar b$ the image of $b$ under the natural projection $B\rightarrow \bar B$. Let $f$ be a set-theoretic section of the natural projection $\grp: B\arw \bar B$, meaning that $f: \bar B \arw B$  is a map such that $\grp\circ f=id_{\bar B}$. The following proposition rephrases proposition 2.10 in \cite{marinG26}.
		 \begin{prop}
		 	$H_{\tilde{\ZZ}}$ inherits a structure of $R_{\tilde{\ZZ}}^+$-module, where $x$ acts by the image of $z$ in $H_{\tilde{\ZZ}}$.
		  Moreover, there is an isomorphism $\grF_f$ between the $R_{\tilde{\ZZ}}^+$-modules $H_{ \tilde{\ZZ}}$ and $R_{ \tilde{\ZZ}}^+\bar B/\{Q_s(\bar\grs)\}_{s\in\mathcal{S}}$,  where $Q_s(X)=x^{c_{\grs}\text{deg}P_s}\cdot
		P_s(x^{-c_{\grs}}\cdot X)\in R_{ k}^+[X]$, the $c_{\grs}\in \ZZ$ being defined by $f(\bar\grs)=z^{c_{\grs}}\grs$.
		\label{propERI}
	\end{prop}
		In the next two propositions we relate the algebra $A_+(C)$ with the algebra $H_{\tilde \ZZ}$. 
	\begin{prop}
			Let $W$ be an exceptional group of rank 2, apart from 	$G_{13}$ and $G_{15}$. There is a ring morphism $\gru: R^C\twoheadrightarrow R_{\tilde \ZZ}^+$ inducing  $ \grC: A_+(C)\otimes_{\gru} R_{\tilde \ZZ}^+ \twoheadrightarrow R_{\tilde \ZZ}^+\bar B/\{Q_{s}(\bar \grs)\}_{s\in\mathcal{S}}$ through $A_{13}\mapsto \bar \gra$, $A_{32}\mapsto \bar \grb$, $A_{21}\mapsto \bar \grg$, where $\bar \gra$, $\bar \grb$ and $\bar \grg$ are the images in $\bar B$ of the generators of the ER representation of $B$ (see Appendix \ref{ap}, table \ref{t2}).
			\label{ERSUR}
		\end{prop}
		\begin{proof}
			From lemma \ref{presentation} we have that $A_+(C)$ can be presented as follows:
			\begin{equation}\left\langle
			\begin{array}{l|cl}
			&(A_{13}-t_{13,1})( A_{13}-t_{13,2})=0&\\
			A_{13},  A_{32},  A_{21}&(A_{32}-t_{32,1})( A_{32}-t_{32,2})( A_{32}-t_{32,3})=0,&  A_{13} A_{32} A_{21}=1\\
			&( A_{21}-t_{21,1})( A_{21}-t_{21,2})\dots( A_{21}-t_{21,m})=0&
			\end{array}\right\rangle,
			\label{paki}
			\end{equation}
			where $m$ is 3, 4 or 5 for each type, respectively.
			
				We now consider the ER presentation of  $B$ (see Appendix \ref{ap}, table \ref{t2}) and we notice that in every case the group $\bar B$ can be presented as follows:
				\begin{equation}
				\langle\bar \gra, \bar \grb, \bar \grg\;|\;\bar \gra ^{k_{\gra}}=\bar \grb ^{k_{\grb}}=\bar \grg ^{k_{\grg}}=1,\; \bar \gra \bar \grb \bar \grg=1	\rangle,\label{brr}\end{equation}
				where $k_{\gra}\in\{0,2\}$, $k_{\grb}\in\{0,3\}$ and the values of $k_{\grg}$ depend on the family in which the group belongs; for the tetrahedral family $k_{\grg}=0$, for the octahedral family $k_{\grg}\in\{0,4\}$ and for the icosahedral family $k_{\grg}\in\{0,5\}$. Comparing relations (\ref{paki}) and (\ref{brr}), we define $\gru$ by distinguishing the following cases:
			\begin{itemize}[leftmargin=*]
				\item[-]When $k_{\gra}=2$ (cases of $G_4$, $G_5$, $G_8$, $G_{10}$, $G_{16}$, $G_{18}$, $G_{20}$) we have $(\bar\gra-1)(\bar\gra+1)=0$. Therefore, we may define $\gru(t_{13,1}):=1$ and $\gru(t_{13,2}):=-1$.
				
				\item[-] When $k_{\gra}=0$ (cases of $G_6$, $G_7$, $G_9$, $G_{11}$, $G_{12}$, $G_{14}$, $G_{17}$, $G_{19}$, $G_{21}$, $G_{22}$) we notice that  $\gra$  is  a distinguished braided reflection associated to a distinguished reflection $a$ of order 2 (see in tables \ref{t3} and \ref{t2} of Appendix \ref{ap} the images of $a$ and $\gra$ in BMR presentations). As we saw in proposition \ref{propERI}, $\bar \gra$ annihilates the polynomial $Q_{a}(X)=(X-x^{c_{\gra}}\cdot \tilde u_{a,1})(X-x^{c_{\gra}} \cdot\tilde u_{a,2})$. Therefore, we may define $\gru(t_{13,1}):=x^{c_{\gra}} \cdot\tilde u_{a,1}$ and $\gru(t_{13,2}):=x^{c_{\gra}}\cdot \tilde u_{a,2}$.
				
				\item[-]When $k_{\grb}=3$ (cases of $G_4$, $G_6$, $G_8$, $G_9$, $G_{12}$, $G_{16}$, $G_{17}$, $G_{22}$) we have $(\bar\grb-1)(\bar\grb-j)(\bar\grb-j^2)=0$,  where $j$ is a third root of unity. Therefore, we define $\gru(t_{32,1}):=1$, $\gru(t_{32,2}):=j$
				and $\gru(t_{32,3}):=j^2$.
				
				\item[-] When $k_{\grb}=0$ (cases of $G_5$, $G_7$, $G_{10}$, $G_{11}$, $G_{14}$, $G_{18}$, $G_{19}$, $G_{20}$, $G_{21}$) we notice that $\grb$ is a distinguished braided reflection associated to a distinguished reflection $b$ of order 3 (see in tables \ref{t3} and \ref{t2} of Appendix \ref{ap} the images of $b$ and $\grb$ in BMR presentations). Therefore, similarly to the case where $k_{\gra}=0$, we may define $\gru(t_{32,1}):=x^{c_{\grb}} \cdot\tilde u_{b,1}$, $\gru(t_{32,2}):=x^{c_{\grb}} \cdot\tilde u_{b,2}$ and $\gru(t_{32,3}):=x^{c_{\grb}}\cdot \tilde u_{b,3}$.
				\item[-] When $k_{\grg}=m$ (cases of $G_{12}$, $G_{14}$, $G_{20}$, $G_{21}$, $G_{22}$) we have $(\bar\grg-1)(\bar\grg-\grz_m)(\bar\grg-\grz_m^2)\dots(\bar\grg-\grz_m^{m-1})=0$, where $\grz_m$ is a $m$-th root of unity. Therefore, we define 
				$\gru(t_{21,1}):=1$, $\gru(t_{21,2}):=\grz_m$, $\gru(t_{21,3}):=\grz_m^2$,\dots,
				$\gru(t_{21,m}):=\grz_m^{m-1}$.
				\item[-] When $k_{\grg}=0$ (cases of $G_4$, $G_5$, $G_6$, $G_7$, $G_8$, $G_9$, $G_{10}$, $G_{11}$, $G_{16}$, $G_{17}$, $G_{18}$, $G_{19}$) we notice that $\grg$ is a distinguished braided reflection  associated to a distinguished reflection $c$ of order $m$ (see in tables \ref{t3} and \ref{t2} of Appendix \ref{ap} the images of $\grg$ and $c$ in BMR presentations). Therefore, similarly to the case where $k_{\gra}=0$, we may define $\gru(t_{21,1}):=x^{c_{\grg}}\cdot \tilde u_{c,1}$, $\gru(t_{21,2}):=x^{c_{\grg}}\cdot\tilde u_{c,2}$, \dots, $\gru(t_{21,m}):=x^{c_{\grg}} \cdot\tilde u_{c,m}$.
				\qedhere
			\end{itemize}
		\end{proof} 
		We will now deal with the cases of $G_{13}$ and $G_{15}$. These groups belong to the octahedral family and, hence, $C$ is of type $B_3$. We  replace $A_+(C)$ with a specialized algebra $\tilde A_+(C)$ and we use the same technique as in proposition \ref{ERSUR}. 
		
		More precisely, 
	we set $\tilde R^C:=\tilde \ZZ\left[t_{13,1}, t_{13,2}, t_{32,1},  t_{32,2},  t_{32,3}, \sqrt{t_{21,1}},\sqrt{t_{21,3}},\right]$
					and let $ \grf: R^C \rightarrow\tilde R^C$, defined by 
					$t_{21,1}\mapsto \sqrt{t_{21,1}}$, 
					$t_{21,2}\mapsto -\sqrt{t_{21,1}}$,
					$t_{21,3}\mapsto \sqrt{t_{21,3}}$ and 
					$t_{21,4}\mapsto -\sqrt{t_{21,3}}$.
					Let $\tilde A_+(C)$ denote the  $\tilde R^C$ algebra $A_+(C)\otimes_{\grf}\tilde R^C$.
		
		\begin{prop}
			Let $W$ be the exceptional group $G_{13}$ or $G_{15}$. There is a ring morphism $\gru: \tilde R^C\twoheadrightarrow R_{\tilde \ZZ}^+$ inducing  $ \grC: \tilde  A_+(C)\otimes_{\gru} R_{\tilde \ZZ}^+ \twoheadrightarrow R_{\tilde \ZZ}^+\bar B/\{Q_s(\bar\grs)\}_{s\in\mathcal{S}}$ through $\tilde A_{13}\mapsto \bar \gra$, $\tilde A_{32}\mapsto \bar \grb$, $\tilde A_{21}\mapsto \bar \grg$,  where $\bar \gra$, $\bar \grb$ and $\grg$ are as in Proposition \ref{ERSUR}.
			\label{ERRSUR}
		\end{prop}
		\begin{proof}
		We consider the ER presentation of the complex braid group $B$ associated to the groups $G_{13}$ and $G_{15}$ (see table \ref{t2} of Appendix \ref{ap}) and we notice that $\bar B$ can be presented as follows:
			\begin{equation}\langle\bar \gra, \bar \grb, \bar \grg\;|\;\bar \grb^{k_{\grb}}=\bar \grg^{4}=1, \bar \gra \bar \grb \bar \grg=1	\rangle, \text{ where }	k_{\grb}\in\{0,3\}.
			\label{g13}
			\end{equation}
	Let $\tilde A_{ij}$ denote the image of $A_{ij}$ inside the algebra $\tilde A_+(C)$. The latter can be presented as follows (see lemma \ref{presentation}) :
			\begin{equation}
		\left\langle
				\begin{array}{l|cl}
				&(\tilde A_{13}-t_{13,1})(\tilde A_{13}-t_{13,2})=0&\\
				\tilde A_{13}, \tilde A_{32}, \tilde A_{21}&(\tilde A_{32}-t_{32,1})(\tilde A_{32}-t_{32,2})(\tilde A_{32}-t_{32,3})=0,& \tilde A_{13}\tilde A_{32}\tilde A_{21}=1\\
				&(\tilde A_{21}^2-t_{21,1})(\tilde A_{21}^2-t_{21,3})=0&
				\end{array}
				\right\rangle.
				\label{gg13}
				\end{equation}
			Comparing relations (\ref{g13}) and (\ref{gg13}), we define $\gru$ as follows: since $(\bar{\grg}^2-1)(\bar{\grg}^2+1)=0$, we may define $\gru(t_{21,1})=1$ and $\gru(t_{21,3})=-1$. Moreover, we notice that $\gra$ is a distinguished braided reflection associated to a distinguished reflection a of order 2 (see in tables \ref{t3} and \ref{t2} of Appendix \ref{ap} the images of $\gra$ and $a$ in BMR presentations). 
			By proposition \ref{propERI} we have that $\bar \gra$ annihilates the polynomial $Q_{a}(X)=(X-x^{c_{\gra}}\cdot \tilde u_{a,1})(X-x^{c_{\gra}} \cdot\tilde u_{a,2})$. Therefore, we may define $\gru$ such that $\gru(t_{13,1}):=x^{a_s} \tilde u_{s_1}$ and $\gru(t_{13,2}):=x^{a_s} \tilde u_{s_2}$.
			
			It remains to define $\gru(t_{32,1})$, $\gru(t_{32,2})$ and $\gru(t_{32,3})$. 
			In the case of $G_{13}$ we have $(\bar\grb-1)(\bar\grb-j)(\bar\grb-j^2)=0$, where $j$ is a third root of unity. Therefore, we may define $\gru(t_{32,1}):=1$, $\gru(t_{32,2}):=j$
			and $\gru(t_{32,3}):=j^2$.
			In the case of $G_{15}$ we notice that $\grb$ is a distinguished braided reflection  associated to a distinguished reflection $b$ of order 3 (see in tables \ref{t3} and \ref{t2} of Appendix \ref{ap} the images of $\grb$ and $b$ in BMR presentations). By proposition \ref{propERI} we have that $\bar \grb$ annihilates the polynomial $Q_{b}(X)=(X-x^{c_{\grb}}\cdot \tilde u_{b,1})(X-x^{c_{\grb}} \cdot\tilde u_{b,2})(X-x^{c_{\grb}} \cdot\tilde u_{b,3})$. Therefore, we may define $\gru(t_{32,1}):=x^{c_{\grb}} \cdot\tilde u_{b,1}$, $\gru(t_{32,2}):=x^{c_{\grb}} \cdot\tilde u_{b,2}$ and $\gru(t_{32,3}):=x^{c_{\grb}}\cdot \tilde u_{b,3}$.
			\end{proof} 
		For every exceptional group of rank 2 we call the surjection $\grC$ as described in propositions \ref{ERSUR} and \ref{ERRSUR} the \emph{ER surjection} associated to $W$.
		
		\subsection{The weak version of the BMR freeness conjecture: a new application}
		Let $W$ be a complex reflection group and $H$ the associated generic Hecke algebra.
		
		\begin{prop}(A weak version of the BMR freeness conjecture)
			$H$ is finitely generated over its ring of definition $R$.
			\label{corER}
		\end{prop}
		\begin{proof}
			The weak version of the BMR freeness conjecture is valid  for the infinite family and for the exceptional groups $G_{23}, \dots, G_{37}$, since for these cases we have the strong version of the conjecture (see theorems \ref{braid} and \ref{2case}). It remains to prove the proposition for the exceptional groups of rank 2. For these groups we have the weak version of the conjecture due to theorem 6.1 in \cite{ERrank2}. The idea, explaining in detail in \cite{marinG26}, is the following: by proposition \ref{propERI} we have that for every set theoretic section $f$ of the natural projection $B\rightarrow \bar B$ we can define an isomorphism $\grF_f$ of
		$R_{\tilde{\ZZ}}^+$-modules between $H_{\tilde{\ZZ}}$ and $R_{ \tilde{\ZZ}}^+\bar B/Q_s\{\bar\grs\}_{s\in\mathcal{S}}$. Therefore, by theorem \ref{thmER} and propositions \ref{ERSUR} and \ref{ERRSUR} we have that $H_{\tilde{\ZZ}}$ is finitely generated as $R_{\tilde \ZZ}^+$-module. However, $R_{\tilde{\ZZ}}$ is free of finite rank over $\ZZ$ (since $\tilde{\ZZ}$ is free $\ZZ$-module of finite rank). Therefore, $H_{\tilde{\ZZ}}$ is finitely generated over $R$ and, hence, $H$ is finitely generated over $R$, since $R$ is noetherian.
	\end{proof}
	A first consequence of the weak version of the conjecture is the following:
	\begin{prop}
	Let $F$ denote the field of fractions of $R$ and $\bar F$ an algebraic closure. Then, $H\otimes_R \bar F\simeq \bar F W$.
		\label{F}
	\end{prop}
	\begin{proof}
		The result follows directly from proposition 2.4(2) of \cite{marinG26} and from proposition \ref{corER}.
	\end{proof}
	Another consequence of the weak version of the conjecture is the following proposition, which states that if the generic Hecke algebra of an exceptional group of rank 2 is torsion-free as $R$-module, it will be sufficient to prove the BMR freeness conjecture for the maximal groups.
		\begin{prop}
			Let $W_0$ be the maximal group $G_7$, $G_{11}$ or $G_{19}$ and	let $W$ be an exceptional group of rank 2, whose associated Hecke algebra $H$ is torsion-free as $R$-module. If the BMR freeness conjecture holds for $W_0$, then it holds for $W$, as well.
			\label{torsion}
		\end{prop}
		\begin{proof}
			Let $R_0$ and $R$ be the rings over which we define the Hecke algebras $H_0$ and $H$ associated to $W_0$ and $W$, respectively. There is a specialization $\gru :R_0\rightarrow R$, that maps some of the parameters of $R_0$ to roots of unity (see tables 4.6, 4.9 and 4.12 in \cite{malle2}). We set $A:=H_0\otimes_{\gru}R$. Due to hypothesis that the BMR freeness conjecture holds for $W_0$, we have that $A$ is a free $R$-module of rank $|W_0|$. 
			
			In proposition 4.2 in \cite{malle2}, G. Malle found a subalgebra $\bar A$ of $A$, such that $H\twoheadrightarrow \bar A$. A presentation of $\bar A$ is given in Appendix A of \cite{chlouverakibook}. He also noticed that if $m=|W_0|/|W|$, then there is an element $\grs\in A$ of order $m$ such that $A=\oplus_{i=0}^{m-1}\grs^i\bar A$. Since $A$ is a free $R$-module of rank $|W_0|$, we also have that $\bar A$ is a free module of rank $|W|$.
			
		We highlight here that for all these results G. Malle does not use the validity of the BMR freeness conjecture. However, if we assume the validity of the conjecture, one may have that the surjection $H\twoheadrightarrow \bar A$ is actually an isomorphism. We see than in our case we can prove this result without using the validity of the conjecture and, hence, proving the validity of the BMR conjecture for $W$.
			
		Let $F$ denotes the field of fractions of $R$ and $\bar F$ an algebraic closure. By proposition \ref{F} we have that $\bar A\otimes_R \bar F \simeq H\otimes_R \bar F$. We have the following commutative diagram:
			$$
			\xymatrix{
				H \ar[r]^{\phi_2} \ar@{->>}[d]^{\psi_2} \ar@{.>}[dr]& H \otimes_R \overline{F}  \ar[d]^{\psi_1}_{\rotatebox{90}{$\simeq$}}\\
				\overline{A} \ar[r]^{\phi_1}& \overline{A} \otimes_R \overline{F}
			}
			$$
			
			We want to prove that $\ker\grc_2=0$. Since  $H$ is torsion-free as $R$-module, we have that $\ker \grf_2=0$. Therefore, it will be sufficient to prove that $\ker\grc_2=\ker \grf_2$. For this purpose, 
			let $h\in \ker\grc_2$. Then, $\grf_1(\grc_2(h))=\grc_1(\grf_2(h))$ and, hence, $\grf_2(h)=(\grc_1^{-1}\circ \grf_1)(\grc_2(h))=(\grc_1^{-1}\circ \grf_1)(0)=0$, which means that $\ker\grc_2\subset \ker\grf_2$. On the other hand, 
			let $h\in \ker\grf_2$. Then, $\grf_1(\grc_2(h))=\grc_1(\grf_2(h))=0$, which means that $\grc_2(h)\in \bar A^{\text{tor}}$. However,  $\bar A^{\text{tor}}=0$, since $\bar A$ is a free $R$-module. Therefore, $\ker\grf_2\subset \ker\grc_2$ and, hence, $\ker\grc_2=\ker\grf_2$.
			\end{proof}

		\subsection{Finding the basis}
		In proposition \ref{torsion} we saw that if $H$ is torsion free, then we only have to prove the validity of the conjecture for the cases of $G_7$, $G_{11}$ and $G_{19}$. Unfortunately, this torsion-free assumption does not appear to be easy to check a priori. In this section we describe another method of proving the BMR freeness conjecture for the first two families, without using this assumption.
		
		In the previous section we saw that $H_{\tilde{\ZZ}}$ is generated as $R_{\tilde \ZZ}^+$-module by the elements $\grC(\tilde{T}_{w_x})$, $x\in \overline{W}$, where $w_x$ is a reduced word that represents $x$ in $\overline{W}$, $\tilde{T}_{w_x}$ the image of $T_{w_x}$ inside $A_+(C)\otimes_{\gru}R^{+}_{\tilde{\ZZ}}$ (or $\tilde{A}_+(C)\otimes_{\gru}R^{+}_{\tilde{\ZZ}}$ for the cases $G_{13}$ and $G_{15}$) and $\grC$ is the ER-surjection associated to $W$ (see propositions \ref{ERSUR} and \ref{ERRSUR} to recall the notations). Motivated by this idea, we will explain in general how we found a spanning set of $H$ over $R$ of $|W|$ elements, when $W$ belongs to the tetrahedral or octahedral family.
		
		For every $x\in \overline{W}$ we fix a reduced word $w_x$ in letters $y_1,y_2$ and $y_3$ that represents $x$ in $\overline{W}$. From the reduced word $w_x$ one can obtain a word $\bar w_x$ that also represents $x$ in $\overline{W}$, defined as follows:
		$$\bar w_x= \begin{cases}
		w_x  &\text{for }x=1\\
		w_x(y_1y_1)^{n_1}(y_2y_2)^{n_2}(y_3y_3)^{n_3}&\text{for }x\not=1
		\end{cases},$$
		where $n_i \in \ZZ_{\geq 0}$ and $(y_iy_i)^{n_i}$ is a shorter notation for the word $\underbrace{(y_iy_i)\dots(y_iy_i)}_{n_i-\text{ times }}$. We notice that if we choose $n_1=n_2=n_3=0$, the word $\bar w_{x}$ coincides with the word $w_{x}$.
		
		Moving some of the pairs   $(y_iy_i)^{n_i}$ somewhere inside $\bar w_x$ and using the braid relations between the generators $y_i$ of the Coxeter group $C$ one can obtain a word $\tilde w_x$, which also represents $x$ in $\bar W$, such that:
		\begin{itemize}[leftmargin=*]
			\item $\ell(\tilde w_x)=\ell(\bar w_x)$, where $\ell(w)$ denotes the length of the word $w$.
			\item Let $m$ be an odd number. Whenever in the word $\tilde w_x$ there is a letter $y_i$ at the $m$th-position from left to right, then in the $(m+1)$th-position there is a letter $y_j$, $j\not=i$.
			\item $\tilde w_x=w_x$ if and only if $\bar w_x=w_x$. In particular, 	$\tilde w_1=w_1$.
		\end{itemize} 
			A word $\tilde w_x$ as described above is called \emph{a base word} associated  to $\bar w_x$.
			
				Let $\tilde w_x$ be a base word. We set $a_{ij}:=y_iy_j$, $i,j\in\{1,2,3\}$ with $i\not=j$. By the definition of $\tilde w_x$ and the fact that $\overline{W}$ is the subgroup of $C$ that contains the elements of even Coxeter length, the word $\tilde w_x$ can be considered as a word in letters $a_{ij}$.  
			\begin{ex}
			Let $W:=G_{15}$, an exceptional group that belongs to the octahedral family. For this family we have that $C$ is the finite Coxeter group of type $B_3$. For the reduced word $w_{x}=y_1y_2y_3y_2y_1y_2$ we choose $\bar w_{x}=w_{x}\mathbf{(y_1y_1)^2}\mathbf{(y_2y_2)(y_3y_3)}$. We now choose a base word $\tilde w_{x}$. We write $w_1\equiv w_2$ if the words $w_1$ and $w_2$ represent the same element inside $\overline W$. We have:
			$$\begin{array}{lcl}
			\bar w_{x}&=&w_{x}\mathbf{(y_1y_1)^2}\mathbf{y_2y_2y_3y_3}\\
			&\equiv&y_1y_2y_3(y_2y_1y_2\mathbf{y_1)y_1}\mathbf{y_1y_1y_2y_2y_3y_3}\\
			&\equiv&y_1y_2y_3y_1y_2y_1y_2y_1\mathbf{y_1y_1y_2y_2y_3y_3}\\
			&\equiv&y_1y_2(y_3y_1)y_2y_1y_2y_1\mathbf{y_1y_1y_2y_2y_3y_3}\\
			&\equiv&y_1y_2y_1y_3y_2y_1y_2y_1\mathbf{y_1y_1y_2y_2y_3y_3}\\
			&\equiv&(y_1y_2y_1\mathbf{y_2)y_2}y_3y_2y_1y_2y_1\mathbf{y_1y_1y_3y_3}\\
			&\equiv&y_2y_1y_2y_1y_2y_3y_2y_1y_2y_1\mathbf{y_1y_1y_3y_3}\\
			&\equiv&y_2(y_1\mathbf{y_3)y_3}y_2y_1y_2\mathbf{y_1y_1}y_3y_2y_1y_2y_1\\
			&\equiv&y_2y_3y_1y_3y_2y_1y_2y_1y_1y_3y_2y_1y_2y_1.
			\end{array}$$
			We choose $\tilde w_{x}=y_2y_3y_1y_3y_2y_1y_2y_1y_1y_3y_2y_1y_2y_1=a_{23}a_{13}a_{21}a_{21}a_{13}a_{21}a_{21}$. 
			\qed 
			\label{exx}\end{ex}
		\begin{rem}
			The choice of $w_x$, the non-negative integers $(n_i)_{\substack{1\leq i\leq 3}}$ and $\tilde w_x$  is a product of experimentation, to provide a simple and robust proof for theorem \ref{main}. We tried more combinations that lead to more complicated and bloated proofs, or others where we couldn't arrive to a conclusion.
		\end{rem}

		We recall that $\bar B$ is generated by the elements 
		$\bar \gra,$ $\bar \grb$ and  $\bar \grg$, where $\gra$, $\grb$ and $\grg$ are generators of $B$ in ER presentation. Inspired by the definition of the ER-surjection (see propositions \ref{ERSUR} and \ref{ERRSUR}), we obtain an element $\bar{b}_{\tilde{w}_x}$ inside $\bar B$, by replacing $a_{13}$, $a_{32}$ and $a_{21}$  with $\bar{\gra}$, $\bar{\grb}$ and $\bar{\grg}$, respectively.
		We use the group isomorphism $\grf_2$ we describe in table \ref{t2} of Appendix \ref{ap}  to write the elements $\gra$, $\grb$ and $\grg$ in the BMR presentation and we set $\grs_{\gra}:=\grf_2(\gra)$, $\grs_{\grb}:=\grf_2(\grb)$ and $\grs_{\grg}:=\grf_2(\grg)$. Therefore, we can also consider the element $\bar{b}_{\tilde{w}_x}$ as being
		a product of $\bar\grs_{\gra} $, $\bar\grs_{\grb} $ and $\bar\grs_{\grg}$. We denote this last element by $\bar v_x$.
		\begin{ex}
			Let $\tilde w_x$  be as in the example \ref{exx} for the case of $G_{15}$. Therefore,  $b_{\tilde{w}_x}=\bar{\grb}^{-1}\bar{\gra}\bar{\grg}^2\bar{\gra}\bar{\grg}^2.$
			We consider now the BMR presentation of the complex braid group $B$ associated to $G_{15}$ and also the isomorphism $\grf_2$. Both can be found in Appendix \ref{ap}, table \ref{t2}.
			We have: $\bar v_x=\bar{u}^{-1}\bar{t}(\bar{t}\bar{u})^{-2}\bar{t}(\bar{t}\bar{u})^{-2}$. Since the element $s(tu)^2$ is central in $B$ (see A.2 in \cite{brouebook}), we have that $(\bar{t}\bar{u})^{-2}=\bar{s}$. Therefore, 
			$\bar v_x=\bar{u}^{-1}\bar{t}\bar{s}\bar{t}\bar{s}.$
			\qed
		\end{ex}
		We now explain  how we arrived to guess a spanning set of $|W|$ elements for the generic Hecke algebra associated to every exceptional group belonging to the first two families. 
		\begin{itemize}[leftmargin=*]
			\item[1.]  Let $W$ be an exceptional group of rank 2, which belongs either to the tetrahedral or octahedral family. For every $x\in \overline{W}$ we choose a specific reduced word $w_x$, specific non-negative integers $(n_i)_{\substack{1\leq i\leq 3}}$ and a specific  base word $\tilde w_x$ associated to the word $\bar w_x$, which is determined by $w_x$ and $(n_i)$.
			
			\item[2.] For every $x\in \overline {W}$, let $x_1^{m_1}x_2^{m_2}\dots x_r^{m_r}$ be the corresponding factorization of $\bar v_x$ into a product of $\bar \grs_{\gra}$, $\bar \grs_{\grb}$ and $\bar \grs_{\grg}$ (meaning that $x_i \in\{\bar \grs_{\gra}, \bar \grs_{\grb},\bar \grs_{\grg}\}$ and $m_i\in \ZZ$). 
			Let $f_0: \bar B\rightarrow B$ be a set theoretic section such that 
			$f_0(x_1^{m_1}x_2^{m_2}\dots x_r^{m_r})=f_0(x_1)^{m_1}f_0(x_2)^{m_2}\dots f_0(x_r)^{m_r}$, $f_0(\bar \grs_{\gra})=\grs_{\gra}$,  $f_0(\bar \grs_{\grb})=\grs_{\grb}$ and $f_0(\bar \grs_{\grg})=\grs_{\grg}$. 
				We set $v_x:=f_0(\bar v_x)$.
			\item[3.] We set $U_W:=\sum\limits_{x\in \overline{W}} \sum\limits_{k=0}^{|Z(W)|-1}Rz^kv_x$, where $R$ is the Laurent polynomial ring over which we define the generic Hecke algebra $H$ associated to $W$. 	We give $U_W$ explicitly for every exceptional group of rank 2 belonging to the first two families.
			\end{itemize}
	\hspace*{0.4cm}\textbf{The tetrahedral family.}

			\begin{itemize}[leftmargin=0.6cm]
		\small{
			\item[$\bullet$]$H_{G_4}:=\langle s,t\;|\; sts=tst,\prod\limits_{i=1}^{3}(s-u_{i})=\prod\limits_{i=1}^{3}(t-u_{i})=0\rangle$.\\
			$U_{G_4}=\sum\limits_{k=0}^1(z^ku_1+z^ku_1t^{-1}u_1)$, where  $z=(st)^3$ and $u_1$  denotes the subalgebra of $H_{G_4}$ generated by $s$.
				\item[$\bullet$]$H_{G_5}:=\langle s,t\;|\; stst=tsts,\prod\limits_{i=1}^{3}(s-u_{s,i})=\prod\limits_{j=1}^{3}(t-u_{t,j})=0\rangle$.\\
				$U_{G_5}=\sum\limits_{k=0}^5(z^ku_1u_2+z^kt^{-1}su_2)$, where  $z=(st)^2$ and $u_1$ and $u_2$ denote the subalgebras of $H_{G_5}$ generated by $s$ and $t$, respectively.
					\item[$\bullet$]$H_{G_6}:=\langle s,t\;|\; ststst=tststs,\prod\limits_{i=1}^{2}(s-u_{s,i})=\prod\limits_{j=1}^{3}(t-u_{t,j})=0\rangle$.\\
					$U_{G_6}=\sum\limits_{k=0}^3(z^ku_2+z^ku_2su_2)$, where  $z=(st)^3$ and $u_2$ denotes the subalgebra of $H_{G_6}$ generated by  $t$.
					\item[$\bullet$]  $H_{G_{7}}=\langle s,t,u\;|\; stu=tus=ust, \;\prod\limits_{i=1}^{2}(s-u_{s,i})=\prod\limits_{j=1}^{3}(t-u_{t,j})=\prod\limits_{l=1}^{3}(u-u_{u,l})=0\rangle$. \\
					$U_{G_7}=\sum\limits_{k=0}^{11}(z^ku_3u_2+z^ktu^{-1}u_2),$ where  $z=stu$ and $u_2$ and $u_3$ denote the subalgebras of $H_{G_7}$ generated by $t$ and $u$, respectively.\\}
			\end{itemize}
			\hspace*{0.4cm}\textbf{The octahedral family.}
			
			\begin{itemize}[leftmargin=0.67cm]
				\small{
				\item[$\bullet$]$H_{G_8}:=\langle s,t\;|\; sts=tst,\prod\limits_{i=1}^{4}(s-u_{i})=\prod\limits_{i=1}^{4}(t-u_{i})=0\rangle$.\\
				$U_{G_8}=\sum\limits_{k=0}^3(z^ku_1+z^ku_1tu_1+z^ku_1t^2)$, where  $z=(st)^3$ and $u_1$  denotes the subalgebra of $H_{G_8}$ generated by $s$.
				\item[$\bullet$] $H_{G_{9}}=\langle s,t\;|\; ststst=tststs, \prod\limits_{i=1}^{2}(s-u_{s,i})=\prod\limits_{j=1}^{4}(t-u_{t,j})=0\rangle$.\\
				 $U_{G_9}=\sum\limits_{k=0}^7(z^ku_2+z^ku_2su_2+z^ku_2st^{-2}s)$, where  $z=(st)^3$ and  $u_2$ denotes the subalgebra of $H_{G_9}$ generated by $t$.
				 \item[$\bullet$] $H_{G_{10}}=\langle s,t\;|\; stst=tsts,\prod\limits_{i=1}^{3}(s-u_{s,i})=\prod\limits_{j=1}^{4}(t-u_{t,j})=0\rangle$.\\
				   $U_{G_{10}}=\sum\limits_{k=0}^{11}(z^ku_2u_1+z^ku_2st^{-1}+z^ku_2s^{-1}t+z^ku_2s^{-1}ts^{-1})$, where $z=(st)^2$  and $u_1$ and $u_2$ denote the subalgebras of $H_{G_{10}}$ generated by $s$ and $t$, respectively.
				   \item[$\bullet$] $H_{G_{11}}=\langle s,t,u\;|\; stu=tus=ust, \;\prod\limits_{i=1}^{2}(s-u_{s,i})=\prod\limits_{j=1}^{3}(t-u_{t,j})=\prod\limits_{l=1}^{4}(u-u_{u,l})=0\rangle$. \\
				   $U_{G_{11}}=\sum\limits_{k=0}^{23}(z^ku_3u_2+z^ku_3tu^{-1}u_2),$ where  $z=stu$ and $u_2$ and $u_3$ denote the subalgebras of $H_{G_{11}}$ generated by $t$ and $u$, respectively.
				\item[$\bullet$] $ H_{G_{12}}=\langle s,t,u\;|\; stus=tust=ustu,\prod\limits_{i=1}^{2}(s-u_{s,i})=\prod\limits_{j=1}^{2}(t-u_{t,j})=\prod\limits_{l=1}^{2}(u-u_{u,l})=0\rangle.$\\
				  $U_{G_{12}}=\sum\limits_{k=0}^1(z^ku_1u_2+z^ku_1uu_2+z^ku_1usu_2+z^ku_1tuu_2+z^ktsu_2+z^ktusu_2+z^ktsuu_2+z^kutsu_2),$ where $z=(stu)^4$ and $u_1$ and $u_2$ denote the subalgebras of $H_{G_{12}}$ generated by $s$ and $t$, respectively.
				  \item[$\bullet$]  $H_{G_{13}}=\langle s,t,u\;|\; ustu=tust,stust=ustus, \prod\limits_{i=1}^{2}(s-u_{s,i})=\prod\limits_{j=1}^{2}(t-u_{t,j})=\prod\limits_{l=1}^{2}(u-u_{u_l})=0\rangle$.\\
				$U_{G_{13}}=\sum\limits_{k=0}^{3}(z^ku_2+z^ku_3u_2su_2+z^ku_2u_1uu_2+z^ktusu_2+z^kstsu_2+z^kstuu_2)$,  where  $z=(stu)^3$ and $u_1$, $u_2$ and $u_3$  denote the subalgebras of $H_{G_{13}}$ generated by $s$, $t$ and $u$, respectively.
				\item[$\bullet$] $H_{G_{14}}=\langle s,t\;|\; stststst=tstststs,\prod\limits_{i=1}^{2}(s-u_{s,i})=\prod\limits_{j=1}^{3}(t-u_{t,j})=0\rangle$.\\  $U_{G_{14}}=\sum\limits_{k=0}^{5}(z^ku_1u_2+z^ku_1tsu_2+z^ku_1t^{-1}su_2+z^ku_1tst^{-1}su_2)$, where $z=(st)^4$ and $u_1$ and $u_2$ denote the subalgebras of $H_{G_{14}}$ generated by $s$ and $t$, respectively.
				\item[$\bullet$] $H_{G_{15}}=\langle s,t,u\;|\; stu=tus,ustut=stutu,\prod\limits_{i=1}^{2}(s-u_{s,i})=\prod\limits_{j=1}^{2}(t-u_{t,j})=\prod\limits_{l=1}^{3}(u-u_{u,l})=0\rangle$.\\
				$U_{G_{15}}=\sum\limits_{k=0}^{11}z^k(u_3+u_3s+u_3t+u_3ts+u_3st+u_3tst+u_3sts+u_3tsts)$, where $z=stutu$ and $u_3$ denotes the subalgebra of $H_{G_{15}}$ generated by $u$.\\}
			\end{itemize}
	
		Let $W$ be one of the exceptional groups belonging to the first two families and let $H_W$ denotes the associated generic Hecke algebra. 
		The main result of this paper is the following theorem. Notice that the second part of it follows directly from proposition \ref{BMR PROP}.
		\begin{thm}
			$H_W=U_W$ and, therefore, the BMR freeness conjecture holds for all the groups belonging to the tetrahedral and octahedral family. 
			\label{main}
		\end{thm}
			By the construction of the base words and by the definition of $U_W$ we have that $1_{H_W}\in U_W$. Therefore, in order to prove the first part of the above theorem, it is enough to prove that $U_W$ is a left (or right)-sided ideal of $H_W$. This result has been proven in Appendix \ref{sbmr} by using a case-by-case analysis. In this proof we use a lot of
			calculations, that are fully-detailed and they do not leave anything to the reader. We also tried to make them as less complicated and short as possible, in order to be fairly
			easy to follow. The next corollary can also been found in \cite{marinpfeiffer} (see Corollary 1.3 there).
		\begin{cor}
				The BMR freeness conjecture holds for all the 2-exceptional groups.
			\end{cor}
			\begin{proof}
				By theorem \ref{2case} we have the validity of the conjecture for all the 2-exceptional groups, apart from $G_{13}$. Therefore, the result follows directly from \ref{main}.
			\end{proof}
		\begin{rem}
			The cases of $G_4$, $G_8$ and $G_{12}$ has already been proven (see theorems \ref{braid} and \ref{2case}). However, using this approach, we managed to give an alternative, computer-free proof for $G_{12}$ and also a new basis for the groups $G_4$ and $G_8$. For the latter, we also managed to give a basis consisting of braid group elements with positive powers (compare with theorem 3.2 in \cite{chavli}).

		\end{rem}
		\newpage
		\appendix
		\label{ap}
		\section{The BMR and ER presentations}

				\begin{table}[h]
			
					\begin{center}
						\small
						\caption{
							\bf{BMR and ER presentations for the exceptional groups of rank 2}}
						\label{t3}
						
						\scalebox{0.79}
						{\begin{tabular}{|c|p{5.1cm}|p{5.6cm}|p{2.7cm}|p{2.7cm}|}
								\hline
								\thead{\\\textbf{Group}\\\\} & \thead{\textbf{BMR presentation}} & \thead{\textbf{ER presentation}}&\thead{\textbf{ $\mathbf{\grf_1}$: BMR $\mathbf{\leadsto}$ ER}}&\thead{\textbf{$\mathbf{\grf_2}$: ER $\mathbf{\leadsto}$ BMR }}\\
								\hline
								\thead{$G_4$\\ \\
									$G_8$\\ \\
									$G_{16}$}&
								\thead{$\langle s,t\;|\;s^3=t^3=1, sts=tst\rangle$\\\\
							$\langle s,t\;|\;s^4=t^4=1, sts=tst\rangle$\\ \\
								$\langle s,t\;|\;s^5=t^5=1, sts=tst\rangle
							$}& 
								\thead{$\langle a,b,c\;|\; a^2=b^{-3}=\text{central}, c^3=1, abc=1\rangle$\\ \\
								$\langle a,b,c\;|\; a^2=b^{-3}=\text{central}, c^4=1, abc=1\rangle$\\ \\
								$\langle a,b,c\;|\; a^2=b^{-3}=\text{central}, c^5=1, abc=1\rangle
							$}&
								\thead{$\begin{array}[t]{lcl}
								s&\mapsto &c\\
								t&\mapsto& c^{-1}b 
								\end{array}$} &
								\thead{$\begin{array}[t]{lcl}\\
								a&\mapsto& (sts)^{-1}\\b&\mapsto& st\\c&\mapsto&s
								\end{array}$}\\
								
								\hline \hline
								\thead{$G_5$\\ \\
								$G_{10}$\\ \\
								$G_{18}$}&
								\thead{
								$\langle s,t\;|\; s^3=t^3=1,stst=tsts\rangle$\\ \\
								$\langle s,t\;|\; s^3=t^4=1,stst=tsts\rangle$\\ \\
								$\langle s,t\;|\; s^3=t^5=1,stst=tsts\rangle$}& \thead{
								$\langle a,b,c\;|\; a^2=\text{central},b^3=c^3=1, abc=1\rangle$ \\ \\
								$\langle a,b,c\;|\; a^2=\text{central},b^3=c^4=1, abc=1\rangle$\\ \\
								$\langle a,b,c\;|\; a^2=\text{central},b^3=c^5=1, abc=1\rangle$}&
							\thead{	$\begin{array}[t]{lcl} 
								s&\mapsto &
								b\\t&\mapsto& c 
								\end{array}$} &
								\thead{$\begin{array}[t]{lcl} 
								a&\mapsto& (st)^{-1}\\b&\mapsto& s\\c&\mapsto&t
								\end{array}$}\\
								
								\hline \hline
								
							\thead{	$G_6$\\ \\
								$G_9$\\ \\
								$G_{17}$}
								&
								\thead{$\begin{array}[t]{lcl} \langle s,t\;|\; s^2=t^3=1,(st)^3=(ts)^3\rangle\\ \\
								\langle s,t\;|\; s^2=t^4=1,(st)^3=(ts)^3\rangle\\ \\
								\langle s,t\;|\; s^2=t^5=1,(st)^3=(ts)^3\rangle\end{array}$}& 
								\thead{$\begin{array}[t]{lcl} \langle a,b,c\;|\; a^2=c^3=1, b^3=\text{central},abc=1\rangle\\\\
								\langle a,b,c\;|\; a^2=c^4=1, b^3=\text{central},abc=1\rangle\\ \\
								\langle a,b,c\;|\; a^2=c^5=1, b^3=\text{central},abc=1\rangle
								\end{array}
								$}&
								\thead{$\begin{array}[t]{lcl}
								s&\mapsto &
								a\\t&\mapsto& c 
								\end{array}$} &
								\thead{$\begin{array}[t]{lcl}\\
								a&\mapsto& s\\b&\mapsto& (ts)^{-1}\\c&\mapsto&t
								\end{array}$}\\
								
								\hline \hline
								\thead{$G_7$\\ \\\\
								$G_{11}$\\ \\\\
								$G_{19}$}&
								\thead{$\left\langle
								\begin{array}{l|cl}
								& s^2=t^3=u^3=1&\\
								s,t,u\\
								&stu=tus=ust&
								\end{array}\right \rangle$\\\\
								$\left\langle
								\begin{array}{l|cl}
								& s^2=t^3=u^4=1&\\s,t,u\\
								&stu=tus=ust&
								\end{array}\right \rangle$\\\\
								$\left\langle
								\begin{array}{l|cl}
								& s^2=t^3=u^5=1&\\s,t,u\\
								&stu=tus=ust&
								\end{array}\right \rangle$}
							& 
							\thead{$	\langle a,b,c\;|\;a^2=b^3=c^3=1, abc=\text{central }\rangle$\\\\\\
								$\langle a,b,c\;|\;a^2=b^3=c^4=1, abc=\text{central }\rangle$\\\\\\
								$\langle a,b,c\;|\;a^2=b^3=c^5=1, abc=\text{central }\rangle$}&
								\thead{$\begin{array}[t]{lcl}
								s&\mapsto &
								a\\t&\mapsto& b\\
								u&\mapsto& c
								\end{array}$} &
								\thead{$\begin{array}[t]{lcl}
								a&\mapsto& s\\b&\mapsto& t\\c&\mapsto&u
								\end{array}$}\\

								\hline \hline

								\thead{$G_{12}$}
								&
								\thead{$\left\langle
								\begin{array}{l|cl}
								& s^2=t^2=u^2=1&\\s,t,u\\
								&stus=tust=ustu&
								\end{array}\right \rangle
								$}& 
								\thead{$\left	\langle \begin{array}{l|cl} &a^2=1&\\ a,b,c &b^3=c^{-4}=\text{central}&\\ &abc=1&
								\end{array}\right\rangle $}&
								\thead{$\begin{array}[t]{lcl}
								s&\mapsto &
								a\\t&\mapsto& a^{-1}c^2b\\
								u&\mapsto&(cb)^{-1} 
								\end{array}$ }&
								\thead{$\begin{array}[t]{lcl}
								a&\mapsto& s\\b&\mapsto& (stus)^{-1}\\c&\mapsto&stu
								\end{array}$} \\
								
								\hline\hline
								\thead{$G_{13}$}&
								
								\thead{$
								\left\langle
								\begin{array}{l|cl}
								&s^2=t^2=u^2=1 &\\s,t,u&stust=ustus&\\
								&tust=ustu&
								\end{array}\right \rangle$}& 
								
								\thead{$ \left	\langle \begin{array}{l|cl} &a^2=d^2=1&\\a,b,c,d &b^3=dc^{-2}=\text{central}&\\ &abc=1&
								\end{array}\right\rangle$}&

								\thead{$\begin{array}[t]{lcl}
								s&\mapsto &
								d\\t&\mapsto& a\\
								u&\mapsto&b(da)^{-1} 
								\end{array}$} &
								\thead{$\begin{array}[t]{lcl}
								a&\mapsto& t\\b&\mapsto& ust\\c&\mapsto&(tust)^{-1}\\
								d&\mapsto&s\\ 
								\end{array}$}\\
								
								\hline\hline
								
								\thead{$
								G_{14}$}
							&
								\thead{$\langle s,t\;|\; s^2=t^3=1,(st)^4=(ts)^4\rangle$}& 
								\thead{$\langle a,b,c\;|\;a^2=b^3=1, c^4=\text{central},abc=1\rangle$}&
								\thead{$\begin{array}[t]{lcl}
								s&\mapsto &
								a\\t&\mapsto& b
								\end{array}$} &
								\thead{$\begin{array}[t]{lcl}
								a&\mapsto& s\\b&\mapsto& t\\c&\mapsto&(st)^{-1}
								\end{array}$}\\
								
								\hline\hline
								\thead{$
								G_{15}$}&
								\thead{$
								\left\langle
								\begin{array}{l|cl}
								& s^2=t^2=u^3=1&\\s,t,u&ustut=stutu&\\
								&tus=stu&
								\end{array}\right \rangle$}& 
								
								\thead{$\left	\langle \begin{array}{l|cl} &a^2=b^3=d^2=1&\\ a,b,c,d &dc^{-2}=\text{central}&\\ &abc=1&
								\end{array}\right\rangle$}&
								
								\thead{$\begin{array}[t]{lcl}
								s&\mapsto &
								d\\t&\mapsto& a\\
								u&\mapsto&b
								\end{array}$} &
								\thead{$\begin{array}[t]{lcl}
								a&\mapsto& t\\b&\mapsto& u\\c&\mapsto&(tu)^{-1}\\
								d&\mapsto&s
								\end{array}$}\\
								
								\hline \hline
								\thead{
								$G_{20}$}
								&
								\thead{$\langle s,t\;|\;s^3=t^3=1, ststs=tstst\rangle$}& 
								\thead{$\left\langle \begin{array}{l|cl} &a^2=\text{central}&\\
								a,b,c&b^3=1,a^4=c^{-5}&\\ &abc=1&
								\end{array}\right\rangle $}&
									$\begin{array}{lcl}
								s&\mapsto &
								b\\
								t&\mapsto& (ac)^{-1}
								\end{array}$ &
								$\begin{array}{lcl}
								a&\mapsto& (ststs)^{-1}\\
								b&\mapsto& s\\c&\mapsto&tsts
								\end{array}$\\
								\hline\hline
							\thead{$
								G_{21}
								$}&
								\thead{$\langle s,t\;|\; s^2=t^3=1,(st)^5=(ts)^5\rangle$}
								& 
								
								\thead{$	\langle a,b,c\;|\;a^2=b^3=1, c^5=\text{central},abc=1\rangle$}&
								
								\thead{$\begin{array}[t]{lcl}
								s&\mapsto &
								a\\t&\mapsto& b
								\end{array}$ }&
								\thead{$\begin{array}[t]{lcl}
								a&\mapsto& s\\b&\mapsto& t\\c&\mapsto&(st)^{-1}
								\end{array}$}\\
								\hline \hline
								\thead{$G_{22}$}&
								\thead{$\left\langle \begin{array}{l|cl}&s^2=t^2=u^2=1&\\
									s,t,u\\ &stust=tustu=ustus&
								\end{array}
								\right\rangle$}
								& 
								\thead{$\left\langle 
								\begin{array}{l|cl}
								&b^3=\text{central}&\\a,b,c&a^2=1,c^5=b^{-6}&
								\\&abc=1&\end{array}\right\rangle$}&
								$\begin{array}{lcl}
								s&\mapsto &
								a\\t&\mapsto& (b^{-1}cb^2)^{-1}\\
								u&\mapsto&(cb)^{-1}
								\end{array}$&
							$\begin{array}{lcl}
								a&\mapsto&s\\b&\mapsto& tustu\\c&\mapsto&(stustu)^{-1}
								\end{array}$\\
								\hline
								\hline
							\end{tabular} }
						\end{center}
					\end{table}
					
					\newpage
						\begin{table}[h]
							\begin{center}
								\small
								\caption{
									\bf{BMR and ER presentations for the complex braid groups associated to the exceptional groups of rank 2}}
								\label{t2}
								
								\scalebox{0.79}
								{\begin{tabular}{|c|c|c|c|c|}
										\hline
										\thead{\\\textbf{Group}\\\\} & \thead{\textbf{BMR presentation}} & \thead{\textbf{ER presentation}}&\thead{\textbf{ $\mathbf{\grf_1}$: BMR $\mathbf{\leadsto}$ ER}}&\thead{\textbf{$\mathbf{\grf_2}$: ER $\mathbf{\leadsto}$ BMR }}\\
										\hline
										\thead{$G_4$\\ \\
											$G_8$\\ \\
											$G_{16}$}
										&\thead{$\langle s,t\;|\;sts=tst\rangle$}
										& 
										\thead{$\langle \gra,\grb,\grg\;|\; \gra^2=\grb^{-3}=\text{central},\gra\grb\grg=1\rangle$}&
										\thead{$\begin{array}[t]{lcl} 
											s&\mapsto &\grg\\t&\mapsto& \grg^{-1}\grb
											\end{array}$} &
										\thead{$\begin{array}[t]{lcl}
											\gra&\mapsto& (sts)^{-1}\\\grb&\mapsto& st\\\grg&\mapsto&s
											\end{array}$}\\
										
										\hline \hline
										\thead{$G_5$\\ \\
											$G_{10}$\\ \\
											$G_{18}$}
										
										&\thead{$\langle s,t\;|\; stst=tsts\rangle$}& 	\thead{$\langle \gra,\grb,\grg\;|\; \gra^2=\text{central},\gra\grb\grg=1\rangle$}&
										
										\thead{$\begin{array}[t]{lcl} 
											s&\mapsto &
											\grb\\t&\mapsto& \grg
											\end{array}$} &
										\thead{$\begin{array}[t]{lcl} 
											\gra&\mapsto& (st)^{-1}\\\grb&\mapsto& s\\\grg&\mapsto&t
											\end{array}$}\\
										
										\hline\hline
										
										\thead{$G_6$\\ \\
											$G_9$\\ \\
											$G_{17}$}&
										\thead{$\langle s,t\;|\; (st)^3=(ts)^3\rangle$}& 
										\thead{$\langle \gra,\grb,\grg\;|\; \grb^3=\text{central},\gra\grb\grg=1\rangle $}&
										\thead{$\begin{array}[t]{lcl}
											s&\mapsto &
											\gra\\t&\mapsto& \grg 
											\end{array}$} &
										\thead{$\begin{array}[t]{lcl}
											\gra&\mapsto& s\\\grb&\mapsto& (ts)^{-1}\\\grg&\mapsto&t
											\end{array}$}\\
										
										\hline\hline
										\thead{$G_7$\\ \\
											$G_{11}$\\ \\
											$G_{19}$}
										&
										\thead{$\langle s,t,u\;|\; stu=tus=ust\rangle$}&
										\thead{$\langle \gra,\grb,\grg\;|\;\gra\grb\grg=\text{central }\rangle $}&
										\thead{$\begin{array}[t]{lcl}
											s&\mapsto &
											\gra\\t&\mapsto& \grb\\
											u&\mapsto& \grg
											\end{array}$}&
										\thead{$\begin{array}[t]{lcl}
											\gra&\mapsto& s\\\grb&\mapsto& t\\\grg&\mapsto&u
											\end{array}$}\\

										\hline\hline

										\thead{
											$G_{12}$}&
										\thead{$\langle s,t,u\;|\; stus=tust=ustu\rangle$}
										& \thead{$\langle \gra,\grb,\grg\;|\; \grb^3=\grg^{-4}=\text{central},\gra\grb\grg=1\rangle$} &
										\thead{$\begin{array}[t]{lcl}
											s&\mapsto &
											\gra\\t&\mapsto& \gra^{-1}\grg^2\grb\\
											u&\mapsto&(\grg\grb)^{-1} 
											\end{array}$} &
										\thead{$\begin{array}[t]{lcl}
											\gra&\mapsto& s\\\grb&\mapsto& (stus)^{-1}\\\grg&\mapsto&stu
											\end{array}$}\\ 
										
										\hline\hline
										\thead{$G_{13}$}
										&
										\thead{
											$\left\langle
											\begin{array}{l|cl}
											& stust=ustus&\\s,t,u\\
											&tust=ustu&
											\end{array}\right \rangle
											$}&

										\thead{$
											\left\langle
											\begin{array}{l|cl}
											& \grb^3=\grd\grg^{-2}=\text{central}&\\
											\gra,\grb,\grg,\grd&\grg^4=\text{central}\\
											&\gra\grb\grg=1&
											\end{array}\right \rangle$}
										&

										\thead{$\begin{array}[t]{lcl}s&\mapsto &
											\grd\\t&\mapsto& \gra\\
											u&\mapsto&\grb(\grd\gra)^{-1} 
											\end{array}$} &
										\thead{$\begin{array}[t]{lcl}
											\gra&\mapsto& t\\\grb&\mapsto& ust\\\grg&\mapsto&(tust)^{-1}\\
											\grd&\mapsto&s
											\end{array}$}\\
										
										\hline\hline
										
										\thead{$G_{14}$}&
										\thead{$\langle s,t\;|\; (st)^4=(ts)^4\rangle$}& 
										\thead{$\langle \gra,\grb,\grg\;|\; \grg^4=\text{central},\gra\grb\grg=1\rangle $}&
										\thead{$\begin{array}[t]{lcl}
											s&\mapsto &
											\gra\\t&\mapsto& \grb
											\end{array}$} &
										\thead{$\begin{array}[t]{lcl}
											\gra&\mapsto& s\\\grb&\mapsto& t\\\grg&\mapsto&(st)^{-1}
											\end{array}$}\\
										
										\hline\hline
										\thead{$
											G_{15}$}
										&
										\thead{$
											\left\langle
											\begin{array}{l|cl}
											& ustut=stutu&\\s,t,u\\
											&tus=stu&
											\end{array}\right \rangle$}&
										
										\thead{$
											\left\langle
											\begin{array}{l|cl}
											& \grd\grg^{-2}=\text{central}&\\
											\gra,\grb,\grg,\grd&\grg^4=\text{central}\\
											&\gra\grb\grg=1&
											\end{array}\right \rangle$}&

										\thead{$\begin{array}[t]{lcl}
											s&\mapsto &
											\grd\\t&\mapsto& \gra\\
											u&\mapsto&\grb
											\end{array}$} &
										\thead{$\begin{array}[t]{lcl}
											\gra&\mapsto& t\\\grb&\mapsto& u\\\grg&\mapsto&(tu)^{-1}\\
											\grd&\mapsto&s
											\end{array}$}\\
										
										\hline\hline
										\thead{$
											G_{20}$}
										&
										\thead{$\langle s,t\;|\; ststs=tstst\rangle$}& 
										\thead{$\begin{array}[t]{lcl}\langle \gra,\grb,\grg\;|\; \gra^2=\text{central},\gra^4=\grg^{-5},\gra\grb\grg=1\rangle \end{array}$}&
										\thead{$\begin{array}[t]{lcl}
											s&\mapsto &
											\grb\\t&\mapsto& (\gra\grg)^{-1}
											\end{array}$ }&
										\thead{$\begin{array}[t]{lcl}
											\gra&\mapsto& (ststs)^{-1}\\\grb&\mapsto& s\\\grg&\mapsto&tsts
											\end{array}$}\\
										\hline\hline
										\thead{$G_{21}$}&
										\thead{$\langle s,t\;|\; (st)^5=(ts)^5\rangle
											$}& 
										\thead{$\langle \gra,\grb,\grg\;|\; \grg^5=\text{central},\gra\grb\grg=1\rangle $}&
										\thead{$\begin{array}[t]{lcl}
											s&\mapsto &
											\gra\\t&\mapsto& \grb
											\end{array}$} &
										\thead{$\begin{array}[t]{lcl}
											\gra&\mapsto& s\\\grb&\mapsto& t\\\grg&\mapsto&(st)^{-1}
											\end{array}$}\\
										\hline\hline
										\thead{$
											G_{22}$}&
										\thead{$\langle s,t,u\;|\; stust=tustu=ustus\rangle
											$}& 
										\thead{$\langle \gra,\grb,\grg\;|\; \grb^3=\text{central},\grg^5=\text{central},\gra\grb\grg=1\rangle $}&
										\thead{$\begin{array}[t]{lcl}
											s&\mapsto &
											\gra\\t&\mapsto& (\grb^{-1}\grg\grb^2)^{-1}\\
											u&\mapsto&(\grg\grb)^{-1}
											\end{array}$} &
										\thead{$\begin{array}[t]{lcl}
											\gra&\mapsto&s\\\grb&\mapsto& tustu\\\grg&\mapsto&(stustu)^{-1}
											\end{array}$}\\
										\hline
									\end{tabular} }
								\end{center}
							\end{table}

			\section{The proof of the BMR freeness conjecture for the first two families}
			\label{sbmr}

			\subsection{The Tetrahedral family}
			In this family we encounter the exceptional groups $G_4$, $G_5$, $G_6$ and $G_7$. We know that the BMR freeness conjecture holds for $G_4$ (see theorem \ref{braid}). We prove the conjecture for the rest of the  groups belonging in this family, using a case-by-case analysis. We also prove in a different way the validity of the BMR freeness conjecture for the exceptional group $G_4$. Let $P_s(X)$ denote the polynomials defining $H$ over $R$. If we  expand the relations $P_s(\grs)=0$, where $\grs$ is a distinguished braided reflection associated to $s$, we obtain equivalent relations of the form \begin{equation}\grs^n=a_{n-1}\grs^{n-1}+...+a_1\grs+a_0,\label{ones} \end{equation}
			where $n$ is the order of $s$, $a_i\in R$, for every $i\in\{1, \dots n-1\}$ and $a_0 \in R^{\times}$.
			We multiply $(\ref{ones})$ by $\grs_i^{-n}$ and since $a_0$ is invertible in $R$ we have:
			\begin{equation}\grs_i^{-n}=-a_0^{-1}a_1\grs^{-n+1}-a_0^{-1}a_2\grs^{-n+2}-...-a_0^{-1}a_{n-1}\grs^{-1}+a_0^{-1} \label{twos}\end{equation}
			We multiply ($\ref{ones}$) with a suitable power of $\grs$. Then, for every $m\in \NN$ we have :\begin{equation}\grs^{m}\in R\grs^{m-1}+\dots+R\grs^{m-(n-1)}+R^{\times}\grs^{m-n}\label{ooo}\end{equation}
			Similarly, we multiply ($\ref{twos}$) with a suitable power of $\grs$. Then, for every $m\in \NN$, we have: \begin{equation}\grs^{-m}\in  R\grs^{-m+1}+\dots+R\grs^{-m+(n-1)}+R^{\times}\grs^{-m+n}.\label{oooo}\end{equation}
			For the rest of this section we use directly (\ref{ooo}) and (\ref{oooo}). 
			\subsubsection{\textbf{The case of} $\mathbf{G_4}$}
			Let $R=\ZZ[u_{s,i}^{\pm},u_{t,j}^{\pm}]_{\substack{1\leq i,j\leq 3 }}$ and $ H_{G_{4}}=\langle s,t\;|\; sts=tst,\prod\limits_{i=1}^{3}(s-u_{s,i})=\prod\limits_{j=1}^{3}(t-u_{t,j})=0\rangle.$
			We set $\bar R:=\ZZ[u_{s,i}^{\pm},]_{\substack{1\leq i\leq 3 }}$. Under the specialization $\grf: R\twoheadrightarrow \bar R$, defined by $u_{t,j}\mapsto u_{s,i}$, the algebra $\bar H_{G_{4}}:=H_{G_{4}}\otimes_{\grf}\bar R$ is the generic Hecke algebra associated to $G_{4}$. Let $u_1$ be the subalgebra of $H_{G_{4}}$ generated by $s$.
			We know that $z:=(st)^3=(ts)^3$  generates the center of the associated complex braid group (which, in our case, is the usual braid group on 3 strands). We also know that  $|Z(G_4)|=2$. We set $U=\sum\limits_{k=0}^1(z^ku_1+z^ku_1t^{-1}u_1)$.
			\begin{thm}
				$H_{G_4}=U$.
				\label{thm4}
			\end{thm}
			\begin{proof}
				Since $1\in U$, it is enough to prove that $U$ is a left sided-ideal of $H_{G_4}$. For this purpose, it will be sufficient to prove that $tU\subset U$, since $sU\subset U$ by the definition of $U$. We have:
				$tU\subset \sum\limits_{k=0}^1(z^ktu_1+z^k\mathbf{tu_1t^{-1}}u_1)$.
				By lemma 2.1 in \cite{chavli} we have that $\mathbf{tu_1t^{-1}}=s^{-1}u_2s^{-1}$, where $u_2$ is the subalgebra of $H_{G_{4}}$ generated by $t$. Therefore, $tU\subset \sum\limits_{k=0}^1(z^ktu_1+z^ku_1u_2u_1)$. We expand now $u_2$ as $R+Rt^{-1}+Rt$ and we have that $tU\subset \sum\limits_{k=0}^1(z^ktu_1+z^ku_1+z^ku_1t^{-1}u_1+z^ku_1tu_1)$. Hence, by the definition of $U$ we have that $tU\subset U+\sum\limits_{k=0}^1z^ku_1tu_1$. As a result, it will be sufficient to prove that, for every $k\in\{0,1\}$, $z^kt\in U$. For $k=0$ the result is obvious, since $t=(ts)^3\big(s(tst)s\big)^{-1}=z(s^2ts^2)^{-1}\in zu_1t^{-1}u_1\subset U$. For $k=1$ we have: $zt\in z(R+Rt^{-1}+Rt^{-2})\subset zu_1+zu_1t^{-1}u_1+Rzt^{-2}$. Hence, by the definition of $U$, it is enough to prove that $zt^{-2}\in U$. Indeed, $zt^{-2}=t^{-2}(ts)^3=t^{-1}(sts)ts\in u_1t^2u_1$. We expand $t^2$ as a linear combination of 1, $t^{-1}$ and $t$ and we have that $zt^{-2}\in u_1+u_1t^{-1}u_1+u_1tu_1$. The result follows from the definition of $U$ and the fact that $t\in U$ (case where $k=0$).
			\end{proof}
			\begin{cor}
				The BMR freeness conjecture holds for the generic Hecke algebra $\bar H_{G_{4}}$.
			\end{cor}
			\begin{proof}
				By theorem \ref{thm4} we have that  $H_{G_{4}}$ is generated as right $u_1$-module by 8 elements and, hence, as $R$-module by $|G_{4}|=24$ elements (recall that $u_1$ is generated as $R$-module by 3 elements). Therefore, $\bar H_{G_{4}}$ is generated as $\bar R$-module by $|G_{4}|=24$ elements, since the action of $\bar R$ factors through $R$. The result follows from proposition \ref{BMR PROP}.
			\end{proof}
			\subsubsection{\textbf{The case of $\mathbf{G_5}$}}
			
			Let $R=\ZZ[u_{s,i}^{\pm},u_{t,j}^{\pm}]_{\substack{1\leq i,j\leq 3}}$ and let $H_{G_5}:=\langle s,t\;|\; stst=tsts,\prod\limits_{i=1}^{3}(s-u_{s,i})=\prod\limits_{j=1}^{3}(t-u_{t,i})=0\rangle$ be the generic Hecke algebra associated to $G_5$. Let $u_1$ be the subalgebra of $H_{G_5}$ generated by $s$ and $u_2$ the subalgebra of $H_{G_5}$ generated by $t$. 
			We know that $z:=(st)^2=(ts)^2$  generates the center of the associated complex braid group and that $|Z(G_5)|=6$. We set  $U=\sum\limits_{k=0}^5(z^ku_1u_2+z^kt^{-1}su_2)$. By the definition of $U$ we have the following remark:
			\begin{rem} $Uu_2\subset U$.
				\label{r55}
			\end{rem}
			To make it easier for the reader to follow the calculations, we will underline the elements that  belong to $U$ by definition. Moreover, we will use directly remark \ref{r55}; this means that every time we have a power of $t$ at the end of an element we may ignore it. To remind that to the reader, we put a parenthesis around the part of the element we consider.
			
			Our goal is to prove that $H_{G_5}=U$ (theorem \ref{thm55}). In order to do so, we  first need to prove some preliminary results.
			\begin{lem}
				\mbox{}
				\vspace*{-\parsep}
				\vspace*{-\baselineskip}\\
				\begin{itemize}[leftmargin=0.6cm]
					\item [(i)] For every $k\in\{1,\dots, 4\}$, $z^ktu_1\subset U$.
					\item[(ii)]For every $k\in\{1,\dots,5\}$, $z^kt^{-1}u_1\subset U$.
					\item[(iii)]For every $k\in\{1,\dots,4\}$, $z^ku_2u_1\subset U$.
				\end{itemize}
				\label{ts55}
			\end{lem}
			\begin{proof}\mbox{}
				\vspace*{-\parsep}
				\vspace*{-\baselineskip}\\
				\begin{itemize}[leftmargin=0.6cm]
					\item[(i)]$z^ktu_1=z^kt(R+Rs+Rs^{-1})\subset \underline{z^ku_2}+Rz^k(ts)^2s^{-1}t^{-1}+Rz^kts^{-1}\subset U+\underline{z^{k+1}u_1u_2}+Rz^kts^{-1}$. It remains to prove that $z^kts^{-1}\in U$. Indeed, we have $z^kts^{-1}\in z^k(R+Rt^{-1}+Rt^{-2})s^{-1}\subset\underline{z^ku_1}+Rz^k(st)^{-2}st+Rz^kt^{-1}(st)^{-2}st\subset U+\underline{z^{k-1}u_1u_2}+\underline{z^{k-1}t^{-1}su_2}$.
					\item[(ii)]$z^kt^{-1}u_1=z^kt^{-1}(R+Rs+Rs^{-1})\subset \underline{z^ku_2}+\underline{z^kt^{-1}su_2}+Rz^k(st)^{-2}st\subset U+\underline{z^{k-1}u_1u_2}.$
					\item[(iii)]By definition, $u_2=R+Rt+Rt^{-1}$. The result follows from (i) and (ii).
					\qedhere		\end{itemize}
			\end{proof}
			
			From now on, we will double-underline the elements described in lemma \ref{ts55} and we will use directly the fact that these elements are inside $U$.
			The following proposition leads us to the main theorem of this section.
			\begin{prop}
				$u_1U\subset U$.
				\label{su55}
			\end{prop}
			\begin{proof}
				Since $u_1=R+Rs+Rs^2$, it is enough to prove that $sU\subset U$. By the definition of $U$ and by remark \ref{r55}, we can restrict ourselves to proving that  $z^kst^{-1}s\in U$, for every $k\in\{0,\dots,5\}$. We distinguish the following cases:
				\begin{itemize}[leftmargin=*]
					\item \underline{$k\in\{0,\dots,3\}$}:
					
					$\hspace*{-0.6cm}\small{\begin{array}[t]{lcl}
						z^kst^{-1}s&\in&z^ks(R+Rt+Rt^2)s\\
						&\in&\underline{z^ku_1}+Rz^k(st)^2t^{-1}+Rz^k(st)^2t^{-1}s^{-1}ts\\
						&\in&U+\underline{z^{k+1}u_2}+Rz^{k+1}t^{-1}(R+Rs+Rs^2)ts\\
						&\in&U+\underline{z^{k+1}u_1}+Rz^{k+1}t^{-1}(st)^2t^{-1}+Rz^{k+1}t^{-1}s(st)^2t^{-1}\\
						&\in&U+\underline{z^{k+2}u_2}+\underline{z^{k+2}t^{-1}su_2}.
						\end{array}}$
					\item \underline{$k\in\{4,5\}$}:
					
					$\hspace*{-0.6cm}\small{\begin{array}[t]{lcl}z^kst^{-1}s&\in&z^k(R+Rs^{-1}+Rs^{-2})t^{-1}(R+Rs^{-1}+Rs^{-2})\\
						&\in&\underline{\underline{z^kt^{-1}u_1}}+\underline{z^ku_1u_2}+Rz^ks^{-1}t^{-1}s^{-1}+Rz^ks^{-1}t^{-1}s^{-2}+
						Rz^ks^{-2}t^{-1}s^{-1}+Rz^ks^{-2}t^{-1}s^{-2}\\
						&\in&U+Rz^k(ts)^{-2}t+Rz^k(ts)^{-2}ts^{-1}+Rz^ks^{-1}(ts)^{-2}t+
						Rz^ks^{-1}(ts)^{-2}ts^{-1}\\
						&\in&U+\underline{z^{k-1}u_2}+\underline{z^{k-1}ts^{-1}u_2}+\underline{z^{k-1}u_1u_2}+Rz^{k-1}s^{-1}(R+Rt^{-1}+Rt^{-2})s^{-1}\\
						&\in&U+\underline{z^{k-1}u_1}+Rz^{k-1}(ts)^{-2}t+Rz^{k-1}(ts)^{-2}ts(st)^{-2}st\\
						&\in&U+\underline{z^{k-2}u_2}+\underline{\underline{(z^{k-3}tu_1)t}}. \phantom{=================================}
						\qedhere
						\end{array}}$
				\end{itemize}
				
			\end{proof}
			We can now prove the main theorem of this section.
			
			\begin{thm} $H_{G_5}=U$.
				\label{thm55}
			\end{thm}
			\begin{proof}
				Since $1\in U$, it is enough to prove that $U$ is a left-sided ideal of $H_{G_5}$. For this purpose, one may check  that $sU$ and $tU$ are subsets of $U$. However, by proposition \ref{su55} we restrict ourselves to proving that $tU\subset U$. By the definition of $U$ we have that $tU\subset \sum\limits_{k=0}^5(z^ktu_1u_2+\underline{z^ku_1u_2})\subset U+\sum\limits_{k=0}^5z^ktu_1u_2.$ Therefore, by remark \ref{r55} it will be sufficient to prove  that, for every $k\in\{0,\dots,5\}$, $z^ktu_1\subset U$. However, this holds for every $k\in\{1,\dots,4\}$, by lemma \ref{ts55}(iii).
				For $k=0$ we have: $tu_1=(ts)^2s^{-1}t^{-1}u_1\subset u_1(\underline{\underline{zt^{-1}u_1}})\subset u_1U\stackrel{\ref{su55}}{\subset}U$. It remains to prove the case where $k=5$:\\
				$\hspace*{-0.2cm}\small{\begin{array}[t]{lcl}
					z^5tu_1&=&z^5t(R+Rs^{-1}+Rs^{-2})\\
					&\subset&\underline{z^5u_2}+Rz^5t(ts)^{-2}tst+Rz^5(R+Rt^{-1}+Rt^{-2})s^{-2}\\
					&\subset&U+\underline{\underline{(z^4u_2u_1)t}}+\underline{z^5u_1}+\underline{\underline{z^5t^{-1}u_1}}+Rz^5t^{-1}(st)^{-2}sts^{-1}\\
					&\subset&U+Rz^4t^{-1}(R+Rs^{-1}+Rs^{-2})ts^{-1}\\
					&\subset&U+\underline{z^4u_1}+Rz^4(st)^{-2}st^2s^{-1}+Rz^4(st)^{-2}sts^{-1}ts^{-1}\\
					&\subset&U+u_1\underline{\underline{z^3u_2u_1}}+Rz^3sts^{-1}(R+Rt^{-1}+Rt^{-2})s^{-1}\\
					&\subset&U+u_1U+u_1\underline{\underline{z^3tu_1}}+Rz^3st(ts)^{-2}t+Rz^3st(ts)^{-2}ts(st)^{-2}st\\
					&\subset&U+u_1U+\underline{z^2u_1u_2}+u_1\underline{\underline{(zu_2u_1)t}}\\
					&\subset&U+u_1U.
					\end{array}}$\\\\
				The result follows from proposition \ref{su55}.
			\end{proof}
			
			\begin{cor}
				The BMR freeness conjecture holds for the generic Hecke algebra $H_{G_5}$.
			\end{cor}
			\begin{proof}
				By theorem \ref{thm55} we have that $H_{G_5}=U$. Therefore,	the result follows from proposition \ref{BMR PROP} since, by definition, $U$ is generated as $R$-module by $|G_5|=72$ elements.
			\end{proof}
			
			\subsubsection{\textbf{The case of} $\mathbf{G_6}$}
			Let $R=\ZZ[u_{s,i}^{\pm},u_{t,j}^{\pm}]_{\substack{1\leq i\leq 2 \\1\leq j\leq 3}}$ and let $H_{G_6}:=\langle s,t\;|\; ststst=tststs,\prod\limits_{i=1}^{2}(s-u_{s,i})=\prod\limits_{j=1}^{3}(t-u_{t,j})=0\rangle$ be the generic Hecke algebra associated to $G_6$. Let $u_1$ be the subalgebra of $H_{G_5}$ generated by $s$ and $u_2$ the subalgebra of $H_{G_6}$ generated by $t$. We recall that $z:=(st)^3=(ts)^3$  generates the center of the associated complex braid group and that $|Z(G_6)|=4.$ We set $U=\sum\limits_{k=0}^3z^ku_2u_1u_2$. By the definition of $U$ we have the following remark:
			\begin{rem} $u_2Uu_2\subset U$.
				\label{r66}
			\end{rem}
			
			Our goal is to prove that $H_{G_6}=U$ (theorem \ref{thm66}). In order to do so, we first need to prove some preliminary results.
			\begin{lem}
				\mbox{}
				\vspace*{-\parsep}
				\vspace*{-\baselineskip}\\
				\begin{itemize}[leftmargin=0.6cm]
					\item [(i)] For every $k\in\{0,1,2\}$, $z^ku_1tu_1\subset U$.
					\item[(ii)]For every $k\in\{1,2,3\}$, $z^ku_1t^{-1}u_1\subset U$.
					\item[(iii)]For every $k\in\{1,2\}$, $z^ku_1u_2u_1\subset U$.
				\end{itemize}
				\label{sts66}
			\end{lem}
			\begin{proof}
				By definition, $u_2=R+Rt+Rt^{-1}$. Therefore, we only need to prove (i) and (ii). 
				\begin{itemize} [leftmargin=0.6cm]
					\item[(i)]$z^ku_1tu_1=z^k(R+Rs)t(R+Rs)\subset z^ku_2u_1u_2+z^ksts\subset U+z^k(st)^3t^{-1}s^{-1}t^{-1}\subset U+z^{k+1}u_2u_1u_2$.  The result follows from the definition of $U$.
					\item[(ii)]$z^ku_1t^{-1}u_1=z^k(R+Rs^{-1})t^{-1}(R+Rs^{-1})\subset z^ku_2u_1u_2+z^ks^{-1}t^{-1}s^{-1}\subset U+z^k(ts)^{-3}tst\subset U+z^{k-1}u_2u_1u_2$. The result follows again from the definition of $U$.
					\qedhere
				\end{itemize}
			\end{proof}
			We can now prove the main theorem of this section.
			\begin{thm} $H_{G_6}=U$.
				\label{thm66}
			\end{thm}
			\begin{proof}
				Since $1\in U$, it is enough to prove that $U$ is a left-sided ideal of $H_{G_6}$. For this purpose, one may check that $sU$ and $tU$ are subsets of $U$. However, by the definition of $U$, we only have to prove that $sU\subset U$. By the definition of $U$ and by remark \ref{r66}, we must prove that for every $k\in\{0,\dots,3\}$, $z^ksu_2u_1\subset U$. However, this holds for every $k\in\{1,2\}$, by lemma \ref{sts66}(iii).
				
				For $k=0$ we have: $su_2u_1\subset s(R+Rt+Rt^2)(R+Rs)\subset u_2u_1u_2+u_1tu_1+Rst^2s$. By the definition of $U$ and lemma \ref{sts66}(i), it will be sufficient to prove that $st^2s\in U$. We have:
				$$\small{\begin{array}[t]{lcl}st^2s&=&(st)^3t^{-1}s^{-1}t^{-1}s^{-1}ts\\
					&\in& zu_2s^{-1}t^{-1}(R+Rs)ts\\
					&\in&zu_2u_1u_2+zu_2s^{-1}t^{-1}(st)^3t^{-1}s^{-1}t^{-1}\\
					&\in& U+u_2(z^2u_1u_2u_1)u_2.			
					\end{array}}$$
				The result follows from lemma \ref{sts66}(iii) and remark \ref{r66}.
				
				It remains to prove the case where $k=3$. We have: $z^3su_2u_1\subset s(R+Rt^{-1}+Rt^{-2})(R+Rs^{-1})\subset u_2u_1u_2+u_1t^{-1}u_1+Rst^{-2}s^{-1}$. By the definition of $U$ and lemma \ref{sts66}(ii), we need to prove that $st^{-2}s^{-1}\in U$. We have:
				$$\small{\begin{array}[t]{lcl}
					z^3st^{-2}s^{-1}&\in&z^3(R+Rs^{-1})t^{-2}s^{-1}\\
					&\in&z^3u_2u_1u_2+Rz^3(ts)^{-3}tstst^{-1}s^{-1}\\
					&\in&U+z^2u_2st(R+Rs^{-1})t^{-1}s^{-1}\\
					&\in&U+z^2u_2u_1u_2+z^2u_2st(ts)^{-3}tst\\
					&\in&U+u_2(zu_1u_2u_1)u_2.
					\end{array}}$$
				The result follows again from lemma \ref{sts66}(iii) and remark \ref{r66}.
			\end{proof}
			\begin{cor}
				The BMR freeness conjecture holds for the generic Hecke algebra $H_{G_6}$.
			\end{cor}
			\begin{proof}
				By theorem \ref{thm66} we have that $H_{G_6}=U$. By definition, $U=\sum\limits_{k=0}^3z^ku_2u_1u_2$. We expand $u_1$ as $R+Rs$ and we have that $U=\sum\limits_{k=0}^3(z^ku_2+z^ku_2su_2)$. Therefore, $U$ is generated as $u_2$-module by 16 elements. Since $u_2$ is generated as $R$-module by 3 elements, we have that $H_{G_6}$ is generated as
				$R$-module by $|G_6|=48$ elements and the result follows from proposition \ref{BMR PROP}.
			\end{proof}
		\subsubsection{\textbf{The case of} $\mathbf{G_7}$}
		
		Let $R=\ZZ[u_{s,i}^{\pm},u_{t,j}^{\pm},u_{u,l}^{\pm}]_{\substack{1\leq i\leq 2 \\1\leq j,l\leq 3}}$. We also let $$H_{G_{7}}=\langle s,t,u\;|\; stu=tus=ust, \;\prod\limits_{i=1}^{2}(s-u_{s,i})=\prod\limits_{j=1}^{3}(t-u_{t,j})=\prod\limits_{l=1}^{3}(u-u_{u,l})=0\rangle$$ be the generic Hecke algebra associated to $G_{7}$. Let $u_1$ be the subalgebra of $H_{G_{7}}$ generated by $s$, $u_2$ the subalgebra of $H_{G_{7}}$ generated by $t$ and $u_3$ the subalgebra of $H_{G_{7}}$ generated by $u$. We recall that $z:=stu=tus=ust$  generates the center of the associated complex braid group and that $|Z(G_7)|=12$.
		We set $U=\sum\limits_{k=0}^{11}(z^ku_3u_2+z^ktu^{-1}u_2).$
		By the definition of $U$, we have the following remark.
		\begin{rem}
			$Uu_2 \subset U$.
			\label{r77}
		\end{rem}
		To make it easier for the reader to follow the calculations, we will underline the elements that  belong to $U$ by definition.  Moreover, we will use directly  remark \ref{r77}; this means that every time we have a power of $t$ at the end of an element, we may ignore it. In order to remind that to the reader, we put a parenthesis around the part of the element we consider.
		
		Our goal is to prove that $H_{G_{7}}=U$ (theorem \ref{thm77}). In order to do so, we first need  to prove some preliminary results.
		
		\begin{lem}
			\mbox{}
			\vspace*{-\parsep}
			\vspace*{-\baselineskip}\\
			\begin{itemize}[leftmargin=0.8cm]
				\item[(i)]For every $k\in\{0,\dots,10\}$, $z^ku_1\subset U$.
				\item [(ii)]For every $k\in\{0,\dots,9\}$, $z^ku_2u\subset U$.
				\item[(iii)]For every $k\in\{1,\dots,10\}$, $z^ku_1u_3\subset U$.
				\label{l77}
			\end{itemize}
		\end{lem}
		\begin{proof}
			\mbox{}
			\vspace*{-\parsep}
			\vspace*{-\baselineskip}\\
			\begin{itemize}[leftmargin=0.6cm]
				\item[(i)] 
				
				$z^ku_1=z^k(R+Rs)\subset \underline{z^ku_3}+Rz^k(stu)u^{-1}t^{-1}\subset U+\underline{z^{k+1}u_3u_2}$.
				\item[(ii)]$z^ku_2u=z^k(R+Rt+Rt^2)u\subset 
				\underline{z^ku_3}+Rz^k(tus)s^{-1}+Rz^kt(tus)s^{-1}\subset
				U+z^{k+1}u_1+Rz^{k+1}ts^{-1}$. Since $z^{k+1}u_1\in U$ (see (i)), we only have to prove that $z^{k+1}ts^{-1}\in U$. We have: $z^{k+1}ts^{-1}\in z^{k+1}t(R+Rs)\subset \underline{z^{k+1}u_2}+z^{k+1}t(stu)u^{-1}t^{-1}\subset U+\underline{z^{k+2}tu^{-1}u_2}$.
				\item[(iii)] $z^ku_1u_3=z^k(R+Rs^{-1})u_3\subset \underline{z^ku_3}+z^k(s^{-1}u^{-1}t^{-1})tu_3\subset U+z^{k-1}tu_3$.
				The result follows from the definition of $U$ and (i), if we expand $u_3$ as $R+Ru^{-1}+Ru$.
				\qedhere
			\end{itemize}
		\end{proof}
		
		From now on, we will double-underline the elements described in lemma \ref{l77} and we will use directly the fact that these elements are inside $U$. 
		\begin{prop}$u_3U\subset U$.
			\label{pr77}
		\end{prop}
		\begin{proof}
			Since $u_3=R+Ru+Ru^2$, it is enough to prove that $uU\subset U$. By the definition of $U$ and remark \ref{r77} it will be sufficient to prove that for every $k\in\{0,\dots,11\}$, $z^kutu^{-1}\in U$. We distinguish the following cases:
			\begin{itemize}[leftmargin=*]
				\item \underline{$k\in\{0,\dots,7\}$}:

				$\hspace*{-0.6cm}  \small{\begin{array}[t]{lcl}
					z^kutu^{-1}&\in&z^kut(R+Ru+Ru^2)\\
					&\in&\underline{z^ku_3}+z^ku(tus)s^{-1}+z^ku(tus)s^{-1}u\\
					&\in&U+Rz^{k+1}u(R+Rs)+Rz^{k+1}u(R+Rs)u\\
					&\in&U+\underline{z^{k+1}u_3}+Rz^{k+1}(ust)t^{-1}+Rz^{k+1}(ust)t^{-1}u\\
					&\in&U+\underline{z^{k+2}u_2}+\underline{\underline{z^{k+2}u_2u}}.
					\end{array}}$
				
				\item \underline{$k\in\{8,\dots,11\}$}:
				
				$\hspace*{-0.6cm}\small{\begin{array}[t]{lcl}
					z^kutu^{-1}&\in&z^ku(R+Rt^{-1}+Rt^{-2})u^{-1}\\
					&\in&\underline{z^ku_3}+Rz^ku(t^{-1}s^{-1}u^{-1})usu^{-1}+Rz^kut^{-2}(u^{-1}t^{-1}s^{-1})st\\
					&\in&U+z^{k-1}u_3(R+Rs^{-1})u^{-1}+Rz^{k-1}ut^{-2}(R+Rs^{-1})t\\
					&\in&U+\underline{z^{k-1}u_3}+Rz^{k-1}u_3(s^{-1}u^{-1}t^{-1})t+\underline{z^{k-1}u_3u_2}+
					Rz^{k-1}ut^{-1}(t^{-1}s^{-1}u^{-1})ut\\
					&\in&U+\underline{z^{k-2}u_3u_2}+Rz^{k-2}(R+Ru^{-1}+Ru^{-2})t^{-1}ut\\
					&\in&U+\underline{\underline{(z^{k-2}u_2u)t}}+Rz^{k-2}(u^{-1}t^{-1}s^{-1})sut+Rz^{k-2}u^{-1}(u^{-1}t^{-1}s^{-1})sut\\
					&\in&U+\underline{\underline{(z^{k-3}u_1u_3)t}}+Rz^{k-3}u^{-1}s(R+Ru^{-1}+Ru^{-2})t\\
					&\in&U+Rz^{k-3}u^{-1}(stu)u^{-1}+Rz^{k-3}u^{-1}(R+Rs^{-1})u^{-1}t+
					Rz^{k-3}u^{-1}(R+Rs^{-1})u^{-2}t\\
					&\in&U+\underline{z^{k-2}u_3}+\underline{z^{k-3}u_3u_2}+Rz^{k-3}u^{-1}(s^{-1}u^{-1}t^{-1})t^2+Rz^{k-3}u^{-1}(s^{-1}u^{-1}t^{-1})tu^{-1}t\\
					&\in&U+\underline{z^{k-4}u_3u_2}+Rz^{k-4}u^{-1}(R+Rt^{-1}+Rt^{-2})u^{-1}t\\
					&\in&U+\underline{z^{k-4}u_3u_2}+Rz^{k-4}(u^{-1}t^{-1}s^{-1})su^{-1}t+
					Rz^{k-4}(u^{-1}t^{-1}s^{-1})st^{-1}u^{-1}t\\
					&\in&U+\underline{\underline{(z^{k-5}u_1u_3)t}}+Rz^{k-5}s(t^{-1}s^{-1}u^{-1})usu^{-1}t\\
					&\in&U+Rz^{k-6}su(R+Rs^{-1})u^{-1}t\\
					&\in&U+\underline{\underline{(z^{k-6}u_1)t}}+Rz^{k-6}su(s^{-1}u^{-1}t^{-1})t^2\\
					&\in&U+\underline{\underline{(z^{k-7}u_1u_3)t}}.\phantom{=====================================.}
					\qedhere
					\end{array}}$
			\end{itemize}
		\end{proof}
		We can now prove the main theorem of this section.
		\begin{thm} $H_{G_7}=U$.
			\label{thm77}
		\end{thm}
		
		\begin{proof}
			Since $1\in U$, it will be sufficient to prove that $U$ is a left-sided ideal of $H_{G_7}$. For this purpose, one may check that  $sU$, $tU$ and $uU$ are subsets of $U$. However, by proposition \ref{pr77} we
			only have to prove that $tU$ and $sU$ are subsets of $U$. We recall that $z=stu$, therefore $s=zu^{-1}t^{-1}$ and $s^{-1}=z^{-1}tu$. We notice that $U=
			\sum\limits_{k=0}^{10}z^k(u_3u_2+tu^{-1}u_2)+z^{11}(u_3u_2+tu^{-1}u_2).$
			Hence, \\
			$\small{\begin{array}[t]{lcl}sU&\subset&\sum\limits_{k=0}^{10}
				z^ks(u_3u_2+tu^{-1}u_2)+
				z^{11}s(u_3u_2+tu^{-1}u_2)\\
				&\subset& \sum\limits_{k=0}^{10}z^{k+1}u^{-1}t^{-1}(u_3u_2+tu^{-1}u_2)+z^{11}(R+Rs^{-1})(u_3u_2+
				tu^{-1}u_2)\\
				&\subset& \sum\limits_{k=0}^{10}u^{-1}t^{-1}(\underline{z^{k+1}u_3u_2}+\underline{z^{k+1}tu^{-1}u_2})+
				\underline{z^{11}u_3u_2}+\underline{z^{11}tu^{-1}u_2}+z^{11}s^{-1}u_3u_2+
				z^{11}s^{-1}tu^{-1}u_2\\ 
				&\subset&u_3u_2U+z^{10}tu_3u_2+
				z^{10}tutu^{-1}u_2\\ 
				&\subset&u_3u_2U+
				t\underline{z^{10}u_3u_2}+tu(\underline{z^{10}tu^{-1}u_2})\\
				&\subset&u_3u_2u_3U.
				\end{array}}$\\\\
			By proposition \ref{pr77} we have that $u_3U\subset U$. 
			If we also suppose that $u_2U\subset U$ then, obviously, we have  $tU\subset U$ but we  also have  $sU\subset U$ (since $sU\subset u_3u_2u_3U$). Hence, in order to prove that $U=H_{G_7}$ we restrict ourselves to proving that $u_2U \subset U$.
			
			By definition, $u_2=R+Rt^{-1}+Rt^{-2}$, therefore it will be sufficient to prove that $t^{-1}U\subset U$.
			By the definition of $U$ and remark \ref{r77}, this is the same as proving that, for every $k\in\{0,\dots,11\}$, $z^kt^{-1}u_3\subset U$. 
			For $k\in\{2,\dots,11\}$ the result is obvious, since $z^kt^{-1}u_3=z^k(t^{-1}s^{-1}u^{-1})usu_3\subset u_3(\underline{\underline{z^{k-1}u_1u_3}})\subset u_3U\stackrel{\ref{pr77}}{\subset}U$. It remains to prove the case where  $k\in\{0,1\}$.  We have:
			\\\\
			$\hspace*{-0.2cm}\small{\begin{array}[t]{lcl}
				z^kt^{-1}u_3&=&z^kt^{-1}(R+Ru+Ru^2)\\
				&\subset& \underline{z^ku_2}+\underline{\underline{z^ku_2u}}+z^k(R+Rt+Rt^2)u^2\\
				&\subset&U+\underline{z^ku_3}+Rz^kt(R+Ru+Ru^{-1})+Rz^kt(tus)s^{-1}u\\
				&\subset&U+\underline{z^ku_2}+\underline{\underline{z^ku_2u}}+\underline{z^ktu^{-1}u_2}+Rz^{k+1}(tus)s^{-1}u^{-1}s^{-1}u\\
				&\subset&U+Rz^{k+2}(R+Rs)u^{-1}(R+Rs)u\\
				&\subset&U+\underline{\underline{z^{k+2}u_1}}+u_3\underline{\underline{z^{k+2}u_1u_3}}+Rz^{k+2}su^{-1}su\\
				&\subset&U+u_3U+Rz^{k+2}s(R+Ru+Ru^2)su\\
				&\subset&U+u_3U+\underline{\underline{z^{k+2}u_1u_3}}+Rz^{k+2}s(ust)t^{-1}u+Rz^{k+2}su(ust)t^{-1}u\\
				&\subset&U+u_3U+Rz^{k+3}(stu)u^{-1}t^{-2}u+Rz^{k+3}su(R+Rt+Rt^2)u\\
				&\subset&U+u_3U+u_3\underline{\underline{z^{k+4}u_2u}}+\underline{\underline{z^{k+3}u_1u_3}}+Rz^{k+3}su(tus)s^{-1}+Rz^{k+3}sut(tus)s^{-1}\\
				&\subset&U+u_3U+Rz^{k+4}su(R+Rs)+Rz^{k+4}sut(R+Rs)\\
				&\subset&U+u_3U+\underline{\underline{z^{k+4}u_1u_3}}+Rz^{k+4}s(ust)t^{-1}+
				\underline{\underline{(z^{k+4}u_1u_3)t}}+Rz^{k+4}(stu)u^{-1}t^{-1}uts\\
				&\subset&U+u_3U+\underline{\underline{(z^{k+5}u_1)t}}+z^{k+5}u_3(R+Rt+Rt^2)uts\\
				&\subset&U+u_3U+z^{k+5}u_3t(stu)u^{-1}t^{-1}+z^{k+5}u_3(tus)s^{-1}ts+
				z^{k+5}u_3t(tus)s^{-1}ts\\
				&\subset&U+u_3U+u_3\underline{z^{k+5}tu^{-1}u_2}+z^{k+5}u_3(R+Rs)ts+
				z^{k+6}u_3t(R+Rs)ts\\
				&\subset&U+u_3U+z^{k+5}u_3t(stu)u^{-1}t^{-1}+z^{k+5}u_3(stu)u^{-1}s+
				z^{k+6}u_3t^2(R+Rs^{-1})+\\&&+z^{k+6}u_3t(stu)u^{-1}s\\
				&\subset&U+u_3U+u_3\underline{z^{k+6}tu^{-1}u_2}+u_3\underline{z^{k+6}u_1}+\underline{z^{k+6}u_3u_2}+z^{k+6}u_3t^2(s^{-1}u^{-1}t^{-1})tu+\\&&+
				z^{k+6}u_3t(R+Ru+Ru^2)s\\
				&\subset&U+u_3U+u_3\underline{\underline{z^{k+5}u_2u}}+z^{k+6}u_3t(stu)u^{-1}t^{-1}+z^{k+6}u_3(tus)+
				z^{k+6}u_3(tus)s^{-1}(ust)t^{-1}\\
				&\subset&U+u_3U+u_3\underline{z^{k+7}tu^{-1}u_2}+\underline{z^{k+7}u_3}+u_3\underline{\underline{(z^{k+8}u_1)t^{-1}}}\\
				&\subset&U+u_3U.
				\end{array}}$
			\\\\
			The result follows from  proposition \ref{pr77}.	
		\end{proof}
		\begin{cor}
			The BMR freeness conjecture holds for the generic Hecke algebra $H_{G_7}$.
		\end{cor}
		\begin{proof}
			By theorem \ref{thm77} we have that $H_{G_7}=U$. The result follows from proposition \ref{BMR PROP}, since by definition $U$ is generated as a right $u_2$-module by 48 elements and, hence, as $R$-module by $|G_7|=144$ elements (recall that $u_2$ is generated as $R$-module by 3 elements).
		\end{proof}
			\subsection{The Octahedral family}
			In this family we encounter the exceptional groups $G_8$, $G_9$, $G_{10}$, $G_{11}$, $G_{12}$, $G_{13}$, $G_{14}$ and $G_{15}$. By theorem \ref{braid} we have the validity of the BMR freeness conjecture  for $G_8$. Moreover, we know the validity of the conjecture for the group $G_{12}$ (see theorem \ref{2case}). We now prove the conjecture for the rest of the exceptional groups in this family using a case-by-case analysis. We also prove in a different way the validity of the BMR freeness conjecture for the exceptional groups $G_8$ and $G_{12}$. As in the tetrahedral case, we use directly the relations (\ref{ooo}) and (\ref{oooo}).
			\subsubsection{\textbf{The case of} $\mathbf{G_8}$}
			
			Let $R=\ZZ[u_{s,i}^{\pm},u_{t,j}^{\pm}]_{\substack{1\leq i,j\leq 3 }}$ and let $ H_{G_{8}}=\langle s,t\;|\; sts=tst,\prod\limits_{i=1}^{4}(s-u_{s,i})=\prod\limits_{j=1}^{4}(t-u_{t,j})=0\rangle.$
			Let also $\bar R=\ZZ[u_{s,i}^{\pm},]_{\substack{1\leq i\leq 4 }}$. Under the specialization $\grf: R\twoheadrightarrow \bar R$, defined by $u_{t,j}\mapsto u_{s,i}$ the algebra $\bar H_{G_{8}}:=H_{G_{8}}\otimes_{\grf}\bar R$ is the generic Hecke algebra associated to $G_{8}$. Let $u_1$ be the subalgebra of $H_{G_{8}}$ generated by $s$ and $u_2$ the subalgebra of $H_{G_{8}}$ generated by $t$.
			We recall that $z:=(st)^3=(ts)^3$  generates the center of the associated complex braid group (which, in our case, is the usual braid group on 3 strands) and that $|Z(G_8)|=4$. Using the braid relation, we can also notice that $z=s^2ts^2t=t^2st^2s=st^2st^2=ts^2ts^2$. We set $U=\sum\limits_{k=0}^3(z^ku_1+z^ku_1tu_1+z^ku_1t^2)$. By the definition of $U$, we have the following remark:
			\begin{rem}
				$u_1U \subset U$.
				\label{rem8}
			\end{rem}
			
			From now on, we will underline the elements that belong to $U$  by definition. 
			Moreover, we will use directly the remark \ref{rem8}; this means that every time we have a power of  $s$ at the beginning of an element, we may ignore it. In order to remind that to the reader, we put a parenthesis around the part of the element we consider.
			Our goal is to prove that $H_{G_{8}}=U$ (theorem \ref{thm88}). 
			For this purpose, we first need to prove some preliminary results. 
			
			\begin{lem}
				For every $k\in\{1,2,3\}$, $z^ku_1u_2\subset U$.
				\label{lem88}
			\end{lem}
			\begin{proof}
				$z^ku_1u_2=z^ku_1(R+Rt+Rt^2+Rt^{-1})\subset \underline{z^ku_1}+\underline{z^ku_1tu_1}+\underline{z^ku_1t^2}+z^ku_1t^{-1}\subset U+z^ks^2(s^{-2}t^{-1}s^{-2}t^{-1})ts^2\subset U+\underline{z^{k-1}u_1tu_1}$.
			\end{proof}
			To make it easier for the reader to follow the calculations, we will double-underline the elements as described in lemma \ref{lem88} and  we will use directly the fact that these elements are inside $U$.
			
			\begin{prop}
				$Uu_1\subset U$.
				\label{pr88}
			\end{prop}
			\begin{proof}
				By the definition of $U$, it will be sufficient to prove that $z^kt^2s\in U$, for every $k\in\{0,1,2,3\}$. For $k\not=3$ the result is obvious, since $z^kt^2s=z^k(t^2st^2s)s^{-1}t^{-2}\in \underline{\underline{z^{k+1}u_1u_2}}$. It remains to prove the case where $k=3$. Since $s\in R+Rs^{-1}+Rs^{-2}+Rs^{-3}$ we have that $z^3t^2s\in \underline{z^3u_1}+\sum\limits_{m=1}^3Rz^3t^2s^{-m}\subset U+\sum\limits_{m=1}^3Rz^3(R+Rt+Rt^{-1}+Rt^{-2})s^{-m}\subset U+\underline{z^3u_1}+\underline{z^3tu_1}+\sum\limits_{m=1}^3Rz^3t^{-1}s^{-m}+\sum\limits_{m=1}^3Rz^3t^{-2}s^{-m}$. However, $z^3t^{-1}s^{-m}=z^3(t^{-1}s^{-2}t^{-1}s^{-2})s^2ts^{2-m}\in \underline{z^2u_1tu_1}$. Moreover, $z^3t^{-2}s^{-m}=z^3s(s^{-1}t^{-2}s^{-1}t^{-2})t^2s^{1-m}=z^2st^2s^{1-m}$. Therefore, it remains to prove that $z^2st^2s^{1-m}$ is inside $U$, for every $m\in\{1,2,3\}$. For $m=1$ the result follows directly from the definition of $U$. For $m=2$ we have: $z^2st^2s^{-1}\in z^2s(R+Rt+Rt^{-1}+Rt^{-2})s^{-1}\subset \underline{z^2u_1}+\underline{z^2u_1tu_1}+Rz^2s^3(s^{-2}t^{-1}s^{-2}t^{-1})ts+Rz^2s^2(s^{-1}t^{-2}s^{-1}t^{-2})t^2\subset U+\underline{zu_1tu_1}+\underline{zu_1t^2}$. Finally, for $m=3$ we have: $z^2st^2s^{-2}\in z^2s(R+Rt+Rt^{-1}+Rt^{-2})s^{-2}\subset \underline{z^2u_1}+\underline{z^2u_1tu_1}+Rz^2s(t^{-1}s^{-2}t^{-1}s^{-2})s^2t+Rz^2s^2(s^{-1}t^{-2}s^{-1}t^{-2})t^2s^{-1}\subset U+\underline{zu_1t}+Rs(zst^2s^{-1})$. The result follows if we repeat exactly the same calculations as in case where $m=2$.
			\end{proof}
			The following lemma leads us to the main theorem of this section.
			\begin{lem}
				For every $k\in\{0,1,2, 3\}$ $z^ktu_1t\subset U$.
				\label{lemm88}
			\end{lem}
			\begin{proof}
				For $k\in\{0,1,2\}$ the result is obvious, since $z^ktu_1t=z^kt(R+Rs+Rs^2+Rs^3)t\subset \underline{z^ku_1t^2}+Rz^k(tst)+Rz^k(ts^2ts^{2})s^{-2}+Rz^ks^{-1}(sts)s^2t\subset U+\underline{Rz^ksts}+\underline{z^{k+1}u_1}+Rz^ks^{-1}ts(ts^2ts^2)s^{-2}\subset U+\underline{z^{k+1}u_1tu_1}.$ It remains to prove the case where $k=3$. We first make two remarks. By lemma 2.1 in \cite{chavli} we have that $\mathbf{t^{-1}s^{-1}u_2=u_1t^{-1}s^{-1}}$. Moreover, $\mathbf{t^{-3}s^{-1}t^{-1}=s^{-1}t^{-1}s^{-3}}$. We can now prove that $z^3tu_1t\subset U$. We have:\\\\
				$\hspace*{-0.2cm}\small{\begin{array}{lcl}
					z^3tu_1t&=&z^3t(R+Rs+Rs^{-1}+Rs^{-2})t\\
					&\subset& \underline{z^3u_1t^2}+Rz^k(tst)+Rz^3ts^{-1}t+Rz^3ts^{-2}t\\
					&\subset&U+\underline{Rz^3sts}+Rz^3(R+Rt^{-1}+Rt^{-2}+Rt^{-3})s^{-1}(R+Rt^{-1}+Rt^{-2}+Rt^{-3})+\\&&+Rz^3(R+Rt^{-1}+Rt^{-2}+Rt^{-3})s^{-2}(R+Rt^{-1}+Rt^{-2}+Rt^{-3})\\
					&\subset&U+(\underline{\underline{z^3u_1u_2}})u_1+z^3\mathbf{t^{-1}s^{-1}u_2}+z^3t^{-2}s^{-1}u_2+Rz^3\mathbf{t^{-3}s^{-1}t^{-1}}+Rz^3t^{-3}s^{-1}t^{-2}+Rz^3t^{-3}s^{-1}t^{-3}\\
					&\subset&Uu_1+(\underline{\underline{z^3\mathbf{u_1t^{-1}}}})\mathbf{s^{-1}}+z^3s(s^{-1}t^{-2}s^{-1}t^{-2})u_2+(\underline{\underline{Rz^3\mathbf{s^{-1}t^{-1}}}})\mathbf{s^{-3}}+Rz^3t^{-1}(t^{-2}s^{-1}t^{-2}s^{-1})s+\\&&+Rz^3t^{-1}(t^{-2}s^{-1}t^{-2}s^{-1})st^{-1}\\
					&\subset&Uu_1+\underline{\underline{z^2u_1u_2}}+(\underline{\underline{z^2u_1u_2}})s+Rz^2t^{-1}st^{-1}\\
					&\subset&Uu_1+Rz^2s^2(s^{-2}t^{-1}s^{-2}t^{-1})ts^5(s^{-2}t^{-1}s^{-2}t^{-1})ts^2\\
					&\subset&Uu_1+(tu_1t)u_1.
					\end{array}}$\\\\
				However, as we explained in the beginning of the proof, we have that $z^ktu_1t\subset U$, for every $k\in\{0,1,2\}$.  Therefore, $(tu_1t)u_1\subset Uu_1$. As a result, $z^3tu_1t\subset Uu_1$. The result follows from proposition \ref{pr88}.
			\end{proof}
			
			\begin{thm}$H_{G_8}=U$.
				\label{thm88}
			\end{thm}
			\begin{proof}
				Since $1\in U$, it is enough to prove that $U$ is a right-sided ideal of $H_{G_8}$. For this purpose, it will be sufficient to prove that $Us$ and $Ut$ are subsets of $U$. However, by proposition \ref{pr88} we restrict ourselves to proving that $Ut\subset U$. By the definition of $U$ and remark \ref{rem8} this is the same as proving that $z^ktu_1t\subset U$ and $z^kt^3\in U$, for every $k\in\{0,\dots, 3\}$. By lemmas \ref{lem88} and \ref{lemm88} we only have to prove that $t^3\in U$. We have that $t^3=s^{-1}(st^2st^2)t^{-2}s^{-1}t\in zs^{-1}(R+Rt^{-1}+Rt+Rt^2)\subset \underline{zu_1t}+Rz(s^{-1}t^{-1}s^{-1})t+s^{-1}(ztu_1t)+zs^{-1}t^2s^{-1}t\stackrel{\ref{lemm88}}{\subset}U+(\underline{\underline{zu_1t^{-1}}})s^{-1}+s^{-1}(zt^2s^{-1}t)\subset Uu_1+s^{-1}(zt^2s^{-1}t)$. Therefore, by lemma \ref{lem88} it remains to prove that $zt^2s^{-1}t\in U$. We have that $zt^2s^{-1}t\in zt^2(R+Rs+Rs^2+Rs^3)t\subset \underline{\underline{zu_2}}+Rzt^2st+Rzt^2s^2t+Rzt^2s^3t$. However, $zt^2st=\underline{zsts^2}$. Moreover, $zt^2s^2t=zt(ts^2ts^2)s^{-2}\in (\underline{z^2t})s\subset Uu_1\stackrel{\ref{lem88}}{\subset U}$. It remains to prove that $zt^2s^3t\in U$. Indeed, $zt^2s^3t=zs^{-1}(st^2st^2)t^{-2}(s^2ts^2t)t^{-1}s^{-2}\in (\underline{\underline{z^3u_1u_2}})s^{-2}\subset Uu_1$. The result follows from lemma \ref{lem88}.
			\end{proof}
			\begin{cor}
				The BMR freeness conjecture holds for the generic Hecke algebra $\bar H_{G_{8}}$.
			\end{cor}
			\begin{proof}
				By theorem \ref{thm88} we have that $H_{G_{12}}=U$ and, hence, $H_{G_{8}}$ is generated as right $u_1$-module by 24 elements and, hence, as $R$-module by $|G_{8}|=96$ elements (recall that $u_1$ is generated as $R$-module by 4 elements). Therefore, $\bar H_{G_{8}}$ is generated as $\bar R$-module by $|G_{8}|=96$ elements, since the action of $\bar R$ factors through $R$. The result follows from proposition \ref{BMR PROP}.
			\end{proof}
			\subsubsection{\textbf{The case of} $\mathbf{G_9}$}
			Let $R=\ZZ[u_{s,i}^{\pm},u_{t,j}^{\pm}]_{\substack{1\leq i\leq 2 \\1\leq j\leq 4}}$ and let $H_{G_{9}}=\langle s,t\;|\; ststst=tststs, \prod\limits_{i=1}^{2}(s-u_{s,i})=\prod\limits_{j=1}^{4}(t-u_{t,j})=0\rangle$ be the generic Hecke algebra associated to $G_{9}$. Let $u_1$ be the subalgebra of $H_{G_{9}}$ generated by $s$ and $u_2$ be the subalgebra of $H_{G_{9}}$ generated by $t$.
			We recall that $z:=(st)^3=(ts)^3$  generates the center of the associated complex braid group and that $|Z(G_9)|=8$. We set $U=\sum\limits_{k=0}^7(z^ku_2u_1u_2+z^ku_2st^{-2}s)$. By the definition of $U$, we have the following remark.
			\begin{rem} $u_2U\subset U$.
				\label{rem9}
			\end{rem}
			From now on, we will underline the elements that by definition belong to $U$.  Moreover, we will use directly the remark \ref{rem9}; this means that every time we have a power of $t$ at the beginning of an element, we may ignore it. In order to remind that to the reader, we put a parenthesis around the part of the element we consider.

			Our goal is to prove that $H_{G_{9}}=U$ (theorem \ref{thm9}). The next proposition provides the necessary conditions for this to be true.
			
			\begin{prop}
				If $z^ku_1u_2u_1\subset U$ and $z^kst^{-2}st\in U$ for every $k\in\{0,\dots,7\}$, then $H_{G_9}=U$.
				\label{pr9}
			\end{prop}
			\begin{proof}
				Since $1\in U$, it is enough to prove that $U$ is a right-sided ideal of $H_{G_9}$. For this purpose, one may check that $Us$ and $Ut$ are subsets of $U$. By the definition of $U$ we have that  $Us\subset \sum\limits_{k=0}^7u_2(z^ku_1u_2u_1)$ and $Ut\subset \sum\limits_{k=0}^7(\underline{z^ku_2u_1u_2}+z^ku_2st^{-2}st)\subset U+\sum\limits_{k=0}^7u_2(z^kst^{-2}st).$ The result follows from hypothesis and remark \ref{rem9}.
			\end{proof}
			As a first step, we prove the conditions of the above proposition for a smaller range of the values of $k$.
			\begin{lem}
				\mbox{}
				\vspace*{-\parsep}
				\vspace*{-\baselineskip}\\
				\begin{itemize}[leftmargin=0.6cm]	
					\item[(i)] For every $k\in\{0,\dots,6\}$, $z^ku_1tu_1\subset U$.
					\item[(ii)] For every $k\in\{1,\dots,7\}$, $z^ku_1t^{-1}u_1\subset U$.
					\item[(iii)] For every $k\in\{1,\dots,6\}$, $z^ku_1u_2u_1\subset U$.
					
					\item[(iv)]For every $k\in\{0,\dots,5\}$, $z^ku_1t^2u_1t\subset U+z^{k+2}u_2u_1u_2u_1$. Therefore, for every $k\in\{0,\dots,4\}$, 
					$z^ku_1t^2st\subset U$.
					\item[(v)]For every $k\in\{1,\dots,5\}$, $z^ku_1u_2u_1t\subset U+z^{k+2}u_2u_1u_2u_1$. Therefore, for every $k\in\{0,\dots,4\}$, $z^ku_1u_2u_1t\subset U$.
					\qedhere
				\end{itemize}
				\label{lem9}
			\end{lem}
			\begin{proof}
				\mbox{}
				\vspace*{-\parsep}
				\vspace*{-\baselineskip}\\
				\begin{itemize}[leftmargin=0.6cm]
					\item [(i)] $z^ku_1tu_1=z^k(R+Rs)t(R+Rs)\subset \underline{z^ku_2u_1u_2}+Rz^ksts\subset U+Rz^k(st)^3t^{-1}s^{-1}t^{-1}\subset U+\underline{z^{k+1}u_2u_1u_2}$.
					\item [(ii)] $z^ku_1t^{-1}u_1=z^k(R+Rs^{-1})t^{-1}(R+Rs^{-1})\subset \underline{z^ku_2u_1u_2}+Rz^ks^{-1}t^{-1}s^{-1}\subset U+Rz^k(ts)^{-3}tst\subset U+\underline{z^{k-1}u_2u_1u_2}$.
					\item[(iii)]Since $u_2=R+Rt+Rt^{-1}+Rt^{-2}$, by the definition of $U$ and by (i) and (ii) we only have to prove that, for every $k\in\{1,\dots,6\}$,  $z^ku_1t^{-2}u_1 \subset U$. Indeed, $z^ku_1t^{-2}u_1=z^k(R+Rs)t^{-2}(R+Rs)\subset \underline{z^ku_2u_1u_2}+\underline{z^ku_2st^{-2}s}$.
					\item [(iv)] $z^ku_1t^2u_1t=z^ku_1t^2(R+Rs)t\subset
					\underline{z^ku_1u_2}+z^ku_1t^2st$. However, $z^ku_1t^2st\subset z^k(R+Rs)t^2st
					\subset\underline{z^ku_2u_1u_2}+Rz^kst(ts)^3s^{-1}t^{-1}s^{-1}\
					\subset U+Rz^{k+1}sts^{-1}t^{-1}s^{-1}$. We notice that $z^{k+1}sts^{-1}t^{-1}s^{-1}\in z^{k+1}st(R+Rs)t^{-1}s^{-1}\subset \underline{z^{k+1}u_2}+Rz^{k+1}(st)^3t^{-1}s^{-1}t^{-2}s^{-1}\subset
					U+z^{k+2}u_2u_1u_2u_1.$
					Since $z^{k+2}u_2u_1u_2u_1=u_2(z^{k+2}u_1u_2u_1)$ and since, for every $k\in\{0,\dots, 4\}$, we have $k+2\in \{2,\dots,6\}$, we can use (iii) and we have that for every $k\in\{0,\dots,4\}$, $z^ku_1t^2st\subset U$.

					\item [(v)] $z^ku_1u_2u_1t=z^ku_1(R+Rt+Rt^{-1}+Rt^2)u_1t\subset
					\underline{z^ku_1u_2}+
					z^ku_1tu_1t+z^ku_1t^{-1}u_1t+z^ku_1t^2u_1t$. However, by (iv) we have that $z^ku_1t^2u_1t\subset U$. Therefore, it will be sufficient to prove that 
					$A:=z^ku_1tu_1t+z^ku_1t^{-1}u_1t$ is a subset of $U$. For this purpose, we use different definitions of $u_1$ and we have: $A=
					z^k(R+Rs)t(R+Rs)t+z^k(R+Rs^{-1})t^{-1}(R+Rs^{-1})t\subset 
					\underline{z^ku_2u_1u_2}+Rz^kstst+Rz^ks^{-1}t^{-1}s^{-1}t+z^{k+2}u_2u_1u_2u_1$. However, $z^kstst=z^k(st)^3t^{-1}s^{-1}\in \underline{z^{k+1}u_2u_1u_2}$ and, similarly, $z^ks^{-1}t^{-1}s^{-1}t\in \underline{z^{k-1}u_2u_1u_2}$.
					Therefore, $A\subset 
					U+z^{k+2}u_2u_1u_2u_1.$
					Using the same arguments as in (iv), we can use (iii) and we have that for every $k\in\{0,\dots,4\}$, $z^{k+2}u_2u_1u_2u_1\subset U$ and, hence, $A\subset U$.
					\qedhere
				\end{itemize}
			\end{proof}
			To make it easier for the reader to follow the calculations, we will double-underline the elements described in lemma \ref{lem9} and  we will use directly the fact that these elements are inside $U$. The next proposition proves the first condition of \ref{pr9}.
			\begin{prop}For every $k\in\{0,\dots,7\}$, $z^ku_1u_2u_1\subset U$.
				\label{prr9}
			\end{prop}
			\begin{proof}
				By lemma \ref{lem9}(iii), we need to prove the cases where $k\in\{0,7\}$. We have:
				\begin{itemize}[leftmargin=*]
					\item \underline{$k=0$}:
					
					$\hspace*{-0.6cm}\small{\begin{array}[t]{lcl}u_1u_2u_1&=&(R+Rs)(R+Rt+Rt^{-2}+Rt^3)(R+Rs)\\
						&\subset& \underline{u_2u_1u_2}+\underline{\underline{u_1tu_1}}+\underline{Rst^{-2}s}+Rst^3s\\
						&\subset&U+Rt^{-1}(ts)^3s^{-1}t^{-1}s^{-1}t^2s\\
						&\subset&U+zu_2s^{-1}t^{-1}(R+Rs)t^2s\\
						&\subset&U+\underline{\underline{u_2(zu_1tu_1)}}+zu_2(R+Rs)t^{-1}st^2s\\
						&\subset&U+\underline{\underline{u_2(zu_1u_2u_1)}}+zu_2s(R+Rt+Rt^2+Rt^3)st^2s\\
						&\subset&U+\underline{\underline{u_2(zu_1u_2u_1)}}+zu_2(st)^3t^{-1}s^{-1}ts+zu_2st(ts)^3s^{-1}t^{-1}s^{-1}ts+zu_2st^2(ts)^3s^{-1}t^{-1}s^{-1}ts\\
						&\subset&U+\underline{\underline{u_2(z^2u_1tu_1)}}+z^2u_2sts^{-1}t^{-1}(R+Rs)ts+z^2u_2st^2s^{-1}t^{-1}(R+Rs)ts\\
						&\subset&U+\underline{z^2u_2u_1u_2}+z^2u_2st(R+Rs)t^{-1}sts+z^2u_2st^2(R+Rs)t^{-1}sts\\
						&\subset&U+\underline{\underline{u_2(z^2u_1tu_1})}+z^2u_2stst^{-1}sts+z^2u_2(st)^3t^{-1}+
						z^2u_2st^2st^{-1}sts\\
						&\subset&U+z^2u_2stst^{-1}sts+\underline{z^3u_2}+
						z^2u_2st^2st^{-1}sts
						\end{array}}$
					\\\\
					It remains to prove that $B:=z^2u_2stst^{-1}sts+
					z^2u_2st^2st^{-1}sts$ is a subset of $U$. For this purpose, we expand $t^{-1}$ as a linear combination of 1, $t$, $t^2$ and $t^3$ and we have: \\ \\
					$\hspace*{-0.3cm}\small{\begin{array}{lcl}
						B&\subset&z^2u_2sts(R+Rt+Rt^2+Rt^3)sts+
						z^2u_2st^2s(R+Rt+Rt^2+Rt^3)sts\\ 
						&\subset&z^2u_2sts^2ts+z^2u_2(st)^3s+z^2u_2(ts)^3s^{-2}(st)^3t^{-1}+
						z^2u_2(ts)^3s^{-1}t(ts)^3s^{-1}t^{-1}+z^2u_2st^2s^2ts+\\&&+Rz^2st(ts)^3+z^2u_2st(ts)^3s^{-1}t^{-1}s^{-1}(ts)^3s^{-1}t^{-1}+z^2u_2st^2st^2(ts)^3s^{-1}t^{-1}\\ 
						&\subset&z^2u_2st(R+Rs)ts+\underline{z^3u_2u_1}+\underline{z^4u_2u_1u_2}+
						z^4u_2(R+Rs)t(R+Rs)t^{-1}+z^2u_2st^2(R+Rs)ts+\\&&+\underline{z^3u_2u_1u_2}+z^4u_2sts^{-1}t^{-1}s^{-2}t^{-1}+z^3u_2st^2st^2(R+Rs)t^{-1}\\ 
						&\subset&U+\underline{\underline{u_2(z^2u_1u_2u_1)}}+z^2u_2(st)^3t^{-1}+\underline{z^4u_2u_1u_2}+z^4u_2stst^{-1}+\underline{\underline{u_2(z^2u_1u_2u_1)}}+
						z^2u_2st^2sts+\\&&+z^4u_2sts^{-1}t^{-1}(R+Rs^{-1})t^{-1}+\underline{\underline{u_2(z^3u_1t^2st)}}+z^3u_2st^2st^2st^{-1}\\
						&\subset&U+\underline{z^3u_2}+z^4u_2(ts)^3s^{-1}t^{-2}+z^2u_2st(ts)^3s^{-1}t^{-1}+z^4u_2st(R+Rs)t^{-2}+z^4u_2st^2(st)^{-3}s+\\&&+z^3u_2st(ts)^3s^{-1}t^{-1}s^{-1}tst^{-1}\\ 
						&\subset&U+\underline{z^5u_2u_1u_2}+z^3u_2st(R+Rs)t^{-1}+\underline{z^4u_2u_1u_2}+z^4u_2(st)^3t^{-1}s^{-1}t^{-3}+
						\underline{\underline{u_2(z^3u_1u_2u_1)}}+\\&&+z^4u_2st(R+Rs)t^{-1}s^{-1}tst^{-1}\\
						&\subset&U+\underline{(z^3+z^4+z^5)u_2u_1u_2}+z^3u_2(ts)^3s^{-1}t^{-2}+
						z^4u_2(ts)^3s^{-1}t^{-2}(R+Rs)tst^{-1}\\
						&\subset&U+\underline{z^4u_2u_1u_2}+z^5u_2(st)^{-3}sts^2t^{-1}+z^5u_2s^{-1}t^{-3}(ts)^3s^{-1}t^{-2}\\
						\end{array}}$\\

					$\hspace*{-0.3cm}\small{\begin{array}{lcl}
						\phantom{B}

						&\subset&z^4u_2st(R+Rs)t^{-1}+z^6u_2s^{-1}(R+Rt^{-1}+Rt^{-2}+Rt)s^{-1}t^{-2}\\
						&\subset&\underline{z^4u_2u_1}+z^4u_2(ts)^3s^{-1}t^{-2}+
						\underline{z^6u_2u_1u_2}+z^6u_2(st)^{-3}st^{-1}+z^6u_2(st)^{-3}sts^2(ts)^{-3}tst^{-1}+\\&&+z^6u_2(R+Rs)t(R+Rs)t^{-2}\\ 
						&\subset&U+\underline{(z^5+z^6)u_2u_1u_2}+z^4u_2st(R+Rs)tst^{-1}+
						z^6u_2stst^{-2}\\
						&\subset&U+z^4u_2st^2st^{-1}+z^4u_2(ts)^3+z^6u_2(ts)^3s^{-1}t^{-3}\\
						&\subset&U+z^4u_2s(R+Rt+Rt^{-1}+Rt^{-2})st^{-1}+\underline{z^5u_2}+\underline{z^7u_2u_1u_2}\\
						&\subset&U+\underline{z^4u_2u_1u_2}+z^4u_2(ts)^3s^{-1}t^{-2}+
						z^4u_2(R+Rs^{-1})t^{-1}(R+Rs^{-1})t^{-1}+\\&&+z^4u_2(R+Rs^{-1})t^{-2}(R+Rs^{-1})t^{-1}\\
						&\subset&U+\underline{(z^4+z^5)u_2u_1u_2}+z^4u_2s^{-1}t^{-1}s^{-1}t^{-1}+
						z^4u_2s^{-1}t^{-2}s^{-1}t^{-1}\\
						&\subset&U+z^4u_2(st)^{-3}s+z^4u_2s^{-1}t^{-1}(ts)^{-3}sts\\
						&\subset&U+\underline{z^3u_2u_1}+z^3u_2s^{-1}t^{-1}(R+Rs^{-1})ts\\
						&\subset&U+\underline{z^3u_2}+z^3u_2(st)^{-3}st^2s\\
						&\subset& U+\underline{\underline{u_2(z^2u_1u_2u_1)}}.
						\end{array}}$\\
					\item \underline{$k=7$}:
					
					$\hspace*{-0.6cm}\small{\begin{array}[t]{lcl}z^7u_1u_2u_1&=&z^7(R+Rs)(R+Rt^{-1}+Rt^{-2}+Rt^{-3})(R+Rs)\\
						&\subset& \underline{z^7u_2u_1u_2}+\underline{\underline{z^7u_1t^{-1}u_1}}+\underline{Rz^7st^{-2}s}+Rz^7st^{-3}s\\
						&\subset&U+Rz^7(R+Rs^{-1})t^{-3}(R+Rs^{-1})\\
						&\subset&U+\underline{z^7u_2u_1u_2}+Rz^7s^{-1}t^{-3}s^{-1}\\
						&\subset&U+Rz^7(ts)^{-3}tstst^{-2}s^{-1}\\
						&\subset&U+Rz^6tst(R+Rs^{-1})t^{-2}s^{-1}\\
						&\subset&U+\underline{\underline{t(z^6u_1u_2u_1)}}+Rz^6t(R+Rs^{-1})ts^{-1}t^{-2}s^{-1}\\
						&\subset&U+\underline{\underline{t^2(z^6u_1u_2u_1)}}+Rz^6ts^{-1}(R+Rt^{-1}+Rt^{-2}+Rt^{-3})s^{-1}t^{-2}s^{-1}\\

						&\subset&U+\underline{\underline{t(z^6u_1u_2u_1)}}+Rz^6t^2(st)^{-3}st^{-1}s^{-1}+Rz^6ts^{-1}t^{-1}(st)^{-3}stst^{-1}s^{-1}+\\&&+Rz^6ts^{-1}t^{-2}(st)^{-3}stst^{-1}s^{-1}\\
						&\subset&U+\underline{\underline{t^2(z^5u_1u_2u_1)}}+Rz^5ts^{-1}t^{-1}(R+Rs^{-1})t(R+Rs^{-1})t^{-1}s^{-1}+\\&&+
						Rz^5ts^{-1}t^{-2}st(R+Rs^{-1})t^{-1}s^{-1}\\

						&\subset&U+\underline{\underline{t(z^5u_1u_2u_1)}}+Rz^5t(sts)^{-1}t(sts)^{-1}+Rz^5ts^{-1}t^{-2}(R+Rs^{-1})ts^{-1}t^{-1}s^{-1}\\
						&\subset&U+Rz^5t(ts)^{-3}tst^3(st)^{-3}st+Rz^5t(ts)^{-3}t+Rz^5ts^{-1}t^{-2}s^{-1}t^2(st)^{-3}st\\
						&\subset&U+\underline{\underline{t^2(z^3u_1u_2u_1t)}}+\underline{z^4u_2}+
						Rz^4ts^{-1}t^{-2}s^{-1}t^2st.
						\end{array}}$
					\\\\
					It remains to prove that the element $z^4ts^{-1}t^{-2}s^{-1}t^2st$ is inside $U$. For this purpose, we expand $t^2$ as a linear combination of 1, $t$, $t^{-1}$ and $t^{-2}$ and we have:
					\\\\
					$\hspace*{-0.3cm}\small{\begin{array}{lcl}
						z^4ts^{-1}t^{-2}s^{-1}t^2st&\in&				z^4ts^{-1}t^{-2}s^{-1}(R+Rt+Rt^{-1}+Rt^{-2})st\\
						&\in&U+\underline{z^4u_2u_1u_2}+Rz^4ts^{-1}t^{-2}(R+Rs)tst+
						Rz^4ts^{-1}t^{-2}s^{-1}t^{-1}(R+Rs^{-1})t+\\&&+Rz^4ts^{-1}t^{-2}s^{-1}t^{-2}(R+Rs^{-1})t\\
						&\in&U+\underline{\underline{t(z^4u_1u_2u_1t)}}+Rz^4ts^{-1}t^{-3}(ts)^3s^{-1}+\underline{\underline{t(z^4u_1u_2u_1)}}+
						Rz^4ts^{-1}t^{-1}(st)^{-3}st^2+\\&&+Rz^4ts^{-1}t^{-1}(st)^{-3}sts+
						Rz^4ts^{-1}t^{-1}(st)^{-3}stst^{-1}s^{-1}t\\
						&\in&U+\underline{\underline{t(z^5u_1u_2u_1)}}+Rz^3ts^{-1}t^{-1}(R+Rs^{-1})t^2+
						Rz^3ts^{-1}t^{-1}(R+Rs^{-1})ts+\\&&+Rz^3ts^{-1}t^{-1}st(R+Rs^{-1})t^{-1}s^{-1}t\\
						&\in&U+\underline{z^3u_2u_1u_2}+Rz^3t^2(st)^{-3}st^3+Rz^3t^2(ts)^{-3}st^2s+\\&&+Rz^3ts^{-1}t^{-1}(R+Rs^{-1})ts^{-1}t^{-1}s^{-1}t\\
						&\in&U+\underline{z^2u_2u_1u_2}+\underline{z^2u_2st^2s}+Rz^3ts^{-2}t^{-1}s^{-1}t+Rz^3t^2(st)^{-3}st^3(st)^{-3}st^2\\
						&\in&U+Rz^3t(R+Rs^{-1})t^{-1}s^{-1}t+Rzt^2s(R+Rt+Rt^2+Rt^{-1})st^2\\
						&\in&U+\underline{(z+z^3)u_2u_1u_2}+Rz^3t^2(st)^{-3}st^2+Rzt(ts)^3s^{-1}t+Rzt(ts)^3s^{-1}t^{-1}s^{-1}tst^2+\\&&+Rzt^2(R+Rs^{-1})t^{-1}(R+Rs^{-1})t^2\\
						&\in&U+\underline{(z+z^2)u_2u_1u_2}+Rz^2ts^{-1}t^{-1}(R+Rs)tst^2+
						Rzt^2s^{-1}t^{-1}s^{-1}t^2\\
						&\in&U+\underline{z^2u_2}+Rz^2ts^{-1}t^{-2}(ts)^3s^{-1}t+Rzt^3(st)^{-3}st^3\\
						&\in&U+\underline{\underline{t(z^3u_1u_2u_1t)}}+\underline{u_2u_1u_2}.\phantom{===========================}\qedhere
						\end{array}}$
				\end{itemize}
			\end{proof}
			\begin{cor} $Uu_1\subset U$
				\label{cor9}
			\end{cor}
			\begin{proof}
				By the definition of $U$ we have that
				$Uu_1\subset \sum\limits_{k=0}^7(z^ku_2u_1u_2u_1+z^ku_2st^{-2}u_1)$.
				By proposition \ref{prr9}, we only have to prove that for every $k\in\{0,\dots,7\}$, $z^ku_2st^{-2}u_1\subset U$, which follows from the definition of $U$ if we expand $u_1$ as $R+Rs$.
			\end{proof}
			For the rest of this section, we will use directly corollary \ref{cor9}; this means that every time we have a power of $s$ at the end of an element, we may ignore it, as we did for the powers of $t$ in the beginning of the elements. In order to remind that to the reader, we put again a parenthesis around the part of the element we consider.
			\begin{thm}$H_{G_9}=U$.
				\label{thm9}
			\end{thm}
			\begin{proof}                       
				By propositions \ref{pr9} and \ref{prr9}, we only have to prove that  $z^kst^{-2}st\in U$, for every $k\in\{0,\dots,7\}$,
				However, by lemma \ref{lem9}(v) we have that  $z^kst^{-2}st\in U+u_2(z^{k+2}u_1u_2u_1)$, for every $k\in\{1,\dots, 5\}$. Therefore,  by proposition \ref{prr9}, we only the cases where $k\in\{0\}\cup\{6,7\}$.
				\begin{itemize}[leftmargin=*]
					\item \underline{$k=0$}: 
					
					$\hspace*{-0.6cm}\small{\begin{array}[t]{lcl}
						st^{-2}st&\in&s(R+Rt+Rt^2+Rt^3)st\\
						&\in& \underline{u_1u_2}+R(st)^3t^{-1}s^{-1}+Rst(ts)^3s^{-1}t^{-1}s^{-1}+Rst^2(ts)^3s^{-1}t^{-1}s^{-1}\\
						&\in&U+\underline{zu_2u_1}+zst(R+Rs)t^{-1}u_1+zst^2(R+Rs)t^{-1}u_1\\
						&\in&U+\underline{zu_1}+z(st)^3t^{-1}s^{-1}t^{-2}u_1+\underline{\underline{zu_1u_2u_1}}+zst^2s(R+Rt+Rt^2+Rt^3)u_1\\
						&\in&U+\underline{(z^2u_2u_1u_2)u_1}+\underline{\underline{zu_1u_2u_1}}+\underline{\underline{(zu_1t^2u_1t)u_1}}+zst(ts)^3s^{-1}t^{-1}s^{-1}tu_1+
						zst(ts)^3s^{-1}t^{-1}s^{-1}t^2u_1\\
						&\in&U+z^2st(R+Rs)t^{-1}s^{-1}tu_1+z^2st(R+Rs)t^{-1}s^{-1}t^2u_1\\
						&\in&U+\underline{z^2u_2u_1}+z^2(st)^3t^{-1}s^{-1}t^{-2}s^{-1}tu_1+z^2(st)^3t^{-1}s^{-1}t^{-2}s^{-1}t^2u_1\\
						&\in&U+\underline{\underline{t^{-1}(z^3u_1u_2u_1t)u_1}}+z^3t^{-1}s^{-1}t^{-2}s^{-1}t^2u_1
						
						\end{array}}$
					\\\\
					It remains to prove that the element $z^3t^{-1}s^{-1}t^{-2}s^{-1}t^2$ is inside $U$. For this purpose, we expand $t^{-2}$ as a linear combination of 1, $t^{-1}$, $t$ and $t^2$ and we have:
					\\\\
					$\small{\begin{array}{lcl}
						z^3t^{-1}s^{-1}t^{-2}s^{-1}t^2	&\in&z^3t^{-1}s^{-1}(R+Rt^{-1}+Rt+Rt^2)s^{-1}t^2\\
						&\in&\underline{z^3u_2u_1u_2}+Rz^3(st)^{-3}st^3+
						Rz^3t^{-1}(R+Rs)t(R+Rs)t^2+\\&&+
						Rz^3t^{-1}(R+Rs)t^2(R+Rs)t^2\\
						&\in&U+\underline{\underline{u_2\big((z^2+z^3)u_1u_2\big)}}+	Rz^3t^{-1}stst^2+Rz^3t^{-1}st^2st^2\\
						&\in&U+Rz^3t^{-2}(ts)^3s^{-1}t+Rz^3t^{-1}st(ts)^3s^{-1}t^{-1}s^{-1}t\\
						&\in&U+\underline{\underline{t^{-2}(z^4u_1u_2)}}+rz^4t^{-1}st(R+Rs)t^{-1}s^{-1}t\\
						&\in&U+\underline{z^4u_1}+Rz^4t^{-2}(ts)^3s^{-1}t^{-2}s^{-1}t\\
						&\in&U+z^5u_2(u_1u_2u_1t).
						\end{array}}$
					\\\\
					However, by lemma 	\ref{lem9}(v) we have that $z^5u_2(u_1u_2u_1t)\subset u_2U+
					\underline{(z^7u_2u_1u_2)u_1}$. The result follows from remark \ref{rem9}.\\
					\item \underline{$k\in\{6,7\}$}: 
					
					$\hspace*{-0.7cm}\small{\begin{array}[t]{lcl}
						z^kst^{-2}st&\in&z^k(R+Rs^{-1})t^{-2}(R+Rs^{-1})t\\
						&\in&\underline{z^ku_2u_1u_2}+Rz^{k}s^{-1}t^{-2}s^{-1}t\\
						&\in&U+Rz^kt(st)^{-3}stst^{-1}s^{-1}t\\
						&\in&U+Rz^{k-1}tst(R+Rs^{-1})t^{-1}s^{-1}t\\
						&\in&U+\underline{z^{k-1}u_2}+
						Rz^{k-1}tst^2(st)^{-3}st^2\\
						&\in&U+Rz^{k-2}ts(R+Rt+Rt^{-1}+Rt^{-2})st^2\\
						&\in&U+\underline{z^{k-2}u_2u_1u_2}+Rz^{k-2}(ts)^3s^{-1}t+Rz^{k-2}t(R+Rs^{-1})t^{-1}(R+Rs^{-1})t^2+\\&&+Rz^{k-2}t(R+Rs^{-1})t^{-2}(R+Rs^{-1})t^2\\
						&\in&U+\underline{(z^{k-1}+z^{k-2})u_2u_1u_2}+Rz^{k-2}ts^{-1}t^{-1}s^{-1}t^2+Rz^{k-2}ts^{-1}t^{-2}s^{-1}t^2\\
						&\in&U+Rz^{k-2}t^2(st)^{-3}st^3+Rz^{k-2}t^2(st)^{-3}stst^{-1}s^{-1}t^2\\
						&\in&U+\underline{z^{k-3}u_2u_1u_2}+z^{k-3}u_2st(R+Rs^{-1})t^{-1}s^{-1}t^2\\
						&\in&U+\underline{z^{k-3}u_2}+z^{k-3}u_2st^2(st)^{-3}st^3\\
						&\in&U+z^{k-4}u_2st^2s(R+Rt+Rt^2+Rt^{-1})\\
						&\in&U+\underline{\underline{u_2(z^{k-4}u_1u_2u_1)}}+\underline{\underline{z^{k-4}u_2(u_1t^2u_1t)}}+z^{k-4}u_2st(ts)^3s^{-1}t^{-1}s^{-1}t+z^{k-4}u_2st^2(R+Rs^{-1})t^{-1}\\
						&\in&U+z^{k-3}u_2st(R+Rs)t^{-1}s^{-1}t+\underline{z^{k-4}u_2u_1u_2}+z^{k-4}u_2st^3(st)^{-3}sts\\
						&\in&U+\underline{z^{k-3}u_2}+z^{k-3}u_2(ts)^3s^{-1}t^{-2}s^{-1}t+\underline{\underline{u_2(z^{k-5}u_1u_2u_1t)s}}\\
						&\in&U+z^{k-2}u_2(u_1u_2u_1t).
						\end{array}}$\\\\
					Again, by lemma \ref{lem9}(v) we have that $z^{k-2}u_2(u_1u_2u_1t)\subset u_2U+
					\underline{(z^ku_2u_1u_2)u_1}$. The result follows from remark \ref{rem9}.
					\qedhere
				\end{itemize}
			\end{proof}
			
			\begin{cor}
				The BMR freeness conjecture holds for the generic Hecke algebra $H_{G_9}$.
			\end{cor}
			\begin{proof}
				By theorem \ref{thm9} we have that $H_{G_9}=U=\sum\limits_{k=0}^7z^k(u_2+u_2su_2+u_2st^{-2})$. The result then follows from proposition \ref{BMR PROP}, since $H_{G_9}$ is generated as left $u_2$-module by 48 elements and, hence, as $R$-module by $|G_9|=192$ elements (recall that $u_2$ is generated as $R$-module by 4 elements.)
			\end{proof}
			Let $R=\ZZ[u_{s,i}^{\pm},u_{t,j}^{\pm}]_{\substack{1\leq i\leq 3 \\1\leq j\leq 4}}$ and let $H_{G_{10}}=\langle s,t\;|\; stst=tsts,\prod\limits_{i=1}^{3}(s-u_{s,i})=\prod\limits_{j=1}^{4}(t-u_{t,i})=0\rangle$ be the generic Hecke algebra associated to $G_{10}$. Let $u_1$ be the subalgebra of $H_{G_{10}}$ generated by $s$ and $u_2$ the subalgebra of $H_{G_{10}}$ generated by $t$. We recall that $z:=(st)^2=(ts)^2$  generates the center of the associated complex braid group and that $|Z(G_{10})|=12$. We set $U=\sum\limits_{k=0}^{11}(z^ku_2u_1+z^ku_2st^{-1}+z^ku_2s^{-1}t+z^ku_2s^{-1}ts^{-1}).$
			
			From now on, we will underline the elements that belong to $U$  by definition. Our goal is to prove that $H_{G_{10}}=U$ (theorem \ref{thm10}). 
			Since $1\in U$, it is enough to prove that $U$ is a right-sided ideal of $H_{G_{10}}$ or, equivalently, that $Us$ and $Ut$ are subsets of $U$. 
			For this purpose, we first need to prove some preliminary results. 
			
			In the following lemmas we prove that some subsets of $z^ku_2u_1u_2u_1$, where $k$ belongs in a smaller range of $\{0,\dots,11\}$, are also subsets of $U$.
			\begin{lem} \mbox{}
				\vspace*{-\parsep}
				\vspace*{-\baselineskip}\\
				\begin{itemize}[leftmargin=0.6cm]
					\item [(i)] For every $k\in\{1,\dots, 10\}$, $z^ku_2u_1u_2\subset U$.
					\item[(ii)]For every $k\in\{1,\dots, 11\}$, $z^ku_2st^{-1}s\subset U$.
					\item[(iii)]For every $k\in\{0,\dots,9\}$, $z^ku_2st^2s\subset U$.
					\item [(iv)]For every $k\in\{1,\dots,9\}$, $z^ku_2su_2s\subset U$.
				\end{itemize}
				\label{lem10}
			\end{lem}
			\begin{proof}
				\mbox{}
				\vspace*{-\parsep}
				\vspace*{-\baselineskip}\\
				\begin{itemize}[leftmargin=0.6cm]
					\item[(i)] 
					$\hspace*{-0.2cm}\small{\begin{array}[t]{lcl}
						z^ku_2u_1u_2\phantom{s}&=&z^ku_2(R+Rs+Rs^{-1})u_2\\
						&\subset&\underline{z^ku_2}+z^ku_2s(R+Rt+Rt^{-1}+Rt^2)+z^ku_2s^{-1}(R+Rt+Rt^{-1}+Rt^{-2})\\
						&\subset& U+\underline{z^ku_2u_1}+Rz^ku_2(st)^2t^{-1}s^{-1}+\underline{z^ku_2st^{-1}}+
						z^ku_2(st)^2t^{-1}s^{-1}t+\underline{z^ku_2s^{-1}t}+\\&&+z^ku_2(ts)^{-2}ts+z^ku_2(ts)^{-2}tst^{-1}\\
						&\subset& U+\underline{z^{k+1}u_2u_1}+\underline{z^{k+1}u_2s^{-1}t}+
						\underline{z^{k-1}u_2u_1}+\underline{z^{k-1}u_2st^{-1}}.
						\end{array}}$
					\item[(ii)] 
					$\hspace*{-0.2cm}\small{\begin{array}[t]{lcl}
						z^ku_2st^{-1}s&=&z^ku_2st^{-1}(R+Rs^{-1}+Rs^{-2})\\
						&\subset&\underline{z^ku_2st^{-1}}+z^ku_2s^2(ts)^{-2}t+
						z^ku_2(R+Rs^{-1}+Rs^{-2})t^{-1}s^{-2}\\
						&\subset&U+z^{k-1}u_2(R+Rs+Rs^{-1})t+\underline{z^{k}u_2u_1}+z^ku_2(ts)^{-2}ts^{-1}+z^ku_2s^{-1}(ts)^{-2}ts^{-1}\\
						&\subset& U+\underline{z^{k-1}u_2u_1}+z^{k-1}u_2(st)^2t^{-1}s^{-1}+\underline{z^{k-1}
							u_2s^{-1}t}+\underline{z^{k-1}u_2s^{-1}ts^{-1}}\\
						&\subset&U+\underline{z^ku_2u_1}.
						\end{array}}$
					
					\item[(iii)]$\hspace*{-0.2cm}\small{\begin{array}[t]{lcl}
						z^ku_2st^2s\phantom{s}&=&z^ku_2st(ts)^2s^{-1}t^{-1}\\
						&\subset&U+z^{k+1}u_2st(R+Rs+Rs^{2})t^{-1}\\
						&\subset&U+\underline{z^{k+1}u_2u_1}+z^{k+1}u_2(st)^2t^{-2}+z^{k+1}u_2(st)^2t^{-1}st^{-1}\\
						&\subset&U+\underline{z^{k+2}u_2}+\underline{z^{k+2}u_2st^{-1}}.
						\end{array}}$
					\item [(iv)] $\hspace*{-0.2cm}\small{\begin{array}[t]{lcl}
						z^ku_2su_2s&=&z^ku_2s(R+Rt+Rt^2+Rt^{-1})\\
						&\subset& \underline{z^ku_2u_1}+
						z^ku_2(st)^2t^{-1}s^{-1}+z^ku_2st^2s+z^ku_2st^{-1}s\\
						&\subset& U+\underline{z^{k+1}u_2u_1}+z^ku_2st^2s+z^ku_2st^{-1}s\\
						&\subset& U+z^ku_2st^2s+z^ku_2st^{-1}s.
						\end{array}}$
					\\\\
					The result follows then from (ii) and (iii).
					\qedhere
				\end{itemize}
			\end{proof}
			\begin{lem}
				\mbox{}
				\vspace*{-\parsep}
				\vspace*{-\baselineskip}\\
				\begin{itemize}[leftmargin=0.6cm]
					\item[(i)] For every $k\in\{0,\dots,10\}$, $z^ku_2u_1tu_1\subset U$.
					
					\item[(ii)] For every $k\in\{2,\dots,11\}$, $z^ku_2s^{-1}u_2s^{-1}\subset U$.
					\item[(iii)] For every $k\in\{2,\dots,10\}$, $z^ku_2u_1u_2s^{-1}\subset U$.
					\label{lem210}
				\end{itemize}
			\end{lem}
			\begin{proof}
				\mbox{}
				\vspace*{-\parsep}
				\vspace*{-\baselineskip}\\
				\begin{itemize}[leftmargin=0.6cm]
					\item[(i)] 
					$\hspace*{-0.2cm}\small{\begin{array}[t]{lcl}
						z^ku_2u_1tu_1\phantom{sss}&=&z^ku_2(R+Rs+Rs^{-1})tu_1\\
						&\subset&\underline{z^ku_2u_1}+z^ku_2(st)^2t^{-1}s^{-1}u_1+z^ku_2s^{-1}t(R+Rs+Rs^{-1})\\
						&\subset&U+\underline{z^{k+1}u_2u_1}+\underline{z^ku_2s^{-1}t}+z^ku_2s^{-2}(st)^2t^{-1}+\underline{z^ku_2s^{-1}ts^{-1}}\\
						&\subset&U+z^{k+1}u_2(R+Rs+Rs^{-1})t^{-1}\\
						&\subset&U+\underline{z^{k+1}u_2}+\underline{z^{k+1}u_2st^{-1}}+z^{k+1}u_2(ts)^{-2}ts\\
						&\subset&U+\underline{z^ku_2u_1}.
						\end{array}}$
					
					\item[(ii)]
					$\hspace*{-0.2cm}\small{\begin{array}[t]{lcl}
						z^ku_2s^{-1}u_2s^{-1}&=&z^ku_2s^{-1}(R+Rt+Rt^{-1}+Rt^{-2})s^{-1}\\
						&\subset&\underline{z^ku_2u_1}+\underline{z^ks^{-1}ts^{-1}}+z^ku_2(ts)^{-2}t+z^ku_2(ts)^{-2}tst^{-1}s^{-1}\\
						&\subset&U+\underline{z^{k-1}u_2}+z^{k-1}u_2s(st)^{-2}st\\
						&\subset&U+z^{k-2}u_2s^2t\\
						&\subset&U+z^{k-2}u_2(R+Rs+Rs^{-1})t\\
						&\subset&U+\underline{z^{k-2}u_2}+z^{k-2}u_2(st)^2t^{-1}s^{-1}+\underline{z^{k-2}u_2s^{-1}t}\\
						&\subset& U+\underline{z^{k-1}u_2u_1}.
						\end{array}}$
					\item[(iii)]
					$\hspace*{-0.2cm}\small{\begin{array}[t]{lcl}
						z^ku_2u_1u_2s^{-1}&=&z^ku_2(R+Rs+Rs^{-1})u_2s^{-1}\\
						&\subset&\underline{z^ku_2u_1}+z^ku_2su_2s^{-1}+z^ku_2s^{-1}u_2s^{-1}\\
						&\stackrel{(ii)}{\subset}&U+z^ku_2s(R+Rt+Rt^2+Rt^{-1})s^{-1}\\
						&\subset&U+\underline{z^ku_2}+z^ku_2u_1tu_1+z^ku_2(st)^2t^{-1}s^{-1}ts^{-1}+z^ku_2s(st)^{-2}st\\
						&\stackrel{(i)}{\subset}&U+\underline{z^{k+1}u_2s^{-1}ts^{-1}}+z^{k-1}u_2u_1u_2\
						\end{array}}$
					\\\\
					The result follows from lemma \ref{lem10}(i).
					\qedhere
				\end{itemize}
			\end{proof}
			To make it easier for the reader to follow the calculations, we will double-underline the elements as described in lemmas \ref{lem10} and \ref{lem210} and  we will use directly the fact that these elements are inside $U$. In the following lemma we prove that some subsets of $z^ku_2u_1u_2u_1u_2$, where $k$ belongs in a smaller range of $\{0,\dots,11\}$ are also subsets of $U$.
			
			\begin{lem}
				\mbox{}
				\vspace*{-\parsep}
				\vspace*{-\baselineskip}\\
				\begin{itemize}[leftmargin=0.6cm]
					\item [(i)] For every $k\in\{1,\dots,8\}$, $z^ku_2u_1tu_1t\subset U$.
					\item [(ii)] For every $k\in\{1,\dots,6\}$, $z^ku_2su_2st^2\subset U$.
					\item[(iii)] For every $k\in\{1,\dots,5\}$, $z^ku_2u_1tu_1t^2\subset U$.
					\item[(iv)] For every $k\in\{1,\dots,3\}$, $z^ku_2su_2su_2\subset U$.
					\item[(v)]For every $k\in\{3,\dots,5\}$, $z^ku_2s^{-1}u_2s^{-1}u_2\subset U$.
				\end{itemize}
				\label{stst10}
			\end{lem}
			\begin{proof}
				\mbox{}
				\vspace*{-\parsep}
				\vspace*{-\baselineskip}\\
				\begin{itemize}[leftmargin=0.6cm]
					\item[(i)]
					$\hspace*{-0.2cm}\small{\begin{array}[t]{lcl}
						z^ku_2u_1tu_1t&=&z^ku_2u_1t(R+Rs+Rs^2)t\\
						&\subset&\underline{\underline{z^ku_2u_1u_2}}+z^ku_2u_1(st)^2+z^ku_2(R+Rs+Rs^2)ts^2t\\
						&\subset&U+\underline{z^{k+1}u_2u_1}+
						\underline{\underline{z^ku_2u_1u_2}}+z^ku_2(st)^2t^{-1}st+z^ku_2s(st)^2t^{-1}st\\
						&\subset&U+\underline{\underline{z^{k+1}u_2u_1u_2}}+z^{k+1}u_2s(R+Rt+Rt^2+Rt^3)st\\
						&\subset&U+\underline{\underline{z^{k+1}u_2u_1u_2}}+z^{k+1}u_2(st)^2+
						z^{k+1}u_2st(ts)^2s^{-1}+z^{k+1}u_2st^2(ts)^2s^{-1}\\
						&\subset&U+\underline{z^{k+2}u_2}+z^{k+2}u_2(st)^2t^{-1}s^{-2}+z^{k+2}u_2(st)^2t^{-1}s^{-1}ts^{-1}\\
						&\subset&U+\underline{z^{k+3}u_2u_1}+\underline{z^{k+3}u_2s^{-1}ts^{-1}}.
						
						\end{array}}$
					\item[(ii)]
					$\hspace*{-0.2cm}\small{\begin{array}[t]{lcl}
						z^ku_2su_2st^2&=&z^ku_2s(R+Rt+Rt^2+Rt^3)st^2\\
						&\subset&\underline{\underline{z^ku_2u_1u_2}}+z^ku_2(st)^2t+z^ku_2st(ts)^2s^{-1}t+z^ku_2st^2(ts)^2s^{-1}t\\
						&\subset&U+\underline{z^{k+1}u_2}+z^{k+1}u_2u_1tu_1t+z^{k+1}u_2(st)^2t^{-1}s^{-1}ts^{-1}t\\
						&\subset&U+(z^{k+1}+z^{k+2})u_2u_1tu_1t.
						\end{array}}$
					\\\\
					The result follows from (i).
					\item[(iii)]
					$\hspace*{-0.2cm}\small{\begin{array}[t]{lcl}
						z^ku_2u_1tu_1t^2&=&z^ku_2u_1t(R+Rs+Rs^2)t^2\\
						&\subset&\underline{\underline{z^ku_2u_1u_2}}+z^ku_2u_1(ts)^2s^{-1}t+
						z^ku_2(R+Rs+Rs^2)ts^2t^2\\
						&\subset&U+\underline{\underline{z^{k+1}u_2u_1u_2}}+
						\underline{\underline{z^ku_2u_1u_2}}+z^ku_2(st)^2t^{-1}st^2+z^ku_2s(st)^2t^{-1}st^2\\
						&\subset&U+\underline{\underline{z^{k+1}u_2u_1u_2}}+z^{k+1}u_2su_2st^2.
						\end{array}}$
					\\\\
					The result follows from (ii).
					\item[(iv)]
					$\hspace*{-0.2cm}\small{\begin{array}[t]{lcl}
						z^ku_2su_2su_2&=&z^ku_2s(R+Rt+Rt^2+Rt^3)su_2\\
						&\subset&\underline{\underline{z^ku_2u_1u_2}}+z^ku_2(st)^2u_2+z^ku_2(st)^2t^{-1}s^{-2}(st)^2u_2+
						z^ku_2st^3s(R+Rt+Rt^2+Rt^3)\\
						&\subset&U+\underline{z^{k+1}u_2}+\underline{\underline{z^{k+2}u_2u_1u_2}}+
						\underline{\underline{z^ku_2su_2s}}+z^ku_2st^2(ts)^2s^{-1}+z^ku_2su_2st^2+\\&&+
						z^ku_2st^2(ts)^2s^{-1}t^2\\
						&\stackrel{(ii)}{\subset}&U+z^{k+1}u_2(st)^2t^{-1}s^{-1}ts^{-1}+z^{k+1}u_2(st)^2t^{-1}s^{-1}ts^{-1}t^2\\
						&\subset&U+\underline{z^{k+2}u_2s^{-1}ts^{-1}}+z^{k+2}u_2u_1tu_1t^2.
						\end{array}}$
					\\\\
					The result follows from (iii).
					\item[(v)] $z^ku_2s^{-1}u_2s^{-1}u_2=
					z^ku_2(ts)^{-1}tsu_2(ts)^{-1}tsu_2\subset z^{k-2}u_2su_2su_2$.
					The result follows from (iv).
					\qedhere
				\end{itemize}
			\end{proof}
			To make it easier for the reader to follow the calculations, we will also double-underline the elements as described in lemmas \ref{lem10} and \ref{lem210} and  we will use directly the fact that these elements are inside $U$.
			The following lemma helps us to prove that $Uu_1\subset U$ (see proposition \ref{Us10}).
			\begin{lem} For every $k\in\{8,9\}$, $z^ks^{-1}u_2s^{-1}u_2s^{-1}\subset U$.
				\label{ll10}
			\end{lem}
			\begin{proof} We expand $\bold{u_2}$ as $R+Rt^{-1}+Rt^{-2}+Rt^{-3}$ and we have:\\\\
				$\hspace*{-0.2cm}\small{\begin{array}[t]{lcl}
					z^ks^{-1}\bold{u_2}s^{-1}u_2s^{-1}
					&\subset&\underline{\underline{z^ku_2u_1u_2s^{-1}}}+z^k(ts)^{-2}u_2s^{-1}+
					z^k(ts)^{-2}tst^{-1}s^{-1}u_2s^{-1}+
					z^k(ts)^{-2}tst^{-2}s^{-1}u_2s^{-1}\\
					&\subset&U+\underline{z^{k-1}u_2u_1}+z^{k-1}u_2s^2(ts)^{-2}u_2s^{-1}+
					z^{k-1}u_2st^{-1}(st)^{-2}su_2s^{-1}\\
					&\subset&U+\underline{\underline{z^{k-2}u_2u_1u_2s^{-1}}}+z^{k-2}u_2st^{-1}s(R+Rt+Rt^{-1}+Rt^{-2})s^{-1}\\
					&\subset&U+\underline{z^{k-2}u_2st^{-1}}+z^{k-2}u_2st^{-1}(st)^2t^{-1}s^{-2}+z^{k-2}u_2st^{-1}s(st)^{-2}st+\\&&+
					z^{k-2}u_2st^{-1}(R+Rs^{-1}+Rs^{-2})t^{-2}s^{-1}\\
					&\subset&U+z^{k-1}u_2st^{-2}(R+Rs+Rs^{-1})+z^{k-3}u_2st^{-1}s^2t+\underline{\underline{z^{k-2}u_2u_1u_2s^{-1}}}+\\&&+
					z^{k-2}u_2s(st)^{-2}st^{-1}s^{-1}+z^{k-2}u_2s(st)^{-2}sts^{-1}t^{-2}s^{-1}\\
					&\subset&U+\underline{z^{k-1}u_2u_1}+\underline{\underline{z^{k-1}u_2su_2s}}+\underline{\underline{z^{k-1}u_2u_1u_2s^{-1}}}+
					z^{k-3}u_2st^{-1}(R+Rs+Rs^{-1})t+\\&&+\underline{\underline{z^{k-3}u_2u_1u_2s^{-1}}}+
					z^{k-3}u_2s^2t(ts)^{-2}tst^{-1}s^{-1}\\
					&\subset&U+\underline{z^{k-3}u_2u_1}+z^{k-3}u_2st^{-1}(st)^2t^{-1}s^{-1}+
					z^{k-3}u_2s(st)^{-2}st^2+\\&&+
					z^{k-4}u_2s^2t^2(R+Rs^{-1}+Rs^{-2})t^{-1}s^{-1}\\
					&\subset&U+\underline{\underline{z^{k-2}u_2u_1u_2s^{-1}}}+\underline{\underline{z^{k-4}u_2u_1u_2}}+\underline{\underline{z^{k-4}u_2u_1u_2s^{-1}}}+
					z^{k-4}u_2s^2t^2(ts)^{-2}t+\\&&+z^{k-4}u_2s^2t^2s^{-1}(ts)^{-2}t\\
					&\subset&U+\underline{\underline{z^{k-5}u_2u_1u_2}}+z^{k-5}u_2s^2t^3(st)^{-2}st^2\\
					&\subset&U+z^{k-6}u_2(R+Rs+Rs^{-1})t^3st^2\\
					&\subset&U+\underline{\underline{z^{k-6}u_2u_1u_2}}+\underline{\underline{z^{k-6}u_2su_2st^2}}+z^{k-6}u_2(ts)^{-2}tst^4st^2\\
					&\subset&U+\underline{\underline{z^{k-7}u_2su_2st^2}}.
					\end{array}}$
			\end{proof}
			\begin{prop} $Uu_1\subset U$.
				\label{Us10}
			\end{prop}
			\begin{proof}
				Since $u_1=R+Rs+Rs^2$, it is enough to prove that $Us\subset U$. By the definition of $U$ we need to prove that for every $k\in\{0,\dots,11\}$, $z^ku_2st^{-1}s$ and $z^ku_2s^{-1}ts$ are subsets of $U$. 
				However, by lemma \ref{lem10}(ii) it will be sufficient to prove that $u_2st^{-1}s\subset U$. We have:
				$$\small{\begin{array}{lcl}
					u_2st^{-1}s&=&u_2s(R+Rt+Rt^2+Rt^3)s\\
					&\subset&\underline{u_2u_1}+u_2(ts)^2+\underline{\underline{u_2st^2s}}+u_2st^2(ts)^2s^{-1}t^{-1}\\
					&\subset&U+\underline{zu_2}+zu_2st^2(R+Rs+Rs^2)t^{-1}\\
					&\subset&U+\underline{\underline{u_2u_1u_2}}+zu_2(ts)^2s^{-2}(st)^2t^{-2}+zu_2st(ts)^2s^{-1}t^{-1}st^{-1}\\
					&\subset&U+\underline{\underline{z^3u_2u_1u_2}}+z^2u_2st(R+Rs+Rs^2)t^{-1}st^{-1}\\
					&\subset&U+\underline{\underline{z^2u_2u_1u_2}}+z^2u_2(st)^2t^{-2}st^{-1}+
					z^2u_2(ts)^2st^{-1}st^{-1}\\
					&\subset&U+\underline{z^3u_2st^{-1}}+\underline{\underline{z^3u_2su_2su_2}}.
					\end{array}}$$
				It remains to prove that for every $k\in\{0,\dots,11\}$, $z^ku_2s^{-1}ts\subset U$. For $k\not=11$, the result is obvious since 
				$z^ku_2s^{-1}ts\subset z^ku_2(R+Rs+Rs^2)ts\subset \underline{z^ku_2u_1}+z^ku_2(st)^2t^{-1}+z^ku_2s(st)^2t^{-1}\subset U+\underline{z^{k+1}u_2}+\underline{z^{k+1}u_2st^{-1}}.$ Therefore, we only have to prove that $z^{11}u_2s^{-1}ts\subset U.$ \\\\
				$\small{\begin{array}{lcl}
					z^{11}u_2s^{-1}ts&\subset&z^{11}u_2s^{-1}t(R+Rs^{-1}+Rs^{-2})\\
					&\subset&\underline{z^{11}u_2s^{-1}t}+\underline{z^{11}u_2s^{-1}ts^{-1}}+z^{11}u_2s^{-1}(R+Rt^{-1}+Rt^{-2}+Rt^{-3})s^{-2}\\
					&\subset&U+\underline{z^{11}u_2u_1}+z^{11}u_2(ts)^{-2}ts^{-1}+
					z^{11}u_2(ts)^{-2}tst^{-1}s^{-2}+	z^{11}u_2(ts)^{-2}tst^{-2}s^{-2}\\
					&\subset&U+\underline{z^{10}u_2u_1}+z^{10}u_2s(st)^{-2}sts^{-1}+
					z^{10}u_2(R+Rs^{-1}+Rs^{-2})t^{-2}s^{-2}\\
					&\subset&U+\underline{\underline{z^9u_2u_1u_2s^{-1}}}+\underline{z^{10}u_2u_1}+z^{10}u_2(ts)^{-2}ts(st)^{-2}sts^{-1}+
					z^{10}u_2s^{-1}(ts)^{-2}tst^{-1}s^{-2}\\
					&\subset&U+\underline{\underline{z^8u_2u_1u_2s^{-1}}}+z^9u_2s^{-1}t(R+Rs^{-1}+Rs^{-2})t^{-1}s^{-2}\\
					&\subset&U+\underline{z^9u_2u_1}+z^9u_2s^{-1}t(ts)^{-2}ts^{-1}+z^9u_2s^{-1}ts^{-1}(ts)^{-2}ts^{-1}\\
					&\subset&U+\underline{\underline{z^8s^{-1}u_2s^{-1}}}+z^8u_2s^{-1}u_2s^{-1}u_2s^{-1}
					\end{array}}$
				\\\\
				The result follows from lemma \ref{ll10}.
			\end{proof}
			For the rest of this section, we will use directly proposition  \ref{Us10}; this means that every time we have a power of $s$ at the end of an element, we may ignore it. In order to remind that to the reader, we put  a parenthesis around the part of the element we consider.
			\begin{thm} $H_{G_{10}}=U$.
				\label{thm10}
			\end{thm}
			\begin{proof}
				Since $1\in U$, it will be sufficient to prove that $U$ is a right-sided ideal of $H_{G_{10}}$. For this purpose one may check that $Us^{-1}$ and $Ut^{-1}$ are subsets of $U$. By proposition \ref{Us10} it is enough to prove that $Ut^{-1}\subset U$. By the definition of $U$ we have that 
				$$Ut^{-1}\subset\sum\limits_{k=0}^{11}(z^ku_2u_1t^{-1}+z^ku_2st^{-2}+\underline{z^ku_2s^{-1}}+z^ku_2s^{-1}ts^{-1}t^{-1}).$$
				As a result, we have to prove that for $k\in\{0,\dots,11\}$, $z^ku_2u_1t^{-1}$, $z^ku_2st^{-2}$ and $z^ku_2s^{-1}ts^{-1}t^{-1}$ are subsets of $U$. We distinguish the following cases:\\
				\begin{itemize}[leftmargin=0.6cm]
					\item[C1.] \underline{The case of $z^ku_2u_1t^{-1}$}: 
					\begin{itemize}[leftmargin=-0.07cm]
						\item \underline{$k\not=0$}: We expand $u_1$ as $R+Rs+Rs^{-1}$ and we have that 
						$z^ku_2u_1t^{-1}\subset
						\underline{z^ku_2}+\underline{z^ku_2st^{-1}}+z^ku_2(ts)^{-2}ts\subset
						U+\underline{z^{k-1}u_2u_1}.$
						
						\item \underline {$k=0$}: We expand $u_1$ as $R+Rs+Rs^{-1}$ and we have that
						$z^ku_2u_1t^{-1}\subset\underline{u_2}+\underline{u_2st^{-1}}+u_2s^2t^{-1}$.
						Hence, it remains to prove that $u_2u_2s^2t^{-1}\subset U$. For this purpose, we expand $t^{-1}$ as a linear combination of 1, $t$, $t^2$ and $t^3$ and we have:  $u_2u_2s^2t^{-1}\subset\underline{u_2u_1}+u_2s(st)^2t^{-1}s^{-1}+u_2s(st)^2t^{-1}s^{-1}t+u_2s(st)^2t^{-1}s^{-1}t^2\subset U+\underline{(zu_2st^{-1})s^{-1}}+zu_2st^{-1}s^{-1}t+
						zu_2st^{-1}s^{-1}t^2$.
						Therefore, we have to prove that $zu_2st^{-1}s^{-1}t+
						zu_2st^{-1}s^{-1}t^2\subset U$. We have:
						$zu_2st^{-1}s^{-1}t+
						zu_2st^{-1}s^{-1}t^2\subset zu_2st^{-1}(R+Rs+Rs^2)t+
						zu_2st^{-1}(R+Rs+Rs^2)t^2\subset 	\underline{zu_2s}+zu_2st^{-1}(st)^2t^{-1}s^{-1}+zu_2st^{-1}s(st)^2t^{-1}s^{-1}+\underline{\underline{zu_2u_1u_2}}+\underline{\underline{zu_2su_2st^2}}+zu_2st^{-1}s^2t^2\subset U+\underline{\underline{(z^2u_2u_1u_2)s^{-1}}}+\underline{\underline{(z^2u_2su_2su_2)s^{-1}}}+zu_2st^{-1}s^2t^2
						$. It remains to prove that $zu_2st^{-1}s^2t^2\subset U$. We have:
						\\ \\
						$\hspace*{-0.2cm}\small{\begin{array}[t]{lcl}
							zu_2st^{-1}s^2t^2	&\subset&zu_2s(R+Rt+Rt^2+Rt^3)s^2t^2\\
							&\subset&\underline{\underline{zu_2u_1u_2}}+zu_2(st)^2t^{-1}st^2+
							zu_2(st)^2t^{-1}s^{-1}ts^2t^2+zu_2(st)^2t^{-1}s^{-1}t^2s^2t^2\\
							&\subset&U+\underline{\underline{z^2u_2u_1u_2}}+z^2u_2(R+Rs+Rs^2)ts^2t^2+
							z^2u_2(R+Rs+Rs^2)t^2s^2t^2\\
							&\subset&U+\underline{\underline{z^2u_2u_1u_2}}+z^2u_2(st)^2t^{-1}st^2+
							z^2u_2s(st)^2t^{-1}st^2+z^2u_2(st)^2t^{-1}s^{-1}ts^2t^2+\\&&+
							z^2u_2s(st)^2t^{-1}s^{-1}ts^2t^2\\
							&\subset&U+\underline{\underline{z^3u_2u_1u_2}}+\underline{\underline{z^3u_2su_2st^2}}+
							z^3u_2s^{-1}t(R+Rs+Rs^{-1})t^2+\\&&+z^3u_2st^{-1}(R+Rs+Rs^2)ts^2t^2\\
							&\subset&\underline{\underline{z^3u_2u_1u_2}}+z^3u_2s^{-1}(ts)^2s^{-1}t+\underline{\underline{z^3u_2s^{-1}u_2s^{-1}u_2}}+z^3u_2st^{-1}(st)^2t^{-1}st^2+\\&&+z^3u_2st^{-1}s(st)^2t^{-1}st^2\\
							&\subset&U+\underline{\underline{z^4u_2u_1u_2}}+\underline{\underline{z^4u_2su_2st^2}}+z^4u_2st^{-1}st^{-2}(ts)^2s^{-1}t\\
							&\subset&U+z^5u_2st^{-1}st^{-2}s^{-1}t.
							\end{array}}$\\ \\
						We need to prove that $z^5u_2st^{-1}st^{-2}s^{-1}t$ is a subset of $U$. We have that $z^5u_2st^{-1}st^{-2}s^{-1}t\subset z^5u_2st^{-1}s(R+Rt+Rt^{-1}+Rt^2)s^{-1}t \subset U+\underline{z^5u_2u_1}+z^5u_2st^{-1}(st)^2t^{-1}s^{-2}t+z^5u_2st^{-1}s(st)^{-2}st^2+z^5u_2st^{-1}st^2s^{-1}t\subset U+z^6u_2st^{-2}s^{-2}t+z^4u_2st^{-1}s^2t^2+z^5u_2st^{-1}st^2s^{-1}t. $\\\\
						However, $z^6u_2st^{-2}s^{-2}t\subset z^6st^{-2}(R+Rs+Rs^{-1})t\subset \underline{\underline{z^6u_2u_1u_2}}+z^6u_2st^{-2}(st)^2t^{-1}s^{-1}+z^6u_2st^{-1}(ts)^{-2}st^2 \subset U+ \underline{\underline{(z^7u_2u_1u_2)u_1}}+\underline{\underline{z^5u_2su_2st^2}}. $
						Moreover, we have that $z^4u_2st^{-1}s^2t^2\subset z^4u_2s(st)^{-2}sts^3t^2\subset \underline{\underline{z^3u_2u_1tu_1t^2}}.$ It remains to prove that 
						$z^5u_2st^{-1}st^2s^{-1}t\subset U$.  We have:\\\\
						$\hspace*{-0.2cm}\small{\begin{array}[t]{lcl}
							z^5u_2st^{-1}st^2s^{-1}t&\subset& z^5u_2st^{-1}st^2(R+Rs+Rs^2)t\\
							&\subset& z^5u_2st^{-1}st^3+z^5u_2st^{-1}st^2st+z^5u_2st^{-1}st^2s^2t\\
							&\subset& z^5u_2(R+Rs^{-1}+Rs^{-2})t^{-1}(R+Rs^{-1}+Rs^{-2})t^3+z^5u_2st^{-2}(ts)^2s^{-2}(st)^2+\\&&+z^5u_2st^{-1}st^2s^2t\\
							&\subset& \underline{\underline{z^5u_2u_1u_2}}+z^5u_2s^{-1}t^{-1}u_1u_2+z^5u_2s^{-2}t^{-1}s^{-1}t^3+z^5u_2s^{-2}t^{-1}s^{-2}t^3+\\&&+\underline{\underline{(z^7u_2u_1u_2)}}s^{-2}+z^5u_2st^{-1}st^2s^2t\\
							&\subset& U+z^5u_2(ts)^{-2}u_1u_2+z^5u_2s^{-1}(ts)^{-2}t^4+z^5u_2s^{-1}(st)^{-2}t^2(st)^{-2}st^4+\\&&+z^5u_2st^{-1}st^2s^2t\\
							&\subset& U+\underline{\underline{z^4u_2u_1u_2}}+z^3u_2(st)^{-2}st^3st^4+z^5u_2st^{-1}st^2s^2t\\
							&\subset& U+\underline{\underline{z^2u_2su_2su_2}}+z^5u_2st^{-1}st^2s^2t.
							\end{array}}$\\\\
						
						Therefore, in order to finish the proof of this case, it remains to prove that $z^5u_2st^{-1}st^2s^2t\subset U$.
						
						$\hspace*{-0.2cm}\small{\begin{array}[t]{lcl}
							z^5u_2st^{-1}st^2s^{2}t	
							
							&\subset&
							z^5u_2s(R+Rt+Rt^2+Rt^3)st^2s^2t
							\end{array}}$\\
						$\hspace*{-0.2cm}\small{\begin{array}[t]{lcl}
							\phantom{z^5u_2st^{-1}st^2s^{2}t}
							&\subset&
							z^5u_2s^2t^2s^2t+z^5u_2(st)^2ts^2t+z^5u_2(ts)^2s^{-1}(ts)^2s^{-2}(st)^2t^{-2}(ts)^2s^{-1}+\\&&+

							z^5u_2st^2(ts)^2s^{-1}ts^2t\\

							&\subset&z^5u_2s^2t^2(R+Rs+Rs^{-1})t+\underline{\underline{((z^6+z^9)u_2u_1u_2}})u_1+z^6u_2st^2(R+Rs+Rs^2)ts^2t
							\end{array}}$
						\\
						$\hspace*{-0.2cm}\small{\begin{array}[t]{lcl}
							\phantom{z^5u_2st^{-1}st^2s^{2}t}
							&\subset&U+
							\underline{\underline{z^5u_2u_1u_2}}+z^5u_2s^2t(ts)^2s^{-1}+z^5u_2(R+Rs+Rs^{-1})t^2s^{-1}t+
							
							z^6u_2st^3s^2t+\\&&+z^6u_2(st)^2t^{-1}s^{-1}(ts)^2st+
							z^6u_2(st)^2t^{-1}s^{-1}ts(st)^2t^{-1}(st)^2t^{-1}s^{-1}\\

							&\subset&U+\underline{\underline{((z^5+z^6+z^8)u_2u_1u_2)}}u_1+
							z^5u_2st^2s^{-1}t+\underline{\underline{z^5u_2s^{-1}u_2s^{-1}u_2}}+
							\\&&+
							z^6u_2st^3(R+Rs+Rs^{-1})t+
							z^9u_2s^{-1}(ts)^2s^{-1}t^{-3}s^{-1}\\

							&\subset&U+
							
							z^5u_2(st)^2t^{-1}s^{-1}ts^{-1}t+
							\underline{\underline{z^6u_2u_1u_2}}+z^6u_2st^2(ts)^2s^{-1}+
							z^6u_2st^3(ts)^{-2}tst^2+\\&&+\underline{\underline{(z^{10}u_2u_1u_2)s^{-1}}}\\
							&\subset&U+z^6u_2s^{-1}t(R+Rs+Rs^2)t+				
							
							\underline{\underline{(z^7u_2u_1u_2)s^{-1}}}+
							\underline{\underline{z^5u_2su_2st^2}}\\
							
							&\subset&U+\underline{z^6u_2u_1}+z^6u_2s^{-1}(ts)^2s^{-1}+
							z^6u_2s^{-1}(ts)^2s^{-1}t^{-1}(st)^2t^{-1}s^{-1}\\
							&\subset&U+\underline{z^7u_2u_1}+\underline{\underline{(z^8u_2u_1u_2)s^{-1}}}.\\
							\phantom{z^5u_2st^{-1}st^2s^{2}t}&&
							\end{array}}$

					\end{itemize}
					\item [C2.] \underline{The case of $z^ku_2st^{-2}$}:
					\\
					For $k\not=11$, we expand $t^{-2}$ as a linear combination of 1, $t$, $t^{-1}$ and $t^{2}$ and  we have that 
					$z^ku_2st^{-2}\subset \underline{z^ku_2u_1}+z^ku_2(st)^2t^{-1}s^{-1}+\underline{z^ku_2st^{-1}}+\underline{\underline{(z^ku_2st^2s)s^{-1}}}\subset U+\underline{z^{k+1}u_2u_1}\subset U$.
					It remains to prove that $z^{11}u_2st^{-2}\subset U$. We have:
					\\\\
					$\hspace*{-0.2cm}\small{\begin{array}{lcl}
						z^{11}u_2st^{-2}&\subset&z^{11}u_2(R+Rs^{-1}+Rs^{-2})t^{-2}\\
						&\subset&\underline{z^{11}u_2}+z^{11}u_2(ts)^{-2}tst^{-1}+z^{11}u_2s^{-1}(ts)^{-2}tst^{-1}\\
						&\subset&U+\underline{z^{10}u_2st^{-1}}+z^{10}u_2s^{-1}t(R+Rs^{-1}+Rs^{-2})t^{-1}\\
						&\subset&U+\underline{z^{10}u_2u_1}+z^{10}u_2s^{-1}t(ts)^{-2}ts+z^{10}u_2s^{-1}ts^{-1}(ts)^{-2}ts\\
						&\subset&U+\underline{\underline{(z^9u_2u_1u_2)s}}+(z^9u_2s^{-1}u_2s^{-1}u_2s^{-1})s^2.
						\end{array}}$
					\\\\
					The result follows from lemma \ref{ll10}.
					\\ 
					\item [C3.] \underline{The case of $z^ku_2s^{-1}ts^{-1}t^{-1}$}:
					\\
					For $k\not \in\{0,1\}$, we have that $z^ku_2s^{-1}ts^{-1}t^{-1}=z^ku_2s^{-1}t(ts)^{-2}ts=\underline{\underline{(z^{k-1}u_2u_1u_2)s}}\subset U.$ It remains to prove the case where $k\in\{0,1\}$. We have:
					\\\\
					$\hspace*{-0.2cm}\small{\begin{array}{lcl}
						z^ku_2s^{-1}ts^{-1}t^{-1}&\subset&z^ku_2s^{-1}t(R+Rs+Rs^2)t^{-1}\\
						&\subset&\underline{z^ku_2u_1}+z^ku_2s^{-1}(ts)^2s^{-1}t^{-2}+z^ku_2(R+Rs+Rs^2)ts^2t^{-1}\\
						&\subset&U+\underline{\underline{z^{k+1}u_2u_1u_2}}+z^ku_2u_1t^{-1}+z^ku_2(st)^2t^{-1}st^{-1}+z^ku_2s(st)^2t^{-1}st^{-1}\\
						&\stackrel{C1}{\subset}&U+\underline{z^{k+1}u_2st^{-1}}+\underline{\underline{z^{k+1}u_2su_2su_2}}.\phantom{z^ku_2u_1t^{-1}+z^ku_2(st)^2t^{-1}st^{-1}+++++}
						\qedhere
						\end{array}}$
					
				\end{itemize}
			\end{proof}
			\begin{cor}
				The BMR freeness conjecture holds for the generic Hecke algebra $H_{G_{10}}$.
			\end{cor}
			\begin{proof}
				By theorem \ref{thm10} we have that $H_{G_{10}}=U$. The result follows from proposition \ref{BMR PROP}, since by definition $U$ is generated as left $u_2$-module by 28 elements and, hence, as $R$-module by $|G_{10}|=288$ elements (recall that $u_2$ is generated as $R$-module by 4 elements).
			\end{proof}
			\subsubsection{\textbf{The case of }$\mathbf{G_{11}}$}
			Let $R=\ZZ[u_{s,i}^{\pm},u_{t,j}^{\pm},u_{u,l}^{\pm}]$, where $1\leq i\leq 2$, $1\leq j\leq 3$  and $1\leq l\leq 4$. We also let $$H_{G_{11}}=\langle s,t,u\;|\; stu=tus=ust,\prod\limits_{i=1}^{2}(s-u_{s,i})=\prod\limits_{j=1}^{3}(t-u_{t,j})=\prod\limits_{l=1}^{4}(u-u_{u,l})=0\rangle$$ be the generic Hecke algebra associated to $G_{11}$. Let $u_1$ be the subalgebra of $H_{G_{11}}$ generated by $s$, $u_2$ the subalgebra of $H_{G_{11}}$ generated by $t$ and $u_3$ the subalgebra of $H_{G_{11}}$ generated by $u$. We recall that $z:=stu=tus=ust$  generates the center of the associated complex braid group and that $|Z(G_{11})|=24$.
			We set $U=\sum\limits_{k=0}^{23}(z^ku_3u_2+z^ku_3tu^{-1}u_2).$
			By the definition of $U$, we have the following remark.
			
			\begin{rem}
				$Uu_2 \subset U$.
				\label{rem11}
			\end{rem}
			From now on, we will underline the elements that by definition belong to $U$.  Moreover, we will use directly the remark \ref{rem11}; this means that every time we have a power of $t$ at the end of an element, we may ignore it. 
			To remind that to the reader, we put a parenthesis around the part of the element we consider.

			Our goal is to prove that $H_{G_{11}}=U$ (theorem \ref{thm11}). Since $1\in U$, it will be sufficient to prove that $U$ is a left-sided ideal of $H_{G_{11}}$. For this purpose, one may check that $sU$, $tU$ and $uU$ are subsets of $U$. The following proposition states that it is enough to prove $tU\subset U$.
			\begin{prop}
				If $tU\subset U$ then $H_{G_{11}}=U$.
				\label{Ut11}
			\end{prop}
			\begin{proof}
				As we explained above, we have to prove that $sU$, $tU$ and $uU$ are subsets of $U$. However, by the definition of $U$ we have $uU\subset U$ and, hence, by hypothesis we
				only have to prove that $sU\subset U$. We recall that $z=stu$, therefore $s=zu^{-1}t^{-1}$ and $s^{-1}=z^{-1}tu$. We notice that $$U=
				\sum\limits_{k=0}^{22}z^k(u_3u_2+u_3tu^{-1}u_2)+z^{23}(u_3u_2+u_3tu^{-1}u_2).$$
				Hence, we have:\\ \\
				$\hspace*{-0.2cm}\small{\begin{array}[t]{lcl}sU&\subset&
					\sum\limits_{k=0}^{22}z^ks(u_3u_2+u_3tu^{-1}u_2)+z^{23}s(u_3u_2+u_3tu^{-1}u_2)\\
					&\subset& \sum\limits_{k=0}^{22}z^{k+1}u^{-1}t^{-1}(u_3u_2+u_3tu^{-1}u_2)+z^{23}(R+Rs^{-1})(u_3u_2
					+u_3tu^{-1}u_2)\\
					&\subset& \sum\limits_{k=0}^{22}u^{-1}t^{-1}(\underline{z^{k+1}u_3u_2}+\underline{z^{k+1}u_3tu^{-1}u_2})+
					\underline{z^{23}u_3u_2}+\underline{z^{23}u_3tu^{-1}u_2}+z^{23}s^{-1}u_3u_2+
					z^{23}s^{-1}u_3tu^{-1}u_2\\
					&\subset&u_3u_2U+z^{22}tu_3u_2+z^{22}tu_3tu^{-1}u_2\\
					&\subset&u_3u_2U+
					t(\underline{z^{22}u_3u_2}+\underline{z^{22}u_3t^{-1}u_2})\\
					&\subset&u_3u_2U.
					\end{array}}$\\\\
				By hypothesis, $tU\subset U$ and, hence, $u_2U\subset U$, since $u_2=R+Rt+Rt^2$. Moreover, $u_3U\subset U$, by the definition of $U$. Therefore, $sU\subset U$.
			\end{proof}
			
			\begin{cor}
				If $z^ktu_3$ and $z^ktu_3tu^{-1}$ are subsets of $U$ for every $k\in\{0,\dots,23\}$, then $H_{G_{11}}=U$.
				\label{corr11}
			\end{cor}
			\begin{proof}
				
				The result follows directly from the definition of $U$, proposition \ref{Ut11} and remark \ref{rem11}.
			\end{proof}
			
			As a first step we will prove the conditions of corollary \ref{corr11} for some shorter range of the values of $k$, as we can see in proposition  \ref{tu11} and corollary \ref{tuts111}.
			
			\begin{prop}
				\mbox{}
				\vspace*{-\parsep}
				\vspace*{-\baselineskip}\\
				\begin{itemize}[leftmargin=0.6cm]
					\item[(i)] For every $k\in\{0,\dots,21\}$, $z^ku_3tu\subset U$.
					\item[(ii)] For every $k\in\{0,\dots,19\}$, $z^ku_3tu^2\subset U$.
					\item[(iii)] For every $k\in\{0,\dots,19\}$, $z^ku_3tu_3\subset U$.
					\item[(iv)] For every $k\in\{2,\dots,19\}$, $z^ku_3u_2u_3\subset U$.
				\end{itemize}
				\label{tu11}
			\end{prop}
			\begin{proof}
				Since $u_3=R+Ru+Ru^{-1}+Ru^2$, (iii) follows from (i) and (ii) and the definition of $U$. Moreover, (iv) follows directly from (iii), since:
				$$\begin{array}{lcl}
				z^ku_3u_2u_3&=&z^ku_3(R+Rt+Rt^{-1})u_3\\
				&\subset& \underline{z^ku_3}+z^ku_3tu_3+z^ku_3(t^{-1}s^{-1}u^{-1})usu_3\\
				&\subset& U+z^ku_3tu_3+
				z^{k-1}u_3(R+Rs^{-1})u_3\\
				&\subset& U+z^ku_3tu_3+\underline{z^{k-1}u_3}+z^{k-1}u_3(s^{-1}u^{-1}t^{-1})tu_3\\
				&\subset& U+(z^k+z^{k-2})tu_3.\end{array}$$
				Therefore, it is enough to prove (i) and (ii). 
				For every $k\in\{0,\dots,21\}$ we have $z^ku_3tu=z^ku_3(tus)s^{-1}\subset z^{k+1}u_3(R+Rs)\subset \underline{z^{k+1}u_3}+z^{k+1}u_3(ust)t^{-1}\subset U+\underline{z^{k+2}u_3t^{-1}}\subset U$
				and, hence, we prove (i).
				
				For (ii), we notice that, for every $k\in\{0,\dots,19\}$, $z^ku_3tu^2=z^ku_3(tus)s^{-1}u\subset z^{k+1}u_3(R+Rs)u\subset \underline{z^{k+1}u_3}+z^{k+1}u_3su\subset U+z^{k+1}u_3(ust)t^{-1}u\subset U+z^{k+2}u_3t^{-1}u$. However, if we expand $t^{-1}$ as a linear combination of 1, $t$ and $t^2$ we have that $z^{k+2}u_3t^{-1}u\subset
				\underline{z^{k+2}u_3}+z^{k+2}u_3tu+z^{k+2}u_3t(tus)s^{-1}\stackrel{(i)}{\subset}U+z^{k+3}u_3t(R+Rs)\subset U+\underline{z^{k+3}u_3t}+z^{k+3}u_3t(stu)u^{-1}t^{-1}\subset
				U+\underline{z^{k+4}u_3tu^{-1}u_2}$.
			\end{proof}
			
			\begin{lem}
				\mbox{}
				\vspace*{-\parsep}
				\vspace*{-\baselineskip}\\
				\begin{itemize}[leftmargin=0.6cm]
					\item[(i)] For every $k\in\{1,\dots,22\}$, $z^ku_3u_2u_1\subset U$.
					\item[(ii)] For every $k\in\{0,\dots,18\}$, $z^ku_3tu_3u_1\subset U$.
					\item[(iii)] For every $k\in\{3,\dots,20\}$, $z^ku_3tu_1u_3\subset U$.
				\end{itemize}
				\label{tsu11}
			\end{lem}
			\begin{proof}
				\mbox{}
				\vspace*{-\parsep}
				\vspace*{-\baselineskip}\\
				\begin{itemize}[leftmargin=0.6cm]
					\item[(i)]
					$\hspace*{-0.2cm}\begin{array}[t]{lcl}
					z^ku_3u_2u_1&=&z^ku_3u_2(R+Rs)\\
					&\subset& \underline{z^ku_3u_2}+z^ku_3(R+Rt^{-1}+Rt)s\\
					&\subset& z^ku_3(ust)t^{-1}+z^ku_3t^{-1}s+z^ku_3t(stu)u^{-1}t^{-1}\\
					& \subset& U+\underline{z^{k+1}u_3u_2}+z^ku_3t^{-1}(R+Rs^{-1})+\underline{z^{k+1}u_3tu^{-1}u_2}\\
					&\subset& U+\underline{z^{k}u_3u_2}+z^ku_3(u^{-1}t^{-1}s^{-1})\\
					&\subset& U+\underline{z^{k-1}u_3}.
					\end{array}$
					\item[(ii)]
					$z^ku_3tu_3u_1=z^ku_3tu_3(R+Rs)\subset z^ku_3tu_3+z^ku_3tu_3(ust)t^{-1}
					\subset \big(z^ku_3tu_3+z^{k+1}u_3tu_3\big)u_2$. The result follows from proposition \ref{tu11}(iii) .
					\item[(iii)]	
					$z^ku_3tu_1u_3=z^ku_3t(R+Rs^{-1})u_3=z^ku_3tu_3+z^ku_3t^2(t^{-1}s^{-1}u^{-1})u_3\stackrel{\ref{tu11}(iii)}{\subset}U+z^{k-1}u_3u_2u_3$.
					The result follows directly from proposition \ref{tu11}(iv).
					\qedhere
				\end{itemize}
			\end{proof}
			In order to make it easier for the reader to follow the calculations, from now on we will double-underline the elements as described in proposition \ref{tu11} and in lemma \ref{tsu11} and we will use directly the fact that these elements are inside $U$.

			\begin{prop}
				\mbox{}
				\vspace*{-\parsep}
				\vspace*{-\baselineskip}\\
				\begin{itemize}[leftmargin=0.6cm]
					\item[(i)] For every $k\in\{2,\dots, 23\}$, $z^kt^2u^{-1}\in U$.	
					\item[(ii)]For every $k\in\{0,\dots,21\}$, $z^ktutu^{-1} \in U$.
					\item[(iii)]For every $k\in\{0,\dots,15\}$, $z^ktu^2tu^{-1} \in U$.
					\item[(iv)]For every $k\in\{6,\dots,23\}$, $z^ktu^{-1}tu^{-1} \in U+z^ku_3tu_3$.
					Therefore, for every $k\in\{6, \dots, 19\}$, $z^ktu^{-1}tu^{-1}\in U$.
					\item[(v)]For every $k\in\{0,\dots, 5\}$, $z^ktu^3tu^{-1} \in U$.
					\item[(vi)]For every $k\in\{16,\dots,23\}$, $z^ktu^{-2}tu^{-1} \in U+(z^k+z^{k-1}+z^{k-2})u_3tu_3u_2$.
					Therefore, for every $k\in\{16, \dots, 19\}$, $z^ktu^{-2}tu^{-1}\in U$.
					
				\end{itemize}
				\label{tuts11}
			\end{prop}
			\begin{proof}
				\mbox{}
				\vspace*{-\parsep}
				\vspace*{-\baselineskip}\\
				\begin{itemize}[leftmargin=0.6cm]
					\item[(i)]$\hspace*{-0.2cm}\small{\begin{array}[t]{lcl}
						z^kt^2u^{-1}\phantom{us}&\in& z^k(R+Rt+Rt^{-1})u^{-1}\\
						&\in& \underline{z^ku_3}+\underline{z^ku_3tu^{-1}}+z^ku_3(u^{-1}t^{-1}s^{-1})su^{-1}\\
						&\in& U+z^{k-1}u_3(R+Rs^{-1})u^{-1}\\
						&\in& U+\underline{z^{k-1}u_3}+z^{k-1}u_3(s^{-1}u^{-1}t^{-1})t\\
						&\in& U+\underline{z^{k-2}u_3u_2}.
						\end{array}}$\\
					
					\item[(ii)] $\hspace*{-0.2cm}\small{\begin{array}[t]{lcl}
						z^ktutu^{-1}\phantom{s}&=&z^k(tus)s^{-1}tu^{-1}\\
						&\in &z^{k+1}(R+Rs)tu^{-1}\\
						&\in&
						\underline{z^{k+1}u_3tu^{-1}}+z^{k+1}u_3(stu)u^{-2}\\
						&\in& U+\underline{z^{k+2}u_3}\subset U.
						\end{array}}$\\
					\item[(iii)] 
					$\hspace*{-0.2cm}\small{\begin{array}[t]{lcl}
						z^ktu^2tu^{-1}&=&z^k(tus)s^{-1}utu^{-1}
						\\
						&\in& z^{k+1}(R+Rs)utu^{-1}\\
						&\in& \underline{z^{k+1}u_3tu^{-1}}+z^{k+1}u_3(ust)t^{-1}utu^{-1}\\
						
						&\in& U+z^{k+2}u_3(R+Rt+Rt^2)utu^{-1}\\
						&\in& U+\underline{z^{k+2}u_3tu^{-1}}+z^{k+2}u_3tutu^{-1}+z^{k+2}u_3t(tus)s^{-1}tu^{-1}\\
						
						&\stackrel{(ii)}{\in}&U+z^{k+3}u_3t(R+Rs)tu^{-1}\\
						&\in& U+\underline{\underline{z^{k+3}u_3u_2u_3}}+z^{k+3}u_3t(stu)u^{-2}\\
						&\in& U+\underline{\underline{z^{k+4}u_3tu_3}}.
						\end{array}}$\\
					
					\item[(iv)] 
					$\hspace*{-0.2cm}\small{\begin{array}[t]{lcl}
						z^ktu^{-1}tu^{-1}&\in& z^ktu^{-1}(R+Rt^{-1}+Rt^{-2})u^{-1}
						\end{array}}$\\
					$\hspace*{-0.2cm}\small{\begin{array}[t]{lcl}
						\phantom{z^ktu^{-1}tu^{-1}}	&\in& z^ku_3tu_3+z^kt(u^{-1}t^{-1}s^{-1})su^{-1}+z^kt(u^{-1}t^{-1}s^{-1})st^{-1}u^{-1}\\
						&\in& z^ku_3tu_3+z^{k-1}t(R+Rs^{-1})u^{-1}+z^{k-1}tsu(u^{-1}t^{-1}s^{-1})su^{-1}\\
						&\in& z^ku_3tu_3+\underline{z^{k-1}u_3tu^{-1}}+z^{k-1}u_3t^2(t^{-1}s^{-1}u^{-1})+
						z^{k-2}tsu(R+Rs^{-1})u^{-1}\\
						&\in& U+z^ku_3tu_3+\underline{z^{k-2}u_3u_2}+z^{k-2}u_3t(stu)u^{-1}t^{-1}+
						z^{k-2}u_3tsu(s^{-1}u^{-1}t^{-1})t\\
						&\in&U+ z^ku_3tu_3+\underline{z^{k-1}u_3tu^{-1}u_2}+\underline{\underline{(z^{k-3}u_3tu_1u_3)u_2}}\\
						&\in& U+z^ku_3tu_3.
						\end{array}}$
					\\\\
					The result follows from proposition \ref{tu11}(iii).\\
					
					\item[(v)] 
					$\hspace*{-0.2cm}\small{\begin{array}[t]{lcl}
						z^ktu^3tu^{-1}&=&z^k(tus)s^{-1}u^2tu^{-1}\\
						&\in&z^{k+1}(R+Rs)u^2tu^{-1}\\
						&\in&\underline{z^{k+1}u_3tu^{-1}}+z^{k+1}u_3(ust)t^{-1}u^2tu^{-1}\\
						&\in&U+z^{k+2}u_3(R+Rt+Rt^2)u^2tu^{-1}\\
						&\in&U+\underline{z^{k+2}u_3tu^{-1}}+z^{k+2}u_3tu^2tu^{-1}+z^{k+2}u_3t(tus)s^{-1}utu^{-1}\\
						&\stackrel{(iii)}{\in}&U+z^{k+3}u_3t(R+Rs)utu^{-1}\\
						&\in&U+z^{k+3}u_3tutu^{-1}+z^{k+3}u_3tu^{-1}(ust)t^{-2}(tus)s^{-1}tu^{-1}\\
						&\stackrel{(ii)}{\in}&U+z^{k+5}u_3tu^{-1}t^{-2}(R+Rs)tu^{-1}\\
						&\in&U+z^{k+5}u_3t(u^{-1}t^{-1}s^{-1})su^{-1}+z^{k+5}u_3tu^{-1}t^{-2}(stu)u^{-2}\\ 
						&\in&U+\underline{\underline{z^{k+4}u_3tu_1u_3}}+
						z^{k+6}u_3tu^{-1}t^{-2}u^{-2}.
						\end{array}}$
					\\\\
					We expand $t^{-2}$ as a linear combination of 1, $t^{-1}$ and $t$ and we have:\\\\
					$\hspace*{-0.2cm}\small{\begin{array}{lcl}
						z^{k+6}u_3tu^{-1}t^{-2}u^{-2}	&\in&\underline{\underline{z^{k+6}u_3tu_3}}+z^{k+6}u_3t(u^{-1}t^{-1}s^{-1})su^{-2}+
						z^{k+6}u_3tu^{-1}t(R+Ru+Ru^{-1}+Ru^{2})\\ 
						&\in&U+\underline{\underline{z^{k+5}u_3tu_1u_3}}+\underline{z^{k+6}u_3tu^{-1}u_2}+
						z^{k+6}u_3tu^{-1}(tus)s^{-1}+z^{k+6}u_3tu^{-1}tu^{-1}+\\&&+
						z^{k+6}u_3tu^{-1}(tus)s^{-1}(ust)t^{-1}s^{-1} \\ 
						&\stackrel{(iv)}{\in}&U+\underline{\underline{z^{k+7}u_3tu_3u_1}}+z^{k+8}u_3tu^{-1}(R+Rs)t^{-1}s^{-1}\\
						&\in& U+z^{k+8}u_3t(u^{-1}t^{-1}s^{-1})+z^{k+8}u_3tu^{-2}(ust)t^{-2}s^{-1}\\
						&\in&U+\underline{z^{k+7}u_3u_2}+z^{k+9}u_3tu^{-2}(R+Rt^{-1}+Rt)s^{-1}\\
						&\in&U+\underline{\underline{z^{k+9}u_3tu_3u_1}}+z^{k+9}u_3tu^{-1}(u^{-1}t^{-1}s^{-1})+z^{k+9}u_3tu^{-2}t(R+Rs)\\
						&\in&U+\underline{z^{k+8}u_3tu^{-1}}+\underline{\underline{(z^{k+9}u_3tu_3)u_2}}
						+z^{k+9}u_3tu^{-2}t(stu)u^{-1}t^{-1}\\
						&\in& U+(z^{k+10}u_3tu^{-2}tu^{-1})u_2.
						\end{array}}$
					\\\\
					The result follows from (i), (ii), (iii) and (iv), if we expand $u^{-2}$ as a linear combination of 1, $u$, $u^2$ and $u^{-1}$.
					\item[(vi)] 
					$\small{\begin{array}[t]{lcl}
						z^ktu^{-2}tu^{-1}&\in& z^ktu^{-2}(R+Rt^{-1}+Rt^{-2})u^{-1}\\
						&\in& z^ku_3tu_3u_2+z^ku_3tu^{-1}(u^{-1}t^{-1}s^{-1})su^{-1}+z^ku_3tu^{-1}(u^{-1}t^{-1}s^{-1})st^{-1}u^{-1}\\
						&\in&z^ku_3tu_3u_2+z^{k-1}u_3tu^{-1}(R+Rs^{-1})u^{-1}+z^{k-1}u_3tu^{-1}s(t^{-1}s^{-1}u^{-1})usu^{-1}\\
						&\in&(z^k+z^{k-1})u_3tu_3u_2+z^{k-1}u_3tu^{-1}(s^{-1}u^{-1}t^{-1})t+
						z^{k-2}u_3tu^{-1}su(R+Rs^{-1})u^{-1}\\
						&\in&(z^k+z^{k-1})u_3tu_3u_2+\underline{z^{k-2}u_3tu^{-1}u_2}+z^{k-2}u_3tu^{-1}(stu)u^{-1}t^{-1}+\\&&+z^{k-2}u_3tu^{-1}su(s^{-1}u^{-1}t^{-1})t\\
						&\in&U+(z^k+z^{k-1})u_3tu_3u_2+\underline{z^{k-1}u_3}+z^{k-3}u_3tu^{-1}(R+Rs^{-1})ut\\
						&\in&U+(z^k+z^{k-1})u_3tu_3u_2+\underline{z^{k-3}u_3u_2}+z^{k-3}u_3tu^{-1}(s^{-1}u^{-1}t^{-1})tu^2t\\
						&\in&U+(z^k+z^{k-1})u_3tu_3u_2+z^{k-4}u_3tu^{-1}t(R+Ru+Ru^{-1}+Ru^{-2})t\\
						&\in&U+(z^k+z^{k-1})u_3tu_3u_2+\underline{z^{k-4}u_3tu^{-1}u_2}+
						z^{k-4}u_3tu^{-1}(tus)s^{-1}t+
						\\&&+(z^{k-4}u_3
						tu^{-1}tu^{-1})t+z^{k-4}u_3tu^{-1}(R+Rt^{-1}+Rt^{-2})u^{-2}t
						\end{array}}$
					\\
					$\hspace*{-0.2cm}\small{\begin{array}[t]{lcl}
						\phantom{z^ktu^{-2}tu^{-1}}
						&\stackrel{(iv)}{\in}&U+(z^k+z^{k-1})u_3tu_3u_2+z^{k-3}u_3tu^{-1}(R+Rs)t+
						\underline{\underline{(z^{k-4}u_3tu_3)u_2}}+\\&&+
						z^{k-4}u_3t(u^{-1}t^{-1}s^{-1})su^{-2}t+z^{k-4}u_3t(u^{-1}t^{-1}s^{-1})st^{-1}u^{-2}t
						\\
						&\in&U+(z^k+z^{k-1})u_3tu_3u_2+\underline{z^{k-3}u_3tu^{-1}u_2}+z^{k-3}u_3tu^{-2}(ust)+\\&&+
						z^{k-5}u_3t(R+Rs^{-1})u^{-2}t+z^{k-5}u_3ts(t^{-1}s^{-1}u^{-1})usu^{-2}t
						\\
						&\in&U+(z^k+z^{k-1}+z^{k-2})u_3tu_3u_2+\underline{\underline{(z^{k-5}u_3tu_3)u_2}}+
						z^{k-5}u_3t(s^{-1}u^{-1}t^{-1})tu^{-1}t+\\&&+z^{k-6}u_3t(R+Rs^{-1})usu^{-2}t\\
						&\in&U+(z^k+z^{k-1}+z^{k-2})u_3tu_3u_2+\underline{\underline{z^{k-6}u_3u_2u_3}}+z^{k-6}u_3(tus)u^{-2}t+\\&&+
						z^{k-6}u_3ts^{-1}u(R+Rs^{-1})u^{-2}t
						\\
						&\in&U+(z^k+z^{k-1}+z^{k-2})u_3tu_3u_2+\underline{z^{k-5}u_3u_2}+
						z^{k-6}u_3t(s^{-1}u^{-1}t^{-1})t^2+\\&&+
						z^{k-6}u_3t^2(t^{-1}s^{-1}u^{-1})u^2(s^{-1}u^{-1}t^{-1})tu^{-1}t
						\end{array}}$
					\\
					$\hspace*{-0.2cm}\small{\begin{array}[t]{lcl}
						\phantom{z^ktu^{-2}tu^{-1}}
						&\in&U+(z^k+z^{k-1}+z^{k-2})u_3tu_3u_2+\underline{z^{k-7}u_3u_2}+
						z^{k-8}u_3(R+Rt+Rt^{-1})u^2tu^{-1}t\\
						&\in&U+(z^k+z^{k-1}+z^{k-2})u_3tu_3u_2+\underline{z^{k-8}u_3tu^{-1}u_2}+
						(z^{k-8}u_3tu^2tu^{-1})t+\\&&+
						z^{k-8}u_3(u^{-1}t^{-1}s^{-1})su^2tu^{-1}t\\
						&\stackrel{(iii)}{\in}&U+(z^k+z^{k-1}+z^{k-2})u_3tu_3u_2+z^{k-9}u_3(R+Rs^{-1})u^2tu^{-1}t\\
						&\in&U+(z^k+z^{k-1}+z^{k-2})u_3tu_3u_2+\underline{z^{k-9}u_3tu^{-1}u_2}+
						z^{k-9}u_3(s^{-1}u^{-1}t^{-1})tu^3tu^{-1}t\\
						&\in&U+(z^k+z^{k-1}+z^{k-2})u_3tu_3u_2+(z^{k-10}u_3tu^3tu^{-1})t.
						
						\end{array}}$
					\\\\
					However, if we expand $u^3$ as a linear combination of 1, $u$, $u^2$ and $u^{-1}$, we can use (i), (ii), (iii) and (iv) and we have that, for every $k\in\{16,\dots 23\}$, $z^{k-10}u_3tu^3tu^{-1}\subset U$.
					Therefore, for every $k\in\{16,\dots, 23\}$, $z^ktu^{-2}tu^{-1}\in U+(z^k+z^{k-1}+z^{k-2})u_3tu_3u_2$. 
					Moreover, by proposition \ref{tu11}(iii),  we have that  $\big((z^k+z^{k-1}+z^{k-2})u_3tu_3\big)u_2\subset U$, for every $k\in\{16,\dots, 19\}$ and, hence, $z^ktu^{-2}tu^{-1}\in U$ for every $k\in\{16,\dots, 19\}$.
					\qedhere
				\end{itemize}
			\end{proof}
			\begin{cor}
				\mbox{}
				\vspace*{-\parsep}
				\vspace*{-\baselineskip}\\
				\begin{itemize}[leftmargin=0.6cm]
					\item[(i)] For every $k\in\{2, \dots, 19\}$, $z^ktu_3tu^{-1}\in U$.
					\item[(ii)] For every $k\in\{1, \dots, 23\}$, $z^ku_2u^mtu^{-1}\subset U+z^ku_3tu^mtu^{-1}+z^{k-2}u_3tu^{m+1}tu^{-1}$, where $m\in \mathbb{Z}$. Therefore, for every $k\in\{4,\dots, 19\}$, $z^ku_3u_2u_3tu^{-1}\subset U$.
				\end{itemize}
				\label{tuts111}
			\end{cor}
			
			\begin{proof}
				We first prove (i). We use different definitions of $u_3$ and we have:
				\begin{itemize}[leftmargin=*]
					\item For $k\in\{2,\dots,5\}$, we write $u_3=R+Ru+Ru^2+Ru^3$. The result then follows from proposition \ref{tuts11} (i), (ii), (iii) and (v).
					\item For $k\in\{6,\dots,15\}$, we write $u_3=R+Ru+Ru^2+Ru^{-1}$. The result then follows from proposition \ref{tuts11} (i), (ii), (iii) and (iv).
					\item For $k\in\{16,\dots,19\}$, we write $u_3=R+Ru+Ru^{-2}+Ru^{-1}$. The result then follows from proposition \ref{tuts11} (i), (ii), (iv) and (vi).
				\end{itemize}
				For the first part of (ii) we have:
				\\
				$\small{\begin{array}[t]{lcl}
					z^ku_2u^mtu^{-1}&=&z^k(R+Rt+Rt^{-1})u^mtu^{-1}\\
					&\subset& \underline{z^ku_3tu^{-1}}+
					z^ku_3tu^mtu^{-1}+z^ku_3(u^{-1}t^{-1}s^{-1})su^mtu^{-1}\\
					&\subset& U+z^{k}u_3tu^mtu^{-1}+z^{k-1}u_3(R+Rs^{-1})u^mtu^{-1}\\
					&\subset& U+z^ku_3tu^mtu^{-1}
					+\underline{z^{k-1}u_3tu^{-1}}+z^{k-1}u_3(s^{-1}u^{-1}t^{-1})tu^{m+1}tu^{-1}\\
					&\subset& U+z^ku_3tu^mtu^{-1}+z^{k-2}u_3tu^{m+1}tu^{-1}.
					\end{array}}$\\\\
				Hence, $z^ku_2u^mtu^{-1}\subset U+u_3(z^k+z^{k-2})tu_3tu^{-1}$. 
				Therefore, for every $k\in\{4,\dots,19\}$ we have that
				$z^ku_3u_2u_3tu^{-1}\subset z^ku_3(R+Rt+Rt^{-1})u_3tu^{-1}\subset u_3U+
				u_3t^{\pm 1}u_3tu^{-1}\subset u_3U+u_3(z^k+z^{k-2})tu_3tu^{-1}\stackrel{(i)}{\subset}u_3U$. The result follows from the definition of $U$.
			\end{proof}
			We now prove a lemma that leads us to the main theorem of this section (theorem \ref{thm11}).
			
			\begin{lem}
				\mbox{}
				\vspace*{-\parsep}
				\vspace*{-\baselineskip}\\
				\begin{itemize}[leftmargin=0.6cm]
					\item[(i)] For every $k\in\{0,\dots, 15\}$, $z^ku_3t^2u^2\subset U$.
					\item[(ii)] For every $k\in\{6,\dots,16\}$, $z^ktu^{-1}tu^{-1}tu^{-1}\in U+z^{k-6}u_3^{\times}t^2u^3t
					.$ 
					\item [(iii)]For every $k\in\{12,\dots,16\}$, 
					$z^ktu^{-1}tu^{2}tu^{-1}\in U+z^{k-12}u_3^{\times}t^2u^3t.$
					\item [(iv)]For every $k\in\{1,\dots,15\}$, 
					$z^ktu^{-1}tu^{2}tu^{-1}\in U+(z^{k+6}+z^{k+7}+z^{k+8})u_3tu_3u_2+z^{k+8}u_3^{\times}tu^{-3}tu^{-1}t.$
				\end{itemize}
				\label{ttuttu}
			\end{lem}
			\begin{proof}
				\mbox{}
				\vspace*{-\parsep}
				\vspace*{-\baselineskip}\\
				\begin{itemize}[leftmargin=0.6cm]
					\item[(i)]	
					$\hspace*{-0.29cm}\small{\begin{array}[t]{lcl}
						z^ku_3t^2u^2&=&z^ku_3t(tus)s^{-1}u\\
						&\subset& z^{k+1}u_3t(R+Rs)u\\
						&\subset& \underline{\underline{z^{k+1}u_3tu}}+z^{k+1}u_3tu^{-1}(ust)t^{-1}u\\
						&\subset&U+z^{k+2}u_3tu^{-1}(R+Rt+Rt^2)u\\
						&\subset&U+\underline{z^{k+2}u_3u_2}+z^{k+2}u_3tu^{-1}(tus)s^{-1}+
						z^{k+2}u_3tu^{-1}t(tus)s^{-1}\\
						&\subset&\underline{\underline{z^{k+3}u_3tu_3u_1}}+z^{k+3}u_3tu^{-1}t(R+Rs)\\
						&\subset&U+\underline{z^{k+3}u_3tu^{-1}u_2}+z^{k+3}u_3tu^{-1}t(stu)u^{-1}t^{-1}\\
						&\in& U+(z^{k+4}u_3tu_3tu^{-1})t.
						\end{array}}$
					\\\\
					The result follows from corollary \ref{tuts111}(i).\\
					\item[(ii)] 
					$\hspace*{-0.29cm}\small{\begin{array}[t]{lcl}
						z^ktu^{-1}tu^{-1}tu^{-1}&\in&
						z^{k}tu^{-1}(R+Rt^{-1}+R^{\times}t^{-2})u^{-1}tu^{-1}\\
						&\in&z^{k}u_3tu_3tu^{-1}+z^{k}u_3t(u^{-1}t^{-1}s^{-1})su^{-1}tu^{-1}+
						z^ku_3^{\times}t(u^{-1}t^{-1}s^{-1})st^{-1}u^{-1}tu^{-1}\\
						&\stackrel{\ref{tuts111}(i)}{\in}&U+z^{k-1}u_3t(R+Rs^{-1})u^{-1}tu^{-1}+z^{k-1}u_3^{\times}ts(t^{-1}s^{-1}u^{-1})usu^{-1}tu^{-1}\\ 
						&\in&U+z^{k-1}u_3tu_3tu^{-1}+z^{k-1}u_3t(s^{-1}u^{-1}t^{-1})t^2u^{-1}+\\&&+
						z^{k-2}u_3^{\times}tsu(R+R^{\times}s^{-1})u^{-1}tu^{-1}\\ 
						&\stackrel{\ref{tuts111}(i)}{\in}&
						U+\underline{\underline{z^{k-2}u_3u_2u_3}}+z^{k-2}u_3t(stu)u^{-2}+z^{k-2}u_3^{\times}tsu(s^{-1}u^{-1}t^{-1})t^2u^{-1}\\ 
						&\in&U+\underline{\underline{z^{k-1}u_3tu_3}}+z^{k-3}u_3^{\times}t(R+R^{\times}s^{-1})ut^2u^{-1}\\ 	
						&\in&U+z^{k-3}u_3(tus)s^{-1}t^2u^{-1}+
						z^{k-3}u_3^{\times}t^2(t^{-1}s^{-1}u^{-1})u^2t^2u^{-1}\\	
						&\in&U+z^{k-2}u_3(R+Rs)t^2u^{-1}+z^{k-4}u_3^{\times}t^2u^2(R+Rt+R^{\times}t^{-1})u^{-1}\\ 
						&\in&U+\underline{\underline{z^{k-2}u_3u_2u_3}}
						+z^{k-2}u_3(ust)tu^{-1}+\underline{\underline{z^{k-4}u_3u_2u_3}}+z^{k-4}u_3u_2u^2tu^{-1}+\\&&+z^{k-4}u_3^{\times}t^2u^2t^{-1}(u^{-1}t^{-1}s^{-1})st	\\ 
						&\stackrel{\ref{tuts111}(ii)}{\in}&U+\underline{z^{k-1}u_3tu^{-1}}+
						z^{k-6}u_3tu^3tu^{-1}+
						z^{k-5}u_3^{\times}t^2u^2t^{-1}(R+R^{\times}s^{-1})t\\
						&\stackrel{\ref{tuts11}(v)}{\in}&U+z^{k-5}u_3t^2u^2+z^{k-5}u_3^{\times}t^2u^3(u^{-1}t^{-1}s^{-1})t\\
						&\in&U+z^{k-5}u_3t^2u^2+z^{k-6}u_3^{\times}t^2u^3t.
						\end{array}}$
					\\ \\
					The result follows from (i).\\
					\item[(iii)] 
					$\hspace*{-0.29cm}\small{\begin{array}[t]{lcl}
						z^ktu^{-1}tu^{2}tu^{-1}&\in&z^ktu^{-1}t(R+Ru+Ru^{-1}+R^{\times}u^{-2})tu^{-1}\\
						&\in& z^ku_3tu^{-1}t^2u^{-1}+z^ku_3tu^{-1}(tus)s^{-1}tu^{-1}+z^ku_3tu^{-1}tu^{-1}tu^{-1}+\\&&+
						z^ku_3^{\times}tu^{-1}tu^{-2}tu^{-1}
						\\ 
						&\stackrel{(ii)}{\in}&z^ku_3tu^{-1}(R+Rt+Rt^{-1})u^{-1}+z^{k+1}u_3tu^{-1}(R+Rs)tu^{-1}+\\&&+\underline{\underline{(z^{k-6}u_3u_2u_3)t}} +z^{k}u_3^{\times}tu^{-1}(R+Rt^{-1}+R^{\times}t^{-2})u^{-2}tu^{-1}\\ 
						&\in& U+\underline{\underline{z^ku_3tu_3}}+(z^k+z^{k+1})u_3tu_3tu^{-1}+z^ku_3t(u^{-1}t^{-1}s^{-1})su^{-1}+\\&&+z^ku_3tu^{-1}(stu)u^{-2}
						+z^{k}u_3t(u^{-1}t^{-1}s^{-1})su^{-2}tu^{-1}+\\&&+
						z^{k}u_3^{\times}t(u^{-1}t^{-1}s^{-1})st^{-1}u^{-2}tu^{-1}\\ 
						&\stackrel{\ref{tuts111}(i)}{\in}&U+\underline{\underline{z^{k-1}u_3tu_1u_3}}+\underline{\underline{z^{k+1}u_3tu_3}}
						+z^{k-1}u_3t(R+Rs^{-1})u^{-2}tu^{-1}+\\&&+z^{k-1}u_3^{\times}tsu(u^{-1}t^{-1}s^{-1})su^{-2}tu^{-1}
						\\
						&\stackrel{\phantom{\ref{tuts111}(ii)}}{\in}&U+z^{k-1}u_3tu_3tu^{-1}+z^{k-1}u_3t(s^{-1}u^{-1}t^{-1})tu^{-1}tu^{-1}+\\&&+
						z^{k-2}u_3^{\times}tsu(R+R^{\times}s^{-1})u^{-2}tu^{-1}\\ 	
						&\stackrel{\ref{tuts111}(i)}{\in}&U+z^{k-2}u_3u_2u_3tu^{-1}+
						z^{k-2}u_3tsu^{-1}tu^{-1}+
						z^{k-2}u_3^{\times}tsu(s^{-1}u^{-1}t^{-1})tu^{-1}tu^{-1}
						\end{array}}$\\
					
					$\hspace*{-0.29cm}\small{\begin{array}[t]{lcl}
						\phantom{z^ktu^{-1}tu^{2}tu^{-1}}
						&\stackrel{\ref{tuts111}(ii)}{\in}&U+z^{k-2}u_3t(R+Rs^{-1})u^{-1}tu^{-1}+z^{k-3}u_3^{\times}t(R+Rs^{-1})utu^{-1}tu^{-1}
						\\
						&\in&U+z^{k-2}u_3tu_3tu^{-1}+z^{k-2}u_3t(s^{-1}u^{-1}t^{-1})t^2u^{-1}+\\&&+
						z^{k-3}u_3(tus)s^{-1}tu^{-1}tu^{-1}+z^{k-3}u_3^{\times}t^2(t^{-1}s^{-1}u^{-1})u^2tu^{-1}tu^{-1}\\ 
						&\stackrel{\ref{tuts111}(i)}{\in}&U+\underline{\underline{z^{k-3}u_3u_2u_3}}+
						z^{k-2}u_3(R+Rs)tu^{-1}tu^{-1}+\\&&+
						z^{k-4}u_3^{\times}(R+Rt+Rt^{-1})u^2tu^{-1}tu^{-1}\\  \
						&\in&U+(z^{k-2}+z^{k-4})u_3u_2u_3tu^{-1}+z^{k-2}u_3(stu)u^{-2}tu^{-1}+z^{k-4}u_3tu^2tu^{-1}tu^{-1}+\\&&+z^{k-4}u_3^{\times}(u^{-1}t^{-1}s^{-1})su^2tu^{-1}tu^{-1}\\
						&\stackrel{\ref{tuts111}(ii)}{\in}&U+\underline{z^{k-1}u_3tu^{-1}}+
						z^{k-4}u_3tu^2tu^{-1}tu^{-1}+
						z^{k-5}u_3^{\times}(R+R^{\times}s^{-1})u^2tu^{-1}tu^{-1}\\ 
						&\in&U+z^{k-4}u_3tu^2tu^{-1}tu^{-1}+z^{k-5}u_3tu_3tu^{-1}+
						z^{k-5}u_3^{\times}(s^{-1}u^{-1}t^{-1})tu^3tu^{-1}tu^{-1}\\ 
						&\stackrel{\ref{tuts111}(i)}{\in}&U+z^{k-4}u_3tu^2tu^{-1}tu^{-1}+
						z^{k-6}u_3^{\times}tu^3tu^{-1}tu^{-1}
						\\
						&\in&U+z^{k-4}u_3tu^2tu^{-1}tu^{-1}+
						z^{k-6}u_3^{\times}t(R+Ru+Ru^2+R^{\times}u^{-1})tu^{-1}tu^{-1}\\
						&\in&U+(z^{k-4}+z^{k-6})u_3tu^2tu^{-1}tu^{-1}+z^{k-6}u_3u_2u_3tu^{-1}+\\&&+ 
						z^{k-6}u_3(tus)s^{-1}tu^{-1}tu^{-1}+z^{k-6}u_3^{\times}tu^{-1}tu^{-1}tu^{-1}\\ 	
						&\stackrel{\ref{tuts111}(ii)}{\in}&U+(z^{k-4}+z^{k-6})u_3(tus)s^{-1}utu^{-1}tu^{-1}+
						z^{k-5}u_3(R+Rs)tu^{-1}tu^{-1}+\\&&+z^{k-6}u_3^{\times}tu^{-1}tu^{-1}tu^{-1}\\
						&\stackrel{(ii)}{\in}&U+(z^{k-3}+z^{k-5})u_3(R+Rs)utu^{-1}tu^{-1}+z^{k-5}u_3tu_3tu^{-1}+\\&&+z^{k-5}u_3(ust)u^{-1}tu^{-1}+z^{k-12}u_3^{\times}t^2u^3t\\
						&\stackrel{\ref{tuts111}(i)}{\in}&U+(z^{k-3}+z^{k-5})u_3tu_3tu^{-1}+
						(z^{k-3}+z^{k-5})u_3(ust)t^{-2}(tus)s^{-1}tu^{-1}tu^{-1}+\\&&+
						\underline{z^{k-4}u_3tu^{-1}}+
						z^{k-12}u_3^{\times}t^2u^3t\\
						&\stackrel{\ref{tuts111}(i)}{\in}&U+(z^{k-1}+z^{k-3})u_3t^{-2}(R+Rs)tu^{-1}tu^{-1}+
						z^{k-12}u_3^{\times}t^2u^3t	\\
						
						&\in&U+(z^{k-1}+z^{k-3})u_3u_2u_3tu^{-1}+
						(z^{k-1}+z^{k-3})u_3t^{-2}(stu)u^{-2}tu^{-1}+\\&&+z^{k-12}u_3^{\times}t^2u^3t\\
						&\in&U+(z^{k-1}+z^{k-3}+z+z^{k-2})u_3u_2u_3tu^{-1}+z^{k-12}u_3^{\times}t^2u^3t\\
						&\stackrel{\ref{tuts111}(ii)}{\in}&U+z^{k-12}u_3^{\times}t^2u^3t.
						\end{array}}$
					\\
					\item[(iv)]	
					$\hspace*{-0.29cm}\small{\begin{array}[t]{lcl}
						z^ktu^{-1}tu^2tu^{-1}&=&z^ktu^{-1}(tus)s^{-1}utu^{-1}\\
						&\in& z^{k+1}tu^{-1}(R+R^{\times}s)utu^{-1}\\
						&\in&\underline{\underline{z^{k+1}u_3u_2u_3}}+z^{k+1}u_3^{\times}tu^{-2}(ust)t^{-2}(tus)s^{-1}tu^{-1}\\
						&\in&U+z^{k+3}u_3^{\times}tu^{-2}t^{-2}(R+R^{\times}s)tu^{-1}\\
						&\in&U+z^{k+3}u_3tu^{-2}t^{-1}(u^{-1}t^{-1}s^{-1})st+z^{k+3}u_3^{\times}tu^{-2}t^{-2}(stu)u^{-2}\\
						&\in&U+
						z^{k+2}u_3tu^{-2}t^{-1}(R+Rs^{-1})t+z^{k+4}u_3^{\times}tu^{-2}(R+Rt^{-1}+R^{\times}t)u^{-2}\\
						&\in&U+\underline{\underline{z^{k+2}u_3tu_3}}+z^{k+2}u_3tu^{-1}(u^{-1}t^{-1}s^{-1})
						t+\underline{\underline{z^{k+4}u_3tu_3}}+\\&&+
						z^{k+4}u_3tu^{-1}(u^{-1}t^{-1}s^{-1})su^{-2}+z^{k+4}u_3^{\times}tu^{-2}tu^{-2}\\
						&\in&U+\underline{z^{k+1}u_3tu^{-1}u_2}+z^{k+3}u_3tu^{-1}(R+Rs^{-1})u^{-2}+\\&&+
						z^{k+4}u_3^{\times}tu^{-2}t(R+Ru+Ru^{-1}+R^{\times}u^{2})\\
						&\in&U+\underline{\underline{z^{k+3}u_3tu_3}}+
						z^{k+3}u_3tu^{-1}(s^{-1}u^{-1}t^{-1})tu^{-1}+\underline{\underline{(z^{k+4}u_3tu_3)t}}+\\&&+
						z^{k+4}u_3tu^{-2}(tus)s^{-1}+z^{k+4}u_3tu_3tu^{-1}+z^{k+4}u_3^{\times}tu^{-2}(tus)s^{-1}u\\
						&\stackrel{\ref{tuts111}(i)}{\in}&U+z^{k+2}u_3tu_3tu^{-1}+z^{k+5}u_3tu^{-2}(R+Rs)+z^{k+5}u_3^{\times}tu^{-2}(R+Rs)u\\
						&\stackrel{\ref{tuts111}(i)}{\in}&U+\underline{\underline{z^{k+5}u_3tu_3}}+z^{k+5}u_3tu^{-2}(stu)u^{-1}t^{-1}+\underline{z^{k+5}u_3tu^{-1}}+\\&&+
						z^{k+5}u_3^{\times}tu^{-3}(ust)t^{-1}u\\
						&\in&U+z^{k+6}u_3tu_3u_2+z^{k+6}u_3^{\times}tu^{-3}(R+Rt+R^{\times}t^2)u\\
						&\in&U+z^{k+6}u_3tu_3u_2+z^{k+6}u_3tu^{-3}(tus)s^{-1}+
						z^{k+6}u_3^{\times}tu^{-3}t(tus)s^{-1}\\
						&\in&U+z^{k+6}u_3tu_3u_2+z^{k+7}u_3tu^{-3}(R+Rs)+
						z^{k+7}u_3^{\times}tu^{-3}t(R+R^{\times}s)\\
						&\in&U+(z^{k+6}+z^{k+7})u_3tu_3u_2+z^{k+7}u_3tu^{-4}(ust)t^{-1}
						+z^{k+7}u_3^{\times}tu^{-3}t(stu)u^{-1}t\\
						&\in&U+(z^{k+6}+z^{k+7}+z^{k+8})u_3tu_3u_2+z^{k+8}u_3^{\times}tu^{-3}tu^{-1}t.\phantom{==========}
						\qedhere
						\end{array}}$
					
				\end{itemize}
			\end{proof}

			\begin{thm}	
				\mbox{}
				\vspace*{-\parsep}
				\vspace*{-\baselineskip}\\
				\begin{itemize}[leftmargin=0.6cm]
					\item[(i)]For every $k\in\{0,\dots,23\}$, $z^ktu_3\subset U$.
					\item[(ii)]For every $k\in\{0,\dots,23\}$, $z^ktu_3tu^{-1}\subset U$.
					\item[(iii)]$H_{G_{11}}=U$.
					\label{thm11}
				\end{itemize}
			\end{thm}
			
			\begin{proof}
				\mbox{}
				\vspace*{-\parsep}
				\vspace*{-\baselineskip}\\
				
				\begin{itemize}[leftmargin=0.6cm]
					\item [(i)] By proposition \ref{tu11} (iii), we have to prove that $z^ktu_3\subset U$, for every $k\in\{20,\dots,23\}$. We use different definitions of $u_3$ and we have:
					\begin{itemize}[leftmargin=*]
						\item $\underline{k\in\{20,21\}}$: $z^ktu_3=z^kt(R+Ru+Ru^{-1}+Ru^{-2})\subset \underline{z^ku_3t}+\underline{\underline{z^ku_3tu}}+\underline{z^ku_3tu^{-1}}+z^ku_3tu^{-2}$. Therefore, $z^ktu_3\subset U+ \mathbf{z^ku_3tu^{-2}}$.
						\item $\underline{k\in\{22,23\}}$: $z^ktu_3=z^kt(R+Ru^{-1}+Ru^{-2}+Ru^{-3})\subset \underline{z^ku_3t}+{\underline{z^ku_3tu^{-1}}}+z^ku_3tu^{-2}+z^ku_3tu^{-3}$. Therefore, $z^ktu_3\subset U+ \mathbf{z^ku_3tu^{-2}+z^ku_3tu^{-3}}$.\\
					\end{itemize}
					As a result, it will be sufficient to prove that, for every $k\in\{20,\dots,23\}$, $z^ku_3tu^{-2}$ is a subset of $U$,  and that, for every $k\in\{22,23\}$,  $z^ku_3tu^{-3}$ is also a subset of $U$.
					We have:
					\\ \\
					$\hspace*{-0.2cm}\small{\begin{array}{lcl}
						z^ku_3tu^{-2}&=&z^ku_3t^2(t^{-1}s^{-1}u^{-1})usu^{-2}\\
						&\subset& z^{k-1}u_3t^2u(R+Rs^{-1})u^{-2}\\
						&\subset&z^{k-1}u_3t^2(u^{-1}t^{-1}s^{-1})st+z^{k-1}u_3t^2u(s^{-1}u^{-1}t^{-1})tu^{-1}\\
						&\subset& U+z^{k-2}u_3t^2(R+Rs^{-1})t+z^{k-2}u_3u_2utu^{-1}\\
						&\stackrel{\ref{tuts111}(ii)}{\subset}&U+\underline{z^{k-2}u_3u_2}+z^{k-2}u_3t^3(t^{-1}s^{-1}u^{-1})ut+z^{k-2}u_3tutu^{-1}+z^{k-4}u_3tu^2tu^{-1}.
						\end{array}}$
					\\ \\
					Therefore, by proposition  \ref{tuts11}(ii) and by corollary \ref{tuts111}(i) we have
					\begin{equation}
					z^ku_3tu^{-2}\subset U,
					\label{tuu111}
					\end{equation}
					for every $k\in\{20,\dots,23\}$.
					Moreover, since $z^ktu_3 \subset U+z^ku_3tu^{-2}$, $k\in\{20,21\}$,  we use proposition \ref{tu11}(iii) and we have that
					\begin{equation}
					z^ku_3tu_3\subset U,
					\label{tuu1111}
					\end{equation}
					for every $k\in\{0,\dots,21\}$.
					We now prove that $z^ku_3tu^{-3}\subset U$, for every $k\in\{22,23\}$.
					We have:\\ \\
					$\hspace*{-0.2cm}\small{\begin{array}[t]{lcl}
						z^ku_3tu^{-3}&\subset&z^ku_3(R+Rt^{-1}+Rt^{-2})u^{-3}\\
						&\subset&\underline{z^ku_3}+z^ku_3(u^{-1}t^{-1}s^{-1})su^{-3}+
						z^ku_3t^{-1}u(u^{-1}t^{-1}s^{-1})su^{-3}\\ 
						&\subset&U+z^{k-1}u_3(R+Rs^{-1})u^{-3}+z^{k-1}u_3t^{-1}u(R+Rs^{-1})u^{-3}\\ 
						&\subset&U+\underline{z^{k-1}u_3}+z^{k-1}u_3(s^{-1}u^{-1}t^{-1})tu^{-2}+ z^{k-1}u_3(u^{-1}t^{-1}s^{-1})su^{-2}+\\&&+
						z^{k-1}u_3(u^{-1}t^{-1}s^{-1})su(s^{-1}u^{-1}t^{-1})tu^{-2}\\ 
						&\subset& U+z^{k-2}u_3tu^{-2}+z^{k-2}u_3(R+Rs^{-1})u^{-2}+
						z^{k-3}u_3(R+Rs^{-1})utu^{-2}\\
						&\stackrel{(\ref{tuu111})}{\subset}&U+
						\underline{z^{k-2}u_3}+z^{k-2}u_3(s^{-1}u^{-1}t^{-1})tu^{-1}+
						z^{k-3}u_3tu_3+z^{k-3}u_3(s^{-1}u^{-1}t^{-1})tu^2tu^{-2}\\ 
						&\stackrel{(\ref{tuu1111})}{\subset}& U+\underline{z^{k-3}u_3tu^{-1}}+
						z^{k-4}u_3tu^2tu^{-2}.
						\end{array}}$
					\\\\
					Therefore, it will be sufficient to prove that $z^{k-4}u_3tu^2tu^{-2}$ is a subset of $U$. For this purpose, we expand $u^2$ as a linear combination of 1, $u$, $u^{-1}$ and $u^{-2}$ and we have:
					\\
					$\hspace*{-0.2cm}\small{\begin{array}{lcl}
						z^{k-4}u_3tu^2tu^{-2}
						&\subset& U+\underline{\underline{z^{k-4}u_3u_2u_3}}+
						z^{k-4}u_3tutu^{-2}+z^{k-4}u_3tu^{-1}tu^{-2}+
						z^{k-4}u_3tu^{-2}tu^{-2}.
						\end{array}}$
					\\
					However, $	z^{k-4}u_3tutu^{-2}=z^{k-4}u_3(tus)s^{-1}tu^{-2}=z^{k-3}u_3s^{-1}tu^{-2}$. If we expand $s^{-1}$ as a linear combination of 1 and $s$ we have that
					$z^{k-3}u_3s^{-1}tu^{-2}\subset  z^{k-3}u_3tu^{-2}+z^{k-3}u_3(stu)u^{-3}=z^{k-3}u_3tu^{-2}+\underline{z^{k-2}u_3}$.
					Therefore, by relation (\ref{tuu1111}) we have that $z^{k-3}u_3s^{-1}tu^{-2}\subset U$ and, hence, $z^{k-4}u_3tutu^{-2}\subset U$. 
					It remains to prove that $C:=z^{k-4}u_3tu^{-1}tu^{-2}+
					z^{k-4}u_3tu^{-2}tu^{-2}$ is a subset of $U$. We have:
					\\\\
					$\hspace*{-0.2cm}\small{\begin{array}{lcl}
						C&=&z^{k-4}u_3tu^{-1}tu^{-2}+
						z^{k-4}u_3tu^{-2}tu^{-2}\\
						&\subset&z^{k-4}u_3t^2(t^{-1}s^{-1}u^{-1})usu^{-1}tu^{-2}+
						z^{k-4}u_3tu^{-2}(R+Rt^{-1}+Rt^{-2})u^{-2}\\
						
						&\stackrel{\phantom{\ref{ttuttu}(iii)}}{\subset}&z^{k-5}u_3t^2u(R+Rs^{-1})u^{-1}tu^{-2}+
						\underline{\underline{z^{k-4}u_3tu_3}}+
						z^{k-4}u_3tu^{-1}(u^{-1}t^{-1}s^{-1})su^{-2}+\\&&+z^{k-4}u_3tu^{-2}
						t^{-1}(t^{-1}s^{-1}u^{-1})usu^{-2}\\

						&\subset& U+\underline{\underline{z^{k-5}u_3u_2u_3}}+z^{k-5}u_3t^2u(s^{-1}u^{-1}t^{-1})t^2u^{-2}+z^{k-5}u_3tu^{-1}(R+Rs^{-1})u^{-2}+\\&&+z^{k-5}u_3tu^{-2}t^{-1}u(R+Rs^{-1})u^{-2}\\ 
						&\subset& U+z^{k-6}u_3t^2ut^2u^{-2}+\underline{\underline{z^{k-5}u_3u_2u_3}}+
						z^{k-5}u_3tu^{-1}(s^{-1}u^{-1}t^{-1})tu^{-1}+\\&&+
						z^{k-5}u_3tu^{-1}(u^{-1}t^{-1}s^{-1})su^{-1}+
						z^{k-5}u_3tu^{-1}(u^{-1}t^{-1}s^{-1})su(s^{-1}u^{-1}t^{-1})tu^{-1}\\ 
						&\subset& U+z^{k-6}u_3t^2u(R+Rt+Rt^{-1})u^{-2}
						+z^{k-6}u_3tu_3tu^{-1}+z^{k-6}u_3tu^{-1}(R+Rs^{-1})u^{-1}+\\&&+
						z^{k-7}u_3tu^{-1}(R+Rs^{-1})utu^{-1}\\ 
						\end{array}}$
					\\
					$\hspace*{-0.2cm}\small{\begin{array}{lcl}
						\phantom{C}

						&\stackrel{\ref{tuts111}(i)}{\subset}&U+
						\underline{\underline{z^{k-6}u_3u_2u_3}}+z^{k-6}u_3t(tus)s^{-1}tu^{-2}+
						z^{k-6}u_3t^2u^2(u^{-1}t^{-1}s^{-1})su^{-2}+\\&&+
						
						\underline{\underline{(z^{k-6}+z^{k-7})u_3u_2u_3}}+z^{k-6}u_3tu^{-1}(s^{-1}u^{-1}t^{-1})t
						+z^{k-7}u_3tu^{-1}(s^{-1}u^{-1}t^{-1})tu^2tu^{-1}\\ 
						&\subset&U+z^{k-5}u_3t(R+Rs)tu^{-2}+z^{k-7}u_3t^2u^2(R+Rs^{-1})u^{-2}+
						\underline{z^{k-7}u_3tu^{-1}u_2}+\\&&+
						z^{k-8}u_3tu^{-1}tu^2tu^{-1}\\ &\stackrel{\ref{ttuttu}(iii)}{\subset}&U+\underline{\underline{z^{k-5}u_3u_2u_3}}+z^{k-5}u_3t(stu)u^{-3}+
						\underline{z^{k-7}u_3u_2}+z^{k-7}u_3t^2u^2(s^{-1}u^{-1}t^{-1})tu^{-1}+\\&&+
						\underline{\underline{(z^{k-20}u_3u_2u_3)t}}
						\\
						&\subset&U+\underline{\underline{z^{k-4}tu_3}}+
						z^{k-8}u_3u_2u_3tu^{-1}.
						\end{array}}$
					\\\\
					The result follows from corollary \ref{tuts111}(ii).\\
					\item[(ii)] By corollary \ref{tuts111}(i), we restrict ourselves to proving the cases where $k\in\{0,1\}\cup\{20,\dots,23\}$.
					We distinguish the following cases:\\
					\begin{itemize}
						\item [C1.] \underline{$k\in\{20,21\}$}: We expand $u_3$ as $R+Ru+Ru^{-1}+Ru^{-2}$ and by proposition \ref{tuts11}(i), (ii), (iv) and (v) we have that
						$z^ktu_3tu^{-1}\subset U+(z^k+z^{k-1}+z^{k-2})u_3tu_3$. The result follows from (i).\\
						\item [C2.]\underline{$k\in\{22,23\}$}:  We expand $u_3$ as $R+Ru^{-1}+Ru^{-2}+Ru^{-3}$ and by proposition \ref{tuts11}(i), (iv) and (v) we have that $z^ktu_3tu^{-1}\subset U+(z^k+z^{k-1}+z^{k-2})u_3tu_3+z^ktu^{-3}tu^{-1}\stackrel{(i)}{\subset}U+z^ktu^{-3}tu^{-1}$.  However, since $k-8\in\{1,\dots,15\}$, we can apply lemma \ref{ttuttu}(iv) and we have that
						$z^ktu^{-3}tu^{-1}\in
						u_3^{\times}Ut^{-1}+(z^k+z^{k-1}+z^{k-2})u_3tu_3u_2+z^{k-8}u_3^{\times}tu^{-1}tu^2tu^{-1}t^{-1}$.
						However, by (i) and by lemma \ref{ttuttu}(iii) we also have that
						$z^ktu^{-3}tu^{-1}\in u_3^{\times}Uu_2+\underline{\underline{(z^{k-20}u_3u_2u_3)t}}$.
						The result then follows from the definition of $U$ and remark \ref{rem11}.\\
						\item [C3.]\underline{$k\in\{0,1\}$}: We expand $u_3$ as $R+Ru+Ru^2+Ru^{3}$ and by proposition \ref{tuts11}(ii), (iii) and (v) we have that
						$z^ktu_3tu^{-1}\subset U+z^kt^2u^{-1}\subset U+z^kt^2(R+Ru+Ru^2+Ru^3)\subset U+\underline{z^ku_2}+z^ku_3t(tus)s^{-1}+z^ku_3t^2u^2+z^ku_3t^2u^3 \stackrel{\ref{ttuttu}(i)}{\subset} U+\underline{\underline{z^{k+1}u_3u_2u_1}}+z^ku_2t^2u^3$. 
						However, by lemma \ref{ttuttu}(iii), we have that $z^kt^2u_3\in u_3^{\times}Ut^{-1}+z^{k+12}u_3^{\times}tu^{-1}tu^2tu^{-1}t^{-1}$. Since $k+12\in\{1,\dots,15\}$, we can apply lemma \ref{ttuttu}(iv) we have that
						$z^{k+12}tu^{-1}tu^2tu^{-1}t^{-1}\in Ut^{-1}+( z^{k+18}+z^{k+19}+z^{k+20})u_3tu_3u_2+z^{k+20}u_3tu_3tu^{-1}$.
						However, by (i) and by case C1, we have that $z^{k+12}tu^{-1}tu^2tu^{-1}t^{-1}\in Uu_2$.
						Therefore, $z^ktu_3tu^{-1}\subset u_3Ut^{-1}+U$. The result follows from the definition of $U$ and remark \ref{rem11}.\\
					\end{itemize}	
					\item[(iii)]
					The result follows immediately from (i) and (ii) (see corollary \ref{corr11}(iii)).
					\qedhere	
				\end{itemize}	
			\end{proof}	
			\begin{cor}
				The BMR freeness conjecture holds for the generic Hecke algebra $H_{G_{11}}$.
			\end{cor}
			\begin{proof}
				By theorem \ref{thm11}(iii) we have that $H_{G_{11}}=U$. The result follows from proposition \ref{BMR PROP}, since by definition $U$ is generated as left $u_3$ module by 144 elements and, hence, as $R$-module by $|G_{11}|=576$ elements (recall that $u_3$ is generated as $R$-module by 4 elements).
			\end{proof}		
		\subsubsection{\textbf{The case of }$\mathbf{G_{12}}$}
		Let $R=\ZZ[u_{s,i}^{\pm},u_{t,j}^{\pm},u_{u,l}^{\pm}]_{\substack{1\leq i,j,l\leq 2 }}$ and let $$ H_{G_{12}}=\langle s,t,u\;|\; stus=tust=ustu,\prod\limits_{i=1}^{2}(s-u_{s,i})=\prod\limits_{j=1}^{2}(t-u_{t,j})=\prod\limits_{l=1}^{2}(u-u_{u,l})=0\rangle.$$ 
		Let also $\bar R=\ZZ[u_{s,i}^{\pm},]_{\substack{1\leq i\leq 2 }}$. Under the specialization $\grf: R\twoheadrightarrow \bar R$, defined by $u_{t,j}\mapsto u_{s,i}$, $u_{u,l}\mapsto u_{s,i}$ the algebra $\bar H_{G_{12}}:=H_{G_{12}}\otimes_{\grf}\bar R$ is the generic Hecke algebra associated to $G_{12}$. 
		
		Let $u_1$ be the subalgebra of $H_{G_{12}}$ generated by $s$, $u_2$ the subalgebra of $H_{G_{12}}$ generated by $t$ and $u_3$ the subalgebra of $H_{G_{12}}$ generated by $u$. We recall that $z:=(stu)^4$ generates the center of the  complex braid group associated to $G_{12}$ and that $|Z(G_{12})|=2$.  We set $U=\sum\limits_{k=0}^1(z^ku_1u_3u_1u_2+z^ku_1u_2u_3u_2+z^ku_2u_3u_1u_2+z^ku_2u_1u_3u_2+z^ku_3u_2u_1u_2).$
		By the definition of $U$ we have the following remark:
		\begin{rem} $Uu_2\subset U$.
			\label{r12}
		\end{rem}
		To make it easier for the reader to follow the next calculations, we will underline the elements that by definition belong to $U$. Moreover, we will use directly remark \ref{r12}; this means that every time we have a power of $t$ at the end of an element we may ignore it. In order to remind that to the reader, we put a parenthesis around the part of the element we consider.

		\begin{prop}$u_1U\subset U$.
			\label{su12}
		\end{prop}
		\begin{proof}
			Since $u_1=R+Rs$ we have to prove that $sU\subset U$. By the definition of $U$ and by remark \ref{r12}, it is enough to prove that for every $k\in\{0,1\}$, $z^ksu_2u_3u_1$, $z^ksu_2u_1u_3$ and $z^ksu_3u_2u_1$ are subsets of $U$.
			We have:
			\begin{itemize}[leftmargin=*]
				\item $\small{\begin{array}[t]{lcl}
					z^ksu_2u_3u_1&=&z^ks(R+Rt)u_3u_1\\
					&\subset&\underline{z^ku_1u_3u_1}+z^kst(R+Ru)u_1\\
					&\subset&U+z^ku_1u_2u_1+z^kstu(R+Rs)\\
					&\subset&U+z^ku_1u_2u_1+\underline{z^ku_1u_2u_3}+Rz^kstus\\
					&\subset&U+z^ku_1u_2u_1+Rz^ktust\\
					&\subset&U+z^ku_1u_2u_1+\underline{u_2u_3u_1u_2}\\
					&\subset&U+z^ku_1u_2u_1.
					\end{array}}$
				\\
				\item $\small{\begin{array}[t]{lcl}
					z^ksu_2u_1u_3&=&z^ks(R+Rt^{-1})u_1u_3\\
					&\subset&\underline{z^ku_1u_3}+z^kst^{-1}(R+Rs^{-1})u_3\\
					&\subset&U+\underline{z^ku_1u_2u_3}+z^kst^{-1}s^{-1}(R+Ru^{-1})\\
					&\subset&U+z^ku_1u_2u_1+Rz^ks(t^{-1}s^{-1}u^{-1}t^{-1})t\\
					&\subset&U+z^ku_1u_2u_1+Rz^ku^{-1}t^{-1}s^{-1}t\\
					&\subset&U+z^ku_1u_2u_1+\underline{z^ku_3u_2u_1u_2}\\
					&\subset&U+z^ku_1u_2u_1.
					\end{array}}$\\
				\item $\small{\begin{array}[t]{lcl}
					z^ksu_3u_2u_1&\subset&z^k(R+Rs^{-1})u_3u_2u_1\\
					&\subset&\underline{z^ku_3u_2u_1}+z^ks^{-1}(R+Ru^{-1})u_2u_1\\
					&\subset&U+z^ku_1u_2u_1+z^ks^{-1}u^{-1}(R+Rt^{-1})u_1\\
					&\subset&U+z^ku_1u_2u_1+\underline{z^ku_1u_3u_1}+z^ks^{-1}u^{-1}t^{-1}(R+Rs^{-1})\\
					&\subset&U+z^ku_1u_2u_1+\underline{z^ku_1u_3u_2}+Rz^k(s^{-1}u^{-1}t^{-1}s^{-1})\\
					&\subset&U+z^ku_1u_2u_1+Rz^kt^{-1}s^{-1}u^{-1}t^{-1}\\
					&\subset&U+z^ku_1u_2u_1+\underline{z^ku_2u_1u_3u_2}\\
					&\subset&U+z^ku_1u_2u_1.
					\end{array}}$
			\end{itemize}
			Therefore, we have to prove that for every $k\in\{0,1\}$, $z^ku_1u_2u_1\subset U$.
			We distinguish the following cases:
			\begin{itemize}[leftmargin=*]
				\item $\underline{k=0}$: \\
				$\hspace*{-0.2cm}\small{\begin{array}[t]{lcl}
					u_1u_2u_1&=&(R+Rs)(R+Rt)(R+Rs)\\
					&\subset&\underline{u_2u_1u_2}+Rsts\\
					&\subset&U+
					(stu)^4u^{-1}(t^{-1}s^{-1}u^{-1}t^{-1})s^{-1}(u^{-1}t^{-1}s^{-1}u^{-1})s\\
					&\subset&U+zu^{-2}t^{-1}s^{-1}u^{-1}s^{-2}u^{-1}t^{-1}\\
					&\subset&U+zu^{-2}t^{-1}s^{-1}u^{-1}(R+Rs^{-1})u^{-1}t^{-1}\\
					&\subset&U+zu^{-2}t^{-1}s^{-1}u^{-2}t^{-1}+u^{-2}t^{-1}s^{-1}zu^{-1}s^{-1}u^{-1}t^{-1}\\
					&\subset&U+zu^{-2}t^{-1}s^{-1}(R+Ru^{-1})t^{-1}+u^{-2}t^{-1}s^{-1}(stu)^4u^{-1}s^{-1}u^{-1}t^{-1}\\
					&\subset&U+\underline{zu_3u_2u_1u_2}+zu^{-1}(u^{-1}t^{-1}s^{-1}u^{-1})t^{-1}+
					u^{-1}(stus)(tust)s^{-1}u^{-1}t^{-1}\\
					&\subset&U+z(u^{-1}t^{-1}s^{-1}u^{-1})t^{-2}+(stus)\\
					&\subset&U+zt^{-1}s^{-1}u^{-1})t^{-3}+tust\\
					&\subset&U+\underline{zu_2u_1u_3u_2}+\underline{u_2u_3u_1u_2}.
					\end{array}}$
				\item $\underline{k=1}$:\\
				$\hspace*{-0.2cm}\small{\begin{array}[t]{lcl}
					zu_1u_2u_1&=&(R+Rs^{-1})(R+Rt^{-1})(R+Rs^{-1})\\
					&\subset&\underline{zu_2u_1u_2}+zRs^{-1}t^{-1}s^{-1}\\
					&\subset&U+Rs^{-1}t^{-1}s^{-1}z\\
					&\subset&U+Rs^{-1}t^{-1}s^{-1}(stu)^4\\
					&\subset&U+Rs^{-1}(ustu)s(tust)u\\
					&\subset&U+Rtus^2ustu^2\\
					&\subset&U+Rtu(R+Rs)ustu^2\\
					
					\end{array}}$
				\\
				
				$\hspace*{-0.4cm}\small{\begin{array}[t]{lcl}
					\phantom{zu_1u_2u_1}	&\subset&U+Rtu^2st(R+Ru)+Rtus(ustu)u\\
					&\subset&U+\underline{u_2u_3u_1u_2}+Rtu^2(stu)^4(stu)^{-3}+Rtustustu\\

					&\subset&U+Rztu(t^{-1}s^{-1}u^{-1}t^{-1})s^{-1}u^{-1}t^{-1}s^{-1}+Rtu(stu)^4(stu)^{-2}\\
					&\subset&U+Rzs^{-1}u^{-1}(s^{-1}u^{-1}t^{-1}s^{-1})+	Rz(s^{-1}u^{-1}t^{-1}s^{-1})\\

					&\subset&U+Rz(s^{-1}u^{-1}t^{-1}s^{-1})u^{-1}t^{-1}+R	zt^{-1}s^{-1}u^{-1}t^{-1}\\

					&\subset&U+Rzu^{-1}t^{-1}s^{-1}u^{-2}t^{-1}+\underline{zu_2u_1u_3u_2}\\
					&\subset&U+zu^{-1}t^{-1}s^{-1}(R+Ru^{-1})t^{-1}\\
					&\subset&U+\underline{zu_3u_2u_1u_2}+z(u^{-1}t^{-1}s^{-1}u^{-1})t^{-1}\\
					&\subset&U+zt^{-1}s^{-1}u^{-1}t^{-2}\\
					&\subset&U+\underline{zu_2u_1u_3u_2}.\phantom{======================================}
					\qedhere
					\end{array}}$
			\end{itemize}
			
		\end{proof}
		
		Our goal is to prove that $H_{G_{12}}=U$ (theorem \ref{thm12}). The following proposition provides a criterium for this to be true.
		
		\begin{prop} If $u_3U\subset U$, then $H_{G_{12}}=U$.
			\label{prop12}
		\end{prop}
		\begin{proof}
			Since  $1\in U$, it is enough to prove that $U$ is a left-sided ideal of $H_{G_{11}}$. For this purpose, one may check that $sU$, $tU$ and $uU$ are subsets of $U$. By hypothesis and proposition \ref{prop12}, it is enough to prove that $tU\subset U$. By the definition of $U$ and remark \ref{r12} we have to prove that for every $k\in\{0,1\}$ $z^ktu_1u_3u_1$, $z^ktu_1u_2u_3$ and $z^ktu_3u_2u_1$ are subsets of $U$. We have:
			\begin{itemize}[leftmargin=*]
				\item $\small{\begin{array}[t]{lcl}
					z^ktu_1u_3u_1&\subset&z^k(R+Rt^{-1})u_1u_3u_1\\
					&\subset&\underline{z^ku_1u_3u_1}+z^kt^{-1}(R+Rs^{-1})(R+Ru)(R+Rs)\\
					&\subset&U+\underline{z^ku_2u_3u_1}+\underline{z^ku_2u_1u_3}+Rz^kt^{-1}s^{-1}us\\
					&\subset&U+Rz^kt^{-1}s^{-1}(ustu)u^{-1}t^{-1}\\
					&\subset&U+Rz^kusu^{-1}t^{-1}\\
					&\subset&U+u_3\underline{z^ku_1u_3u_2}\\
					&\subset& U+u_3U.
					\end{array}}$\\
				\item $\small{\begin{array}[t]{lcl}
					z^ktu_1u_2u_3&\subset&z^k(R+Rt^{-1})u_1u_2u_3\\
					&\subset&\underline{z^ku_1u_2u_3}+z^kt^{-1}(R+Rs)(R+Rt)(R+Ru)\\
					&\subset&U+\underline{z^ku_2u_1u_3}+\underline{z^ku_2u_1u_2}+Rz^kt^{-1}stu\\
					&\subset&U+Rz^kt^{-1}(stus)s^{-1}\\
					&\subset&U+Rz^kusts^{-1}\\
					&\subset&U+u_3u_1\underline{z^ku_2u_1}\\
					&\subset&U+u_3u_1U.
					\end{array}}$\\
				\item $\small{\begin{array}[t]{lcl}
					z^ktu_3u_2u_1&=&z^kt(R+Ru^{-1})(R+Rt^{-1})(R+Rs^{-1})\\
					&\subset&\underline{z^ku_2u_3u_1u_2}+Rz^ktu^{-1}t^{-1}s^{-1}\\
					&\subset&U+Rz^kt(u^{-1}t^{-1}s^{-1}u^{-1})u\\
					&\subset&U+Rz^ks^{-1}u^{-1}t^{-1}u\\
					&\subset&U+u_1u_3\underline{z^ku_2u_3}\\
					&\subset&U+u_1u_3U.
					\end{array}}$
			\end{itemize}
			The result follows from the hypothesis and proposition \ref{su12}.
			\qedhere
		\end{proof}
		We can now prove the main theorem of this section.
		\begin{thm}$H_{G_{12}}=U$.
			\label{thm12}
		\end{thm}
		\begin{proof}
			By proposition \ref{prop12} it is enough to prove that $u_3U\subset U$. Since $u_3=R+Ru$, it will be sufficient to check that $uU\subset U$. By the definition of $U$ and remark \ref{r12}, we only have to prove that for every $k\in\{0,1\}$, $z^kuu_1u_3u_1$, $z^kuu_1u_2u_3$, $z^kuu_2u_3u_1$ and $z^kuu_2u_1u_3$ are subsets of $U$. 
			\begin{itemize}[leftmargin=0.6cm]
				\item[C1.] We will prove that $z^kuu_1u_3u_1\subset U$,  $k\in\{0,1\}$.
				\begin{itemize}[leftmargin=0.5cm]
					\item [(i)]$\underline{k=0}$: 
					\\
					$\hspace*{-0.2cm}\small{\begin{array}[t]{lcl}
						uu_1u_3u_1&=&u(R+Rs)(R+Ru)(R+Rs)\\
						&\subset&\underline{u_3u_1}+Rusu+Rusus\\
						&\subset&U+Rzuz^{-1}su+Rzusuz^{-1}s\\
						&\subset&U+Rzu(stu)^{-4}su+Rzusu(stu)^{-4}s\\
						&\subset&U+Rzt^{-1}(s^{-1}u^{-1}t^{-1}s^{-1})u^{-1}(t^{-1}s^{-1}u^{-1}t^{-1}u)+\\&&+
						Rzus(t^{-1}s^{-1}u^{-1}t^{-1})s^{-1}(u^{-1}t^{-1}s^{-1}u^{-1})t^{-1}\\
						&\subset&U+Rzt^{-2}s^{-1}u^{-1}t^{-1}u^{-2}t^{-1}s^{-1}+Rzt^{-1}s^{-2}t^{-1}s^{-1}u^{-1}t^{-2}\\
						&\subset&U+Rz(R+Rt^{-1})s^{-1}u^{-1}t^{-1}u^{-2}t^{-1}s^{-1}+
						Rzt^{-1}(R+Rs^{-1})t^{-1}s^{-1}u^{-1}t^{-2}
						\end{array}}$\\
					$\hspace*{-0.2cm}\small{\begin{array}[t]{lcl}
						\phantom{uu_1u_3u_1}&\subset&U+Rzs^{-1}u^{-1}t^{-1}(R+Ru^{-1})t^{-1}s^{-1}+
						Rt^{-1}s^{-1}u^{-1}t^{-1}u^{-2}t^{-1}s^{-1}z+
						\underline{zu_2u_1u_3u_2}+\\&&+Rzt^{-1}s^{-1}(t^{-1}s^{-1}u^{-1}t^{-1})t^{-1}\\
						&\subset&U+u_1\underline{zu_3u_2u_1}+Rs^{-1}u^{-1}t^{-1}u^{-1}t^{-1}s^{-1}z+Rt^{-1}s^{-1}u^{-1}t^{-1}u^{-2}t^{-1}s^{-1}(stu)^4+\\&&+
						Rt^{-1}s^{-2}u^{-1}t^{-1}s^{-1}zt^{-1}\\ 
						
						&\subset&U+u_1U+Rs^{-1}u^{-1}t^{-1}u^{-1}t^{-1}s^{-1}(stu)^4+
						R(t^{-1}s^{-1}u^{-1}t^{-1})u^{-1}(stus)(tust)u+\\&&+
						Rt^{-1}s^{-2}u^{-1}t^{-1}s^{-1}(stu)^4t^{-1}\\
						
						&\subset&U+u_1U+Rs^{-1}u^{-1}t^{-1}(stus)t(ustu)+Rs^{-1}t(ustu)+
						Rt^{-1}s^{-1}(tust)(ustu)t^{-1}\\
						
						&\subset&U+u_1U+Rt^3ust+u_1t^2ust+R(ustu)s\\
						&\subset&U+u_1U+u_1\underline{u_2u_3u_1u_2}+Rstus^2\\
						&\subset&U+u_1U+u_1\underline{u_2u_3u_1}\subset U+u_1U.
						\end{array}}$
					\\\\
					The result follows from proposition \ref{su12}.\\
					\item [(ii)]\underline{$k=1$}:
					\\
					$\hspace*{-0.2cm}\small{\begin{array}[t]{lcl}
						zuu_1u_3u_1&\subset& z(R+Ru^{-1})u_1u_3u_1\\
						&\subset&\underline{zu_1u_3u_1}+z
						u^{-1}(R+Rs^{-1})(R+Ru^{-1})(R+Rs^{-1})\\
						&\subset&U+\underline{u_3u_1}+Rzu^{-1}s^{-1}u^{-1}+Rzu^{-1}s^{-1}u^{-1}s^{-1}\\
						&\subset&U+Ru^{-1}s^{-1}zu^{-1}+Ru^{-1}s^{-1}zu^{-1}s^{-1}\\
						&\subset&U+Ru^{-1}s^{-1}(stu)^4u^{-1}+Ru^{-1}s^{-1}(stu)^4u^{-1}s^{-1}
						\end{array}}$
					\\
					
					$\hspace*{-0.2cm}\small{\begin{array}[t]{lcl}
						\phantom{zuu_1u_3u_1}

						&\subset& U+Ru^{-1}(tust)u(stus)t+Ru^{-1}(tust)us(tust)s^{-1}\\
						
						&\subset&U+Rstu^3stut+R(stus)stu\\
						
						&\subset&U+u_1t(R+Ru)stut+Rust(ustu)\\

						&\subset&U+u_1t(stu)^4(stu)^{-3}t+u_1(tust)ut+Rust^2ust\\

						&\subset&U+zu_1t(u^{-1}t^{-1}s^{-1}u^{-1})t^{-1}s^{-1}u^{-1}t^{-1}s^{-1}t+u_1ustu^2+rus(R+Rt)ust\\
						
						&\subset&U+zu_1u^{-1}t^{-2}s^{-1}u^{-1}t^{-1}s^{-1}t+u_1ust(R+Ru)+
						(	uu_1u_3u_1)u_2+R(ustu)st\\ 
						
						&\stackrel{(i)}{\subset}&U+zu_1u^{-1}(R+Rt^{-1})s^{-1}u^{-1}t^{-1}s^{-1}t+u_1\underline{u_3u_1u_2}+
						u_1(ustu)+Uu_2+Rstus^2t\\
						
						&\stackrel{\ref{r12}}{\subset}&U+u_1U+zu_1u^{-1}(s^{-1}u^{-1}t^{-1}s^{-1})t+u_1(stu)^{-2}zt+
						u_1tust+u_1\underline{u_2u_3u_1u_2}\\
						
						&\subset& U+u_1U+zu_1(u^{-1}t^{-1}s^{-1}u^{-1})+u_1(stu)^{-2}(stu)^4t+
						u_1\underline{u_2u_3u_1u_2}\\
						
						&\subset&U+u_1U+zu_1t^{-1}s^{-1}u^{-1}t^{-1}+u_1t(ustu)t\\
						
						&\subset&U+u_1U+u_1\underline{u_2u_1u_3u_2}+u_1t^2ust^2\\
						
						&\subset&U+u_1U+u_1\underline{u_2u_3u_1u_2}\\
						&\subset&U+u_1U.
						\end{array}}$
					\\\\
					The result follows again from proposition \ref{su12}.\\
				\end{itemize}
				\item [C2.] We will prove that $z^kuu_1u_2u_3\subset U$,  $k\in\{0,1\}$. We notice that  $z^kuu_1u_2u_3=z^ku(R+Rs)(R+Rt)(R+Ru)\subset \underline{z^ku_3u_1u_2}+z^ku_3u_2u_3+z^kuu_1u_3u_1+Rz^k(ustu)\stackrel{C1}{\subset} U+z^ku_3u_2u_2+Rz^ktust\subset U+z^ku_3u_2u_2+\underline{z^ku_2u_3u_1u_2}\subset U+z^ku_3u_2u_2.$
				Therefore, we must prove that, for every $k\in\{0,1\}$, $z^ku_3u_2u_2\subset U$. We distinguish the following cases:
				\begin{itemize}[leftmargin=0.5cm]
					\item [(i)] $\underline{k=0}:$\\
					$\hspace*{-0.2cm}\small{\begin{array}[t]{lcl}
						u_3u_2u_3&=&(R+Ru)(R+Rt)(R+Ru)\\
						&\subset&\underline{u_2u_3u_2}+Rutu\\
						&\subset&U+Rus^{-1}stu\\
						&\subset&U+Ru(R+Rs)stu\\
						&\subset&U+R(ustu)+Rus(stu)^4(stu)^{-3}\\
						&\subset&U+Rtust+Rzus(u^{-1}t^{-1}s^{-1}u^{-1})(t^{-1}s^{-1}u^{-1}t^{-1})s^{-1}\\
						&\subset&U+\underline{u_2u_3u_1u_2}+Rzt^{-1}s^{-2}u^{-1}t^{-1}s^{-2}\\
						&\subset&U+Rzt^{-1}s^{-2}u^{-1}t^{-1}(R+Rs^{-1})\\
						&\subset&U+\underline{zu_2u_1u_3u_2}+Rt^{-1}s^{-1}u^{-1}t^{-1}s^{-1}z\\
						&\subset&U+Rt^{-1}s^{-2}u^{-1}t^{-1}s^{-1}(stu)^4\\
						&\subset&U+Rt^{-1}s^{-1}(tust)(ustu)\\
						&\subset&U+R(ustu)st\\
						&\subset&U+Rstus^2t\\
						&\subset&U+u_1\underline{u_2u_3u_1u_2}\\
						&\subset&U+u_1U.
						\end{array}}$
					\\\\
					The result follows from proposition \ref{su12}. 
					\item[(ii)]$\underline{k=1}:$\\
					$\hspace*{-0.2cm}\small{\begin{array}[t]{lcl}
						zu_3u_2u_3&=&z(R+Ru^{-1})(R+Rt^{-1})(R+Ru^{-1})\\
						&\subset&\underline{zu_2u_3u_2}+Rzu^{-1}t^{-1}u^{-1}\\
						&\subset&U+R(stu)^4u^{-1}t^{-1}u^{-1}\\
						&\subset&U+Rs(tust)u(stus)u^{-1}\\
						&\subset&U+Rs^2tusu^2st\\
						&\subset&U+u_1tus(R+Ru)st\\
						&\subset&U+u_1\underline{u_2u_3u_1u_2}+u_1(stu)^4(stu)^{-3}sust\\
						&\subset&U+u_1U+zu_1u^{-1}t^{-1}s^{-1}u^{-1}(t^{-1}s^{-1}u^{-1}t^{-1})ust\\
						&\subset&U+u_1U+zu_1u^{-1}t^{-1}s^{-1}u^{-2}\\
						&\subset&U+u_1U+zu_1u^{-1}t^{-1}s^{-1}(R+Ru^{-1})\\
						&\subset&U+u_1U+u_1\underline{zu_3u_2u_1}+zu_1(u^{-1}t^{-1}s^{-1}u^{-1})\\
						&\subset&U+u_1U+zu_1t^{-1}s^{-1}u^{-1}t^{-1}\\
						&\subset&U+u_1U+u_1\underline{zu_2u_1u_3u_2}\\
						&\subset&	U+u_1U.
						\end{array}}$
					\\\\
					The result follows again from proposition \ref{su12}.\\
				\end{itemize}
				\item[C3.] For every $k\in\{0,1\}$ we have:
				$z^kuu_2u_3u_1=z^ku(R+Rt)(R+Ru)(R+Rs)\subset \underline{z^ku_3u_2u_1}+Rz^kutu+Rz^kutus\subset U+z^kuu_1u_2u_3+Rz^ku(tust)t^{-1}$.
				Therefore, by C2 we have 	$z^kuu_2u_3u_1\subset U+Rz^k(ustu)st^{-1}\subset U+Rz^kstus^2t^{-1}\subset U+u_1\underline{z^ku_2u_3u_1u_2}\subset U+u_1U$. The result follows from \ref{su12}.\\
				\item[C4.]For every $k\in\{0,1\}$ we have:
				$z^kuu_2u_1u_3=z^ku(R+Rt^{-1})(R+Rs^{-1})(R+Ru^{-1})\subset U+\underline{z^ku_3u_2u_1}+z^kuu_1u_2u_3+z^kuu_1u_3u_1+Rz^kut^{-1}s^{-1}u^{-1}$. Therefore, by C1 and C2 we have  that $z^kuu_2u_1u_3\subset U+Rz^ku(t^{-1}s^{-1}u^{-1}t^{-1})t\subset U+
				Rz^kt^{-1}s^{-1}u^{-1}t\subset U+\underline{z^ku_2u_1u_3u_2}.$
				\qedhere
			\end{itemize}
		\end{proof}
		\begin{cor}
			The BMR freeness conjecture holds for the generic Hecke algebra $\bar H_{G_{12}}$.
		\end{cor}
		\begin{proof}
			By theorem \ref{thm12} we have that $H_{G_{12}}=U=\sum\limits_{k=0}^1z^k(u_2+su_2+uu_2+suu_2+usu_2+tuu_2+tsu_2+susu_2+stuu_2+tusu_2+tsuu_2+utsu_2)$ and, hence, $H_{G_{12}}$ is generated as right $u_2$-module by 24 elements and, hence, as $R$-module by $|G_{12}|=48$ elements (recall that $u_2$ is generated as $R$-module by 2 elements). Therefore, $\bar H_{G_{12}}$ is generated as $\bar R$-module by $|G_{12}|=48$ elements, since the action of $\bar R$ factors through $R$. The result follows from proposition \ref{BMR PROP}.
		\end{proof}
			\subsubsection{\textbf{The case of }$\mathbf{G_{13}}$}
			Let $R=\ZZ[u_{s,i}^{\pm},u_{t,j}^{\pm},u_{u_{u,l}}]_{\substack{1\leq i,j,l\leq 2 }}$ and let  $$H_{G_{13}}=\langle s,t,u\;|\; ustu=tust,stust=ustus, \prod\limits_{i=1}^{2}(s-u_{s,i})=\prod\limits_{j=1}^{2}(t-u_{t,l})=\prod\limits_{l=1}^{2}(u-u_{u_k})=0\rangle$$ be the generic Hecke algebra associated to $G_{13}$. Let $u_1$ be the subalgebra of $H_{G_{13}}$ generated by $s$, $u_2$ the subalgebra of $H_{G_{13}}$ generated by $t$ and $u_3$ be the subalgebra of $H_{G_{13}}$ generated by $u$. We recall that $z:=(stu)^3=(tus)^3=(ust)^3$ generates the center of the associated complex braid group and that $|Z(G_{13})|=4$. We set
			$U=\sum\limits_{k=0}^{3}z^k(u_1u_2u_1u_2+u_1u_2u_3u_2+u_2u_3u_1u_2+u_2u_1u_3u_2+u_3u_2u_1u_2)$. 
			By the definition of $U$, we have the following remark:
			\begin{rem}
				$Uu_2 \subset U$.
				\label{rem13}
			\end{rem}
			From now on, we will underline the elements that by definition belong to $U$. Moreover, we will use directly the remark \ref{rem13}; this means that every time we have a power of  $t$ at the end of an element, we may ignore it. In order to remind that to the reader, we put a parenthesis around the part of the element we consider.
			
			Our goal is to prove $H_{G_{13}}=U$ (theorem \ref{thm13}). Since $1\in U$, it is enough to prove that $U$ is a left-sided ideal of $H_{G_{13}}$. For this purpose, one may check that $sU$, $tU$ and $uU$ are subsets of $U$. We set
			\vbox{%
				\begin{multicols}{6}
					\begin{itemize}[leftmargin=*]
						\item [] $v_1=1$
						\item[]$v_2=u$
						\item []$v_3=s$
						\item []$v_4=ts$
						\item []$v_5=su$
						\item []$v_6=us$
						\item []$v_7=tu$
						\item []$v_8=tsu$
						\item []$v_9=tus$
						\item []$v_{10}=sts$
						\item []$v_{11}=stu$
						\item []$v_{12}=uts$
					\end{itemize}	\end{multicols}}\\
					By the definition of $U$, one may notice that $U=\sum\limits_{k=0}^3\sum\limits_{i=1}^{12}(Rz^kv_i+Rz^kv_it)$. Hence, by remark \ref{rem13} we only have to prove that for every $k\in\{0,\dots,3\}$, $sz^kv_i$, $tz^kv_i$ and $uz^kv_i$ are elements inside $U$, $i=1,\dots,12$. 
					As a first step, we prove this argument for $i=8,\dots, 12$ and for a smaller range of the values of $k$, as we can see in the following proposition.
					
					\begin{prop}
						\mbox{}
						\vspace*{-\parsep}
						\vspace*{-\baselineskip}\\
						\begin{itemize}[leftmargin=0.6cm]
							\item[(i)] For every $k\in\{1,2,3\}$, $z^ksv_{8}\in U.$	
							\item[(ii)] For every $k\in\{0,1,2\}$, $z^ksv_{9}\in U.$	
							\item[(iii)] For every $k\in\{0,1,2\}$, $z^kuv_{10}\in U.$
							
							\item[(iv)]For every $k\in\{0,1,2,3\}$, $z^ktv_{10}\in U.$
							
							\item[(v)]For every $k\in\{0,1,2,3\}$, $z^ktv_{11}\in U.$
							
							\item[(vi)] For every $k\in\{1,2,3\}$, $z^ksv_{12}\in U.$
							\item[(vii)] For every $k\in\{0,1,2,3\}$, $z^ktv_{12}\in U.$

						\end{itemize}
						\label{easy}
					\end{prop}
					\begin{proof}
						
						\mbox{}
						\vspace*{-\parsep}
						\vspace*{-\baselineskip}\\
						\begin{itemize}	[leftmargin=0.6cm]
							
							\item[(i)]
							$\hspace*{-0.2cm}\small{\begin{array}[t]{lcl}z^ksv_{8}&=&z^kstsu\\
								&=&z^kz^k(R+Rs^{-1})(R+Rt^{-1})(R+Rs^{-1})(R+Ru^{-1})\\
								&\in&
								\underline{z^ku_2u_1u_3}+\underline{z^ku_1u_2u_3}+\underline{z^ku_1u_2u_1}+
								Rz^ks^{-1}t^{-1}s^{-1}u^{-1}\\
								&\in& U+Rz^ks^{-1}(ust)^{-3}ustust\\
								&\in& U+Rz^{k-1}s^{-1}(ustus)t\\
								&\in& U+Rz^{k-1}tust^2\\
								&\in& U+\underline{(z^{k-1}u_2u_3u_1)t^2}.
								\end{array}}$
							\item[(ii)]
							$z^ksv_{9}=z^kstus=z^k(stu)^3(u^{-1}t^{-1}s^{-1}u^{-1})t^{-1}=z^{k+1}t^{-1}s^{-1}u^{-1}t^{-2}\in \underline{(z^{k+1}u_2u_1u_3)t^{-2}}$.
							\item[(iii)] 
							$z^kuv_{10}=z^kusts=z^k(ust)^3t^{-1}(s^{-1}u^{-1}t^{-1}s^{-1}u^{-1})s=
							z^{k+1}t^{-2}s^{-1}u^{-1}t^{-1}\in \underline{(z^{k+1}u_2u_1u_3)t}.$
							
							\item[(iv)]We distinguish the following cases:
							\begin{itemize}[leftmargin=*]
								\item[$\bullet$]\underline{$k\in\{0,1\}$}:\\
								$\hspace*{-0.2cm}\small{\begin{array}[t]{lcl}
									z^ktv_{10}&=&z^ktsts\\
									&=&z^kt(stu)^3(u^{-1}t^{-1}s^{-1}u^{-1})t^{-1}s^{-1}u^{-1}s\\
									&=&z^{k+1}s^{-1}u^{-1}t^{-2}s^{-1}u^{-1}s\\
									&\in&z^{k+1}s^{-1}u^{-1}(R+Rt^{-1})s^{-1}u^{-1}s\\
									&\in&Rz^{k+1}s^{-1}u^{-1}s^{-1}(R+Ru)s+Rz^{k+1}(s^{-1}u^{-1}t^{-1}s^{-1}u^{-1})s\\
									&\in&\underline{z^{k+1}u_1u_3}+Rz^{k+1}s^{-1}u^{-1}(R+Rs)us+Rz^{k+1}t^{-1}s^{-1}u^{-1}t^{-1}\\
									&\in&U+\underline{z^{k+1}u_2}+Rz^{k+1}s^{-1}(R+Ru)sus+\underline{(z^{k+1}u_2u_1u_3)t}\\
									&\in&U+\underline{z^{k+1}u_3u_1}+Rz^{k+1}s^{-1}us(ust)^3(t^{-1}s^{-1}u^{-1}t^{-1}s^{-1})u^{-1}t^{-1}\\
									&\in&U+Rz^{k+2}s^{-1}t^{-1}s^{-1}u^{-2}t^{-1}\\
									&\in&U+Rz^{k+2}s^{-1}t^{-1}s^{-1}(R+Ru^{-1})t^{-1}\\
									&\in&U+\underline{(z^{k+2}u_1u_2u_1)t^{-1}}+Rz^{k+1}s^{-1}t^{-1}s^{-1}u^{-1}(ust)^3t^{-1}\\
									&\in&U+Rz^{k+1}s^{-1}(ustus)\\
									&\in& U+Rz^{k+1}tust\\
									&\in& U+\underline{(z^{k+2}u_2u_3u_1)t}.
									\end{array}}$	
								\item[$\bullet$] \underline{$k\in\{2,3\}$}:\\
								$\hspace*{-0.2cm}\small{\begin{array}[t]{lcl}
									z^ktv_{10}&=&z^ktsts\\
									&\in&z^k(R+Rt^{-1})(R+Rs^{-1})(R+Rt^{-1})(R+Rs^{-1})\\
									&\in& \underline{(z^ku_1u_2u_1)u_2}+Rz^kt^{-1}s^{-1}t^{-1}s^{-1}\\
									&\in& U+Rz^{k-1}t^{-1}s^{-1}t^{-1}s^{-1}(stu)^3\\
									&\in&U+Rz^{k-1}t^{-1}s^{-1}(ustus)tu\\
									&\in&U+Rz^{k-1}ust^2u\\
									&\in&U+Rz^{k-1}us(R+Rt)u\\
									&\in&U+Rz^{k-1}usu+Rz^{k-1}(ustu)\\
									&\in&U+Rz^{k-1}(R+Ru^{-1})(R+Rs^{-1})(R+Ru^{-1})+Rz^{k-1}tust\\
									\end{array}}$\\
								
								$\hspace*{-0.2cm}\small{\begin{array}[t]{lcl}
									\phantom{z^ktv_{10}}

									&\in&U+\underline{z^{k-1}u_3u_1}+\underline{z^{k-1}u_1u_3}+Rz^{k-1}u^{-1}s^{-1}u^{-1}+\underline{(z^{k-1}u_2u_3u_1)t}\\
									&\in&U+Rz^{k-2}u^{-1}s^{-1}u^{-1}(ust)^3\\
									&\in&U+Rz^{k-2}u^{-1}(tust)ust\\
									&\in&U+Rz^{k-2}stu^2st\\
									&\in&U+Rz^{k-2}st(R+Ru)st\\
									&\in& U+\underline{(z^{k-2}u_1u_2u_1)t}+(Rz^{k-2}sv_9)t
									\stackrel{(ii)}{\subset}U.
									\end{array}}$\\
							\end{itemize}

							\item[(v)] We distinguish the following cases:
							\begin{itemize}[leftmargin=*]
								
								\item[$\bullet$]\underline{$k\in\{0,1\}$}: \\
								$\hspace*{-0.2cm}\small{\begin{array}[t]{lcl}
									z^ktv_{11}&=&z^ktstu\\
									&=&z^kt(stu)^3(u^{-1}t^{-1}s^{-1}u^{-1})t^{-1}s^{-1}\\
									&=&z^{k+1}s^{-1}u^{-1}t^{-2}s^{-1}\\
									&\in&z^{k+1}s^{-1}u^{-1}(R+Rt^{-1})s^{-1}
									\end{array}}$
								\\
								
								$\hspace*{-0.2cm}\small{\begin{array}[t]{lcl}
									\phantom{z^ktv_{11}}&\in& Rz^{k+1}s^{-1}u^{-1}s^{-1}+Rz^{k+1}s^{-1}u^{-1}t^{-1}s^{-1}\\
									
									&\in& Rz^{k+1}(R+Rs)(R+Ru)(R+Rs)+Rz^ks^{-1}u^{-1}t^{-1}s^{-1}(stu)^3\\
									&\in& \underline{z^{k+1}u_1u_3}+\underline{z^{k+1}u_3u_1}+Rz^{k+1}sus+Rz^kt(ustu)\\
									&\in& U+Rz^{k+1}s(ust)^3(t^{-1}s^{-1}u^{-1}t^{-1}s^{-1})u^{-1}t^{-1}+Rz^kt^2ust\\
									&\in&U+Rz^{k+2}u^{-1}t^{-1}s^{-1}u^{-2}+\underline{(z^ku_2u_3u_1)t}\\
									&\in&U+Rz^{k+2}u^{-1}t^{-1}s^{-1}(R+Ru^{-1})\\
									&\in&U+\underline{z^{k+2}u_3u_2u_1}+Rz^{k+2}(u^{-1}t^{-1}s^{-1}u^{-1})\\
									&\in& U+Rz^{k+2}t^{-1}s^{-1}u^{-1}t^{-1}\\
									&\in& U+\underline{(z^{k+2}u_2u_1u_3)t^{-1}}.
									\end{array}}$
								
								\item[$\bullet$]\underline{$k\in\{2,3\}$}:\\
								$\hspace*{-0.2cm}\small{\begin{array}[t]{lcl}
									z^ktv_{11}&=&z^ktstu\\
									&\in& z^k(R+Rt^{-1})(R+Rs^{-1})(R+Rt^{-1})(R+Ru^{-1})\\
									&\in&\underline{z^ku_1u_2u_3}+\underline{z^ku_2u_1u_3}+\underline{(z^ku_2u_1)u_2}+
									Rz^kt^{-1}s^{-1}t^{-1}u^{-1}\\
									&\in& U+Rz^{k-1}t^{-1}s^{-1}(stu)^3t^{-1}u^{-1}\\
									&\in&U+Rz^{k-1}ust(ustu)t^{-1}u^{-1}\\
									&\in&U+Rz^{k-1}ust^2usu^{-1}\\
									&\in&U+Rz^{k-1}us(Rt+R)usu^{-1}\\
									&\in&U+Rz^{k-1}(ust)^3(t^{-1}s^{-1}u^{-1}t^{-1})u^{-1}+Rz^{k-1}usu(R+Rs^{-1})u^{-1}\\
									&\in&U+Rz^ku^{-1}t^{-1}s^{-1}u^{-2}+\underline{z^{k-1}u_3u_1}+Rz^{k-1}us(R+Ru^{-1})s^{-1}u^{-1}\\
									&\in&U+Rz^ku^{-1}t^{-1}s^{-1}(R+Ru^{-1})+\underline{z^{k-1}u_2}+Rz^{k-1}u(R+Rs^{-1})u^{-1}s^{-1}u^{-1}\\
									&\in& U+\underline{z^ku_3u_2u_1}+Rz^k(u^{-1}t^{-1}s^{-1}u^{-1})+\underline{z^{k-1}u_1u_3}+
									Rz^{k-2}us^{-1}(stu)^3u^{-1}s^{-1}u^{-1}\\
									&\in& U+Rz^kt^{-1}s^{-1}u^{-1}t^{-1}+Rz^{k-2}utu(stust)s^{-1}u^{-1}\\
									&\in&U+\underline{(z^ku_2u_1u_3)t^{-1}}+Rz^{k-2}utu^2st\\
									&\in&U+Rz^{k-2}ut(R+Ru)st\\
									&\in& U+\underline{(z^{k-2}u_3u_2u_1)t}+Rz^{k-2}u(tust)\\
									&\in&U+Rz^{k-2}u^2stu\\
									&\in&U+Rz^{k-2}(R+Ru)stu\\
									&\in& U+\underline{z^{k-2}u_1u_2u_3}+Rz^{k-2}(ustu)\\
									&\in& U+Rz^{k-2}tust\\
									&\in& U+\underline{(z^{k-2}u_2u_3u_1)t}.
									\end{array}}$\\
							\end{itemize}
							
							\item[(vi)] 
							$\hspace*{-0.2cm}\small{\begin{array}[t]{lcl}z^ksv_{12}&=&z^ksuts\\
								&\in& z^k(R+Rs^{-1})(R+Ru^{-1})(R+Rt^{-1})(R+Rs^{-1})\\
								&\in&
								\underline{z^ku_3u_2u_1}+\underline{z^ku_1u_2u_1}+\underline{(z^ku_2u_1u_3)u_2}+
								Rz^ks^{-1}u^{-1}s^{-1}+Rz^ks^{-1}u^{-1}t^{-1}s^{-1}\\
								&\in& U+ Rz^{k-1}s^{-1}(stu)^3u^{-1}s^{-1}+Rz^{k-1}s^{-1}u^{-1}t^{-1}s^{-1}(stu)^3\\
								
								&\in& U+Rz^{k-1}tu(stust)s^{-1}+Rz^{k-1}tustu\\
								&\in& U+Rz^{k-1}tu^2stu+Rz^{k-1}tustu\\
								&\in&U+z^{k-1}tu_3stu\\
								&\in& U+z^{k-1}t(R+Ru)stu\\
								&\in& U+ Rz^{k-1}tv_{11}+Rz^{k-1}t(ustu)\\
								&\stackrel{(v)}{\in}&U+Rz^{k-1}t^2ust\\&\in& U+\underline{(z^{k-1}u_2u_3u_1)t}.
								\end{array}}$\\
							\item[(vii)] We distinguish the following cases:
							\begin{itemize}[leftmargin=*]
								
								\item[$\bullet$]\underline{$k\in\{0,1\}$}:\\
								$\hspace*{-0.2cm}\small{\begin{array}[t]{lcl}
									z^ktv_{12}&=&z^ktuts\\
									&=&z^ks^{-1}(stu)^3(u^{-1}t^{-1}s^{-1}u^{-1})t^{-1}s^{-1}ts\\
									&=&z^{k+1}s^{-1}t^{-1}s^{-1}u^{-1}t^{-2}s^{-1}ts\\
									&\in&z^{k+1}s^{-1}t^{-1}s^{-1}u^{-1}(R+Rt^{-1})s^{-1}ts\\
									&\in&Rz^{k+1}s^{-1}t^{-1}s^{-1}u^{-1}(R+Rs)ts+Rz^{k+1}s^{-1}(ust)^{-3}(ustu)ts\\
									&\in&Rz^{k+1}s^{-1}(ust)^{-3}(ustus)ts+Rz^{k+1}s^{-1}t^{-1}s^{-1}(R+Ru)sts+Rz^ks^{-1}tust^2s\\
									&\in&Rz^ktust^2s+\underline{z^{k+1}u_2}+Rz^{k+1}s^{-1}t^{-1}s^{-1}usts+Rz^k(R+Rs)tus(R+Rt)s\\
									&\in&U+Rz^ktus(R+Rt)s+Rz^{k+1}s^{-1}t^{-1}s^{-1}(ust)^3t^{-1}(s^{-1}u^{-1}t^{-1}s^{-1}u^{-1})s+
									\underline{z^ku_2u_3u_1}+\\&&+Rz^k(tust)s+Rz^kstus^2+Rz^k(stust)s\\
									&\in& U+
									\underline{z^ku_2u_3u_1}+Rz^k(tust)s+Rz^{k+2}s^{-1}t^{-1}s^{-1}t^{-2}s^{-1}u^{-1}t^{-1}+Rz^k(ustus)+\\&&+Rz^kstu(R+Rs)+Rz^kustus^2\\
									
									\end{array}}$
								\\
								
								$\hspace*{-0.2cm}\small{\begin{array}[t]{lcl}
									\phantom{z^ktv_{12}}

									&\in&U+Rz^k(ustus)+Rz^{k+2}s^{-1}t^{-1}s^{-1}(R+Rt^{-1})s^{-1}u^{-1}t^{-1}+
									(Rz^ksv_9)t+
									\underline{z^ku_1u_2u_3}+\\&&+Rz^ksv_9+Rz^kustu(R+Rs)\\
									&\stackrel{(ii)}{\in}&U+(Rz^ksv_9)t+Rz^{k+2}s^{-1}t^{-1}s^{-2}u^{-1}t^{-1}+
									Rz^{k+2}s^{-1}t^{-1}s^{-1}t^{-1}s^{-1}u^{-1}t^{-1}+\\&&+
									Rz^k(ustu)+Rz^k(ustus).
									\end{array}}$
								\\\\
								However, by (ii) we have that $z^ksv_9\in U$. Moreover, $z^k(ustu)=z^ktust\in\underline{(z^ku_2u_3u_1)t}$ and $z^k(ustus)=z^ksv_9t\stackrel{(ii)}{\in}U$. Therefore, it remains to prove that $D:=Rz^{k+2}s^{-1}t^{-1}s^{-2}u^{-1}t^{-1}+
								Rz^{k+2}s^{-1}t^{-1}s^{-1}t^{-1}s^{-1}u^{-1}t^{-1}$ is a subset of $U$. We have:
								\\ \\
								$\hspace*{-0.2cm}\small{\begin{array}{lcl}
									D&=&Rz^{k+2}s^{-1}t^{-1}s^{-2}u^{-1}t^{-1}+
									Rz^{k+2}s^{-1}t^{-1}s^{-1}t^{-1}s^{-1}u^{-1}t^{-1}\\
									&\subset&Rz^{k+2}s^{-1}t^{-1}(R+Rs^{-1})u^{-1}t^{-1}+
									Rz^{k+2}s^{-1}t^{-1}s^{-1}(ust)^{-3}(ustus)\\
									&\subset&U+\underline{(z^{k+2}u_1u_2u_3)t^{-1}}+Rz^{k+2}s^{-1}(ust)^{-3}(ustus)+Rz^{k+1}s^{-1}ust+
									\underline{(z^ku_2u_3u_1)t}\\
									&\in&U+Rz^{k+1}tust+Rz^{k+1}s^{-1}(R+Ru^{-1})(R+Rs^{-1})t\\
									&\in&U+\underline{(z^{k+1}u_2u_3u_1)t}+\underline{(z^{k+1}u_1u_3)t}+
									Rz^{k+1}s^{-1}u^{-1}s^{-1}t\\
									&\in&U+Rz^{k}s^{-1}(stu)^3u^{-1}s^{-1}t\\
									&\in&U+Rz^ktu(stust)s^{-1}t\\
									&\in&U+Rz^ktu^2stut\\
									&\in&U+Rz^kt(R+Ru)stut\\
									&\in&U+(Rz^ktv_{11})t+Rz^{k}(tust)ut\\
									&\stackrel{(v)}{\in}&U+Rz^kustu^2t\\
									&\in&U+Rz^kust(R+Ru)t\\
									&\in& U+\underline{(z^ku_3u_1)t^2}+Rz^k(ustu)t\\&\in& U+\underline{(z^ku_2u_3u_1)t}.
									\end{array}}$
								
								\			\item[$\bullet$]\underline{$k\in\{2,3\}$}: \\
								$\hspace*{-0.2cm}\small{\begin{array}[t]{lcl}
									z^ktv_{12}&=&z^ktuts\\
									&\in& z^k(R+Rt^{-1})(R+Ru^{-1})(R+Rt^{-1})(R+Rs^{-1})\\
									&\in&
									\underline{z^ku_3u_2u_1}+\underline{(z^ku_2u_3u_1)u_2}+Rz^kt^{-1}u^{-1}t^{-1}s^{-1}\\
									&\in& U+Rz^{k-1}t^{-1}u^{-1}t^{-1}s^{-1}(stu)^3\\
									&\in&U+Rz^{k-1}t^{-1}s(tust)u\\
									&\in& U+Rz^{k-1}t^{-1}sustu^2\\
									& \in& U+Rz^{k-1}t^{-1}sust(Ru+R)\\
									&\in& U+Rz^{k-1}t^{-1}s(ustu)+Rz^{k-1}t^{-1}(R+Rs^{-1})ust\\
									&\in&U+Rz^{k-1}t^{-1}stust+\underline{(z^{k-1}u_2u_3u_1)t}+Rz^{k-1}t^{-1}s^{-1}(R+Ru^{-1})st\\
									&\in&U+Rz^{k-1}t^{-1}(stu)^3(u^{-1}t^{-1}s^{-1}u^{-1})+\underline{z^{k-1}u_2}+
									Rz^{k-1}t^{-1}s^{-1}u^{-1}(R+Rs^{-1})t\\
									&\in&U+Rz^kt^{-2}s^{-1}u^{-1}t^{-1}+\underline{(z^{k-1}u_2u_1u_3)t}+
									Rz^{k-1}(ust)^{-3}u(stust)s^{-1}t\\
									&\in&U+\underline{(z^ku_2u_1u_3)t^{-1}}+Rz^{k-2}u^2stut\\
									&\in&U+Rz^{k-2}(R+Ru)stut\\
									&\in&U+\underline{(z^{k-2}u_1u_2u_3)t}+Rz^{k-2}(ustu)\\
									&\in& U+Rz^{k-2}tust\\
									&\in& U+\underline{(z^{k-2}u_2u_3u_1)t}.\phantom{===================================}
									\qedhere
									
									\end{array}}$

							\end{itemize}	
							
						\end{itemize}

					\end{proof}

					\begin{cor} $u_2U\subset U$.
						\label{cor13}
					\end{cor}	
					
					\begin{proof}
						Since $u_2=R+Rt$, it is enough to prove that $tU\subset U$. However, by the definition of $U$ and remark \ref{rem13}(i), this is the same as proving that for every $k\in\{0,\dots,3\}$, $z^ktv_i\in U$, which follows directly from the definition of $U$ and proposition \ref{easy}(v), (vi), (vii).
					\end{proof}
					By remark \ref{rem13} (ii), we have that for every $k\in\{0,\dots,3\}$, $z^ku_iu_ju_k\subset U$, for some combinations of $i,j,l \in\{1,2,3\}$. We can generalize this argument for every $i,j,l \in\{1,2,3\}$.
					\begin{prop}
						For every $k\in\{0,\dots,3\}$ and for every combination of $i,j,l\in\{1,2,3\}$, $z^ku_iu_ju_l\subset U$.
						\label{prr13}
					\end{prop}
					\begin{proof}
						By the definition of $U$  we only have to prove that, for every $k\in\{0,\dots,3\}$, $z^ku_1u_3u_1$, $z^ku_3u_1u_3$ and $z^ku_3u_2u_3$ are subsets of $U$. We distinguish the following 3 cases:
						\begin{itemize}[leftmargin=0.6cm]
							\item[C1.] 
							\begin{itemize}[leftmargin=-0.1cm]
								\item []\underline{$k\in\{0,1,2\}$}:\\
								$\hspace*{-0.2cm}\small{\begin{array}[t]{lcl}z^ku_1u_3u_1&=&z^k(R+Rs)(R+Ru)(R+Rs)\\
									&\subset& z^kv_5+z^kv_6+\underline{z^ku_1}+Rz^ksus\\
									&\subset& U+Rz^ks(ust)^3(t^{-1}s^{-1}u^{-1}t^{-1}s^{-1})u^{-1}\\
									&\subset& U+Rz^{k+1}u^{-1}t^{-1}s^{-1}u^{-2}\\
									&\subset& U+z^{k+1}u^{-1}t^{-1}s^{-1}(R+Ru^{-1})\\
									&\subset& U+\underline{z^{k+1}u_3u_2u_1}+Rz^{k+1}(u^{-1}t^{-1}s^{-1}u^{-1})\\
									&\subset& U+z^{k+1}t^{-1}s^{-1}u^{-1}t^{-1}\\
									&\subset& U+\underline{(z^{k+1}u_2u_1u_3)t^{-1}}.
									\end{array}}$
								\item[]\underline{$k=3$}:\\
								$\hspace*{-0.2cm}\small{\begin{array}[t]{lcl}z^3u_1u_3u_1&=&z^k(R+Rs^{-1})(R+Ru{-1})s\\
									&\subset& \underline{z^3u_3u_1}+\underline{z^3u_1u_3}+\underline{z^3u_1}+Rz^3s^{-1}u^{-1}s\\
									&\subset& U+Rz^{2}s^{-1}(stu)^3u^{-1}s\\
									&\subset& U+Rz^{2}tu(stust)s\\
									&\subset& U+Rz^{2}tu^2stu\\
									&\subset& U+Rz^{2}t(R+Ru)stu\\
									&\subset& U+
									t(\underline{Rz^{2}u_1u_2u_3})+Rz^{2}(ustu)\\
									&\subset& U+u_2U+Rz^{2}v_9t.
									\end{array}}$
								
								The result follows from corollary \ref{cor13} and the definition of $U$.\\
							\end{itemize}
							\item [C2.] 
							\begin{itemize}[leftmargin=-0.1cm]
								\item[] \underline{$k\in\{1,2,3\}$}:\\
								$\hspace*{-0.2cm}\small{\begin{array}[t]{lcl}z^ku_3u_1u_3&=&z^k(R+Ru^{-1})(R+Rs^{-1})(R+Ru^{-1})\\&\subset& \underline{z^ku_1u_3}+\underline{z^ku_3u_1}+Rz^ku^{-1}s^{-1}u^{-1}\\
									&\subset& U+Rz^{k-1}u^{-1}s^{-1}u^{-1}(ust)^3\\&\subset& U+Rz^{k-1}u^{-1}(tust)ust\\
									&\subset& U+Rz^{k-1}stu^2st\\
									&\subset& U+Rz^{k-1}st(R+Ru)st\\
									&\subset& U+\underline{(z^{k-1}u_1u_2u_1)t}+(Rz^{k-1}sv_9)t.\end{array}}$.
								\\\\
								The result follows from proposition \ref{easy}(ii).
								\item[] \underline{$k=0$}:\\
								$\hspace*{-0.2cm}\small{\begin{array}[t]{lcl}
									u_3u_1u_3&=&(R+Ru)(R+Rs)(R+Ru)\\
									&\subset& \underline{u_1u_3}+\underline{u_3u_1}+Rusu\\
									&\subset& U+Rustt^{-1}u\\
									&\subset& U+Rust(R+Rt)u\\
									&\subset& U+R(ustu)+R(ust)^3t^{-1}s^{-1}(u^{-1}t^{-1}s^{-1}u^{-1})tu\\
									&\subset&U+Rtust+Rzt^{-1}s^{-1}t^{-1}s^{-1}\\
									&\subset& U+\underline{(u_2u_3u_1)t}+u_2(\underline{zu_1u_2u_1})\\
									&\subset&U+u_2\underline{(z^{k+1}u_3u_1)t}+u_2z^{k+1}uv_{10}+
									u_2\underline{z^{k+2}u_3u_2u_1}\\
									&\subset& U+u_2U.\end{array}}$
								\\\\
								The result follows from corollary \ref{cor13}.
								\\
							\end{itemize}
							\item[C3.] 
							\begin{itemize}[leftmargin=-0.1cm]
								\item[]\underline{$k\in\{0,1\}$}:\\
								$\hspace*{-0.2cm}\small{\begin{array}[t]{lcl}
									z^ku_3u_1u_3&=&z^k(R+Ru)(R+Rt)(R+Ru)\\
									&\subset& \underline{z^ku_2u_3}+\underline{z^ku_3u_2}+Rz^kutu\\
									&\subset&U+Rz^kus^{-1}stu\\
									&\subset&U+Rz^ku(R+Rs)stu\\
									&\subset&U+Rz^k(ustu)+Rz^kus(stu)^3(u^{-1}t^{-1}s^{-1}u^{-1})t^{-1}s^{-1}\\
									&\subset&U+Rz^ktust+Rz^{k+1}ust^{-1}s^{-1}u^{-1}t^{-2}s^{-1}\\
									&\subset&U+\underline{(z^ku_2u_3u_1)t}+Rz^{k+1}ust^{-1}s^{-1}u^{-1}(R+Rt^{-1})s^{-1}\\
									&\subset&U+Rz^{k+1}us(R+Rt)s^{-1}u^{-1}s^{-1}+Rz^{k+1}us(t^{-1}s^{-1}u^{-1}t^{-1}s^{-1})\\
									&\subset&U+\underline{z^{k+1}u_1}+Rz^{k+1}usts^{-1}(R+Ru)s^{-1}+Rz^{k+1}t^{-1}s^{-1}u^{-1}\\
									&\subset&U+z^{k+1}ustu_1+Rz^{k+1}ust(R+Rs)us^{-1}+\underline{z^{k+1}u_2u_1u_3}\\
									&\subset&U+z^{k+1}ustu_1+Rz^{k+1}(ustu)s^{-1}+
									Rz^{k+1}(ust)^3t^{-1}(s^{-1}u^{-1}t^{-1}s^{-1}u^{-1})sus^{-1}\\
									&\stackrel{\phantom{\ref{easy}(iv)}}{\subset}&U+z^{k+1}u_2ustu_1+Rz^{k+2}t^{-1}(t^{-1}s^{-1}u^{-1}t^{-1})us^{-1}\\
									&\subset&U+z^{k+1}u_2ust(R+Rs)+Rz^{k+2}t^{-1}u^{-1}t^{-1}s^{-2}\\
									&\subset&U+\underline{u_2(z^{k+1}u_3u_1)t}+u_2z^{k+1}uv_{10}+\underline{u_2(z^{k+2}u_3u_2u_1}\\

									&\stackrel{\ref{easy}(iv)}{\subset}&U+u_2U.\end{array}}$
								\\\\
								The result follows from proposition \ref{cor13}(i).\\
								\item[]  \underline{$k\in\{2,3\}$}:\\
								$\hspace*{-0.2cm}\small{\begin{array}[t]{lcl}
									z^ku_3u_1u_3&=&z^k(R+Ru^{-1})(R+Rt^{-1})(R+Ru^{-1})\\
									&\subset& \underline{z^ku_2u_3}+\underline{z^ku_3u_2}+Rz^ku^{-1}t^{-1}u^{-1}\\
									&\subset&U+Rz^ku^{-1}t^{-1}s^{-1}su^{-1}\\
									&\subset&U+Rz^ku^{-1}t^{-1}s^{-1}(R+Rs^{-1})u^{-1}\\
									&\subset&U+Rz^k(u^{-1}t^{-1}s^{-1}u^{-1})+Rz^k(stu)^{-3}st(ustu)s^{-1}u^{-1}\\
									&\subset&U+Rz^kt^{-1}s^{-1}u^{-1}t^{-1}+Rz^{k-1}st^2usts^{-1}u^{-1}\\
									&\subset&U+\underline{(z^{k}u_2u_1u_3)t^{-1}}+Rz^{k-1}s(R+Rt)usts^{-1}u^{-1}\\
									&\subset&U+Rz^{k-1}sus(R+Rt^{-1})s^{-1}u^{-1}+Rz^{k-1}(stust)s^{-1}u^{-1}\\
									&\subset&U+\underline{z^{k-1}u_1}+Rz^{k-1}s(R+Ru^{-1})st^{-1}s^{-1}u^{-1}+Rz^{k-1}ust\\
									&\subset&U+z^{k-1}u_1t^{-1}s^{-1}u^{-1}+Rz^{k-1}su^{-1}(R+Rs^{-1})t^{-1}s^{-1}u^{-1}+
									\underline{(z^{k-1}u_3u_1)t}\\
									&\subset&U+z^{k-1}u_1t^{-1}s^{-1}u^{-1}+Rz^{k-1}s(u^{-1}t^{-1}s^{-1}u^{-1})+Rz^{k-1}su^{-1}s^{-1}(ust)^{-3}(ustus)t\\
									&\subset&U+z^{k-1}u_1t^{-1}s^{-1}u^{-1}u_2+Rz^{k-2}su^{-1}(tust)t\\
									&\subset&U+z^{k-1}(R+Rs)t^{-1}s^{-1}u^{-1}u_2+Rz^{k-2}s^2tut\\
									&\subset&U+\underline{z^{k-1}u_2u_1u_3}+Rz^{k-1}s(R+Rt)(R+Rs)(R+Ru)u_2+\underline{(z^{k-2}u_1u_2u_3)t}\\
									&\subset&U+\underline{(z^{k-1}u_3)u_2}+\underline{(z^{k-1}u_1u_2u_3)u_2}+\underline{(z^{k-1}u_1u_2u_1)u_2}+
									+Rz^{k-1}stsuu_2\\
									&\subset&U+(Rz^{k-1}sv_8)u_2.
									\end{array}}$
								\\\\
								The result follows from proposition \ref{easy}(i).
								\qedhere
							\end{itemize}
						\end{itemize}
					\end{proof}

					\begin{prop} If $u_3U\subset U$, then $H_{G_{13}}=U$.
						\label{prop13}
					\end{prop}
					\begin{proof}
						As we explained in the beginning of this section, in order to prove that $H_{G_{13}}=U$, it will be sufficient to prove that   $sU$, $tU$ and $uU$ are subsets of $U$. By corollary \ref{cor13} and hypothesis, we only have to prove that $sU\subset U$. By the definition of $U$ and remark \ref{rem13}, we have to prove that for every $k\in\{0,\dots,3\}$, $z^ksv_i\in U$, $k=1,\dots,12$. However, for every $i\in\{1,\dots,7\}\cup\{10,11\}$, we have that $z^ksv_i\in z^ku_1u_ju_l$, where $j,l\in\{1,2,3\}$ and not necessarily distinct. Therefore, by proposition \ref{prr13} we only have to prove that $z^ksv_i\in U$, for $i=8,9,12$.
						For this purpose, for every $k\in\{0,\dots,3\}$ we have to check the following cases:
						\begin{itemize}[leftmargin=*]
							\item For the element $z^ksv_8$ we only have to prove that $z^0sv_8\in U$, since by proposition \ref{easy}(i) we have the rest of the cases. We have:
							$sv_8=stsu=sts(ust)^3(t^{-1}s^{-1}u^{-1}t^{-1}s^{-1})u^{-1}t^{-1}s^{-1}=
							zst(u^{-1}t^{-1}s^{-1})u^{-1})u^{-1}t^{-1}s^{-1}
							=zu^{-1}t^{-1}u^{-1}t^{-1}s^{-1}
							\in u_3u_2\underline{zu_3u_2u_1}$.
							Therefore, 	$sv_8
							\in u_3u_2U$.
							The result follows from corollary \ref{cor13} and from hypothesis.
							\item For the element $z^ksv_9$ we use proposition \ref{easy}(ii) and we only have to prove that $z^3sv_9\in U$.
							$z^3sv_9=z^3stus\in z^3(R+Rs^{-1})(R+Rt^{-1})(R+Ru^{-1})(R+Rs^{-1})\subset \underline{z^3u_2u_3u_1}+\underline{z^3u_1u_2u_1}+z^ku_1u_2u_1+Rz^3s^{-1}t^{-1}u^{-1}s^{-1}\stackrel{\ref{prr13}}{\subset}U+Rz^2s^{-1}(stu)^3t^{-1}u^{-1}s^{-1}\subset U+Rz^2tu(stustut^{-1}u^{-1}s^{-1}).$
							By hypothesis and by corollary \ref{cor13}, it is enough to prove that 	$z^2stustut^{-1}u^{-1}s^{-1}\in U$. 
							\\\\
							$\hspace*{-0.2cm}\small{\begin{array}{lcl}
								z^2st(ustu)t^{-1}u^{-1}s^{-1}&=&z^2st^2usu^{-1}s^{-1}\\&\in& z^2s(R+Rt)usu^{-1}s^{-1}\\
								&\in&Rz^2su(R+Rs^{-1})u^{-1}s^{-1}+
								Rz^2(stu)^3u^{-1}(t^{-1}s^{-1}u^{-1}t^{-1})u^{-1}s^{-1}\\
								&\in& U+\underline{z^2u_2}+Rz^2s(R+Ru^{-1})s^{-1}u^{-1}s^{-1}+Rz^3u^{-2}t^{-1}s^{-1}u^{-2}s^{-1}\\
								&\in&U+\underline{z^2u_3u_1}+Rz^2(R+Rs^{-1})u^{-1}s^{-1}u^{-1}s^{-1}+u_3u_2(z^3u_1u_3u_1)\\
								&\stackrel{\ref{prr13}}{\in}&U+u_3(z^2u_1u_3u_1)+Rzs^{-1}u^{-1}s^{-1}(stu)^3u^{-1}s^{-1}+u_3u_2U\\
								&\stackrel{\ref{prr13}}{\in}&U+u_3u_2U+Rzs^{-1}u^{-1}(tust)usts^{-1}\\
								&\in&U+u_3u_2U+Rztu^2sts^{-1}\\
								&\in&U+u_3u_2U+u_2u_3(\underline{zu_1u_2u_1})\\
								&\subset& U+u_3u_2u_3U.
								\end{array}}$
							\\\\
							The result follows from hypothesis and from corollary \ref{cor13}.
							\item For the element $z^ksv_{12}$  we only have to prove that $z^0sv_{12}\in U$, since the rest of the cases have been proven in proposition \ref{easy}(vi). We have:
							\\
							$\hspace*{-0.2cm}\small{\begin{array}[t]{lcl}
								sv_{12}&=&suts\\
								&=&(stu)^3u^{-1}t^{-1}s^{-1}u^{-1}(t^{-1}s^{-1}u^{-1}t^{-1})uts\\
								&=&z u^{-1}t^{-1}s^{-1}u^{-2}t^{-1}s^{-1}ts\\
								&\in&zu^{-1}t^{-1}s^{-1}(R+Ru^{-1})t^{-1}s^{-1}ts\\
								&\in&Rzu^{-1}t^{-1}s^{-1}t^{-1}s^{-1}ts+Rz(stu)^{-3}u^{-1}(ustu)ts\\
								&\in&Rzu^{-1}t^{-1}s^{-1}t^{-1}(R+Rs)ts+Ru^{-1}tust^2s\\
								&\in&\underline{(zu_3)u_2}+Rzu^{-1}t^{-1}s^{-1}(R+Rt)sts+u_3u_2u_3\underline{(u_1u_2u_1)}\\
								&\in&U+\underline{zu_3u_1}+Rzu^{-1}t^{-1}(R+Rs)tsts+u_3u_2u_3U\\
								&\in&U+u_2u_3U+u_3\underline{(zu_1u_2u_1)}+zu_3u_2(stu)^3u^{-1}t^{-1}(s^{-1}u^{-1}t^{-1}s^{-1}u^{-1})sts\\
								&\in&U+u_3u_2u_3U+z^2u_3u_2u^{-1}t^{-2}s^{-1}u^{-1}s\\
								&\in&U+u_3u_2u_3U+z^2u_3u_2u_3u_2\underline{(z^2u_1u_3u_1)}\\
								&\subset& U+u_3u_2u_3u_2U.
								\end{array}}$
							\\\\
							The result follows again from hypothesis and from corollary \ref{cor13}.
							\qedhere
						\end{itemize}
					\end{proof}
					We can now prove the main theorem of this section.
					\begin{thm} $H_{G_{13}}=U$.
						\label{thm13}
					\end{thm}
					\begin{proof}
						By proposition \ref{prop13} it will be sufficient to prove that $u_3U\subset U$. By the definition of $U$ and remark \ref{rem13}(i) we have to prove that for every $k\in\{0,\dots,3\}$, $z^ku_3v_i\subset U$, $k=1,\dots,12$. However, for every $i\in\{1,\dots,7\}\cup\{12\}$, $z^ku_3v_i\subset z^ku_3u_ju_l$, where $j,l\in\{1,2,3\}$ and not necessarily distinct. Therefore, by proposition \ref{prr13} we restrict ourselves to proving that $z^ku_3v_i\subset U$, for $i=8,9,10,11$.
						For every $k\in\{0,\dots,3\}$ we have:
						\begin{itemize}[leftmargin=*]
							
							\item$\hspace*{-0.2cm}\small{\begin{array}[t]{lcl}
								z^ku_3v_8&\subset&z^k(R+Ru^{-1})tsu\\
								&\subset& \underline{z^ku_2u_1u_3}+Rz^ku^{-1}(R+Rt^{-1})(R+Rs^{-1})(R+Ru^{-1})\\
								&\subset& U+z^ku_3u_1u_3+z^ku_3u_2u_3+\underline{z^ku_3u_2u_1}+Rz^k(u^{-1}t^{-1}s^{-1}u^{-1})\\
								&\stackrel{\ref{prr13}}{\subset}&U+Rz^kt^{-1}s^{-1}u^{-1}t^{-1}\\&\subset& U+\underline{(z^ku_2u_1u_3)t^{-1}}.
								\end{array}}$\\
							
							\item $\hspace*{-0.2cm}\small{\begin{array}[t]{lcl}
								z^ku_3v_9&\subset& z^k(R+Ru)tus\\
								&\subset& \underline{z^ku_2u_3u_1}+Rz^ku(tust)t^{-1}\\
								&\subset& U+z^ku^2stut^{-1}\\
								&\subset& U+z^k(R+Ru)stut^{-1}\\
								&\subset& U+\underline{(z^ku_1u_2u_3)t^{-1}}+Rz^k(ustu)t^{-1}\\
								&\subset& U+Rz^ktus\\
								&\subset& U+\underline{z^ku_2u_3u_1}.\end{array}}$\\
							
							\item $z^ku_3v_{10}\subset z^k(R+Ru)v_{10}\subset \underline{Rz^kv_{10}}+Rz^kuv_{10}.$ Therefore, by proposition \ref{easy}(iii), we only have to prove that $z^3uv_{10}\in U$. However, $z^3v_{10}=
							z^3t^{-1}(tust)s=z^3t^{-1}(ustus)=z^3t^{-1}stust\in u_2(z^3stus)u_2$. Hence, by remark \ref{rem13} and corollary \ref{cor13}, we need to prove that $z^3stus\in U$.\\ \\
							$\hspace*{-0.2cm}\small{\begin{array}{lcl}
								z^3stus&=&z^3(R+Rs^{-1})(R+Rt^{-1})(R+Ru^{-1})(R+Rs^{-1})\\
								&\in&\underline{z^3u_1u_2u_3}+\underline{z^3u_1u_2u_1}+\underline{z^3u_2u_3u_1}+
								z^3u_2u_3u_1+Rs^{-1}t^{-1}u^{-1}s^{-1}\\
								&\stackrel{\ref{prr13}}{\in}& U+Rz^2s^{-1}(stu)^3t^{-1}u^{-1}s^{-1}\\
								&\in&U+Rz^2tust(ustu)t^{-1}u^{-1}s^{-1}\\
								&\in&U+Rz^2tust^2usu^{-1}s^{-1}\\
								&\in&U+Rz^2tus(Rt+R)usu^{-1}s^{-1}\\
								&\in&U+Rz^2t(ust)^3(t^{-1}s^{-1}u^{-1}t^{-1})u^{-1}s^{-1}+Rz^2tusu(R+Rs^{-1})u^{-1}s^{-1}\\
								&\in&U+Rz^3tu^{-1}t^{-1}s^{-1}u^{-2}s^{-1}+\underline{z^2u_2u_3}+Rz^2tus(R+Ru^{-1})s^{-1}u^{-1}s^{-1}\\
								&\in&U+Rz^3tu^{-1}t^{-1}s^{-1}(R+Ru^{-1})s^{-1}+\underline{z^2u_2u_1}+
								Rz^2tu(R+Rs^{-1})u^{-1}s^{-1}u^{-1}s^{-1}\\
								&\stackrel{\phantom{\ref{easy}(iii)}}{\in}&U+u_2\underline{z^3u_3u_2u_1}+Rz^2tu^{-1}t^{-1}s^{-1}u^{-1}(ust)^3s^{-1}+
								u_2(z^2u_1u_3u_1)+\\&&+Rztus^{-1}u^{-1}(ust)^3s^{-1}u^{-1}s^{-1}\\
								&\stackrel{\ref{prr13}}{\in}& U+u_2U+Rz^3t(stust)s^{-1}+Rztutu(stust)s^{-1}u^{-1}s^{-1}\\
								&\stackrel{\ref{cor13}}{\in}& U+Rz^3t(ustu)+Rztutu^2sts^{-1}\\
								&\in&U+Rz^3t^2ust+Rztut(R+Ru)sts^{-1}\\
								&\in&U+Rz^3t^2ust+Rztut(R+Ru)sts^{-1}\\
								&\in&U+u_2\underline{(z^3u_3u_1)t}+Rztu(R+Rt^{-1})sts^{-1}+Rztu(tust)s^{-1}\\
								&\in&U+Rztust(R+Rs)+Rztut^{-1}(R+Rs^{-1})ts^{-1}+
								Rztu^2stus^{-1}\\
								&\in&U+u_2U+\underline{(zu_2u_3u_1)t}+zu_2uv_{10}+\underline{zu_2u_3u_1}+Rztut^{-1}s^{-1}(R+Rt^{-1})s^{-1}+\\&&+
								Rzt(R+Ru)stus^{-1}\\
								&\stackrel{\ref{easy}(iii)}{\in}& U+u_2U+u_2\underline{zu_3u_2u_1}+Rzt(R+Ru^{-1})t^{-1}s^{-1}t^{-1}s^{-1}+
								Rztstu(R+Rs)+\\&&+
								Rztustu(R+Rs)\\
								&\stackrel{\ref{cor13}}{\in}& U+\underline{zu_1u_2u_1}+Rzt(stu)^{-3}st(ustu)t^{-1}s^{-1}+Rztv_{11}+u_2zsv_9+
								
								Rzt(ustu)+\\&&+Rzt(ustu)s\\
								&\stackrel{\ref{easy}}{\in}&U+Rst^2u+Rzt^2ust+Rzt^2usts\\
								&\in&U+\underline{u_1u_2}+\underline{(zu_2u_3u_1)t}+zu_2uv_{10}\\
								&\stackrel{\ref{easy}(iii)}{\subset}&U+u_2U.
								\end{array}}$
							\\\\
							The result follows from corollary \ref{cor13}.\\
							
							\item $z^ku_3v_{11}\subset z^k(R+Ru)stu\subset \underline{z^ku_1u_2u_3}+Rz^k(ustu)
							\subset U+Rz^ktust\subset U+\underline{(z^ku_2u_3u_1)t}\subset U.$
						\end{itemize}
					\end{proof}
					
					\begin{cor}
						The BMR freeness conjecture holds for the generic Hecke algebra $H_{G_{13}}$.
					\end{cor}
					\begin{proof}
						By theorem \ref{thm13} we have that $H_{G_{13}}=U=\sum\limits_{k=0}^3\sum\limits_{i=1}^{12}(Rz^kv_i+Rz^kv_it)$. The result follows from proposition \ref{BMR PROP}, since by definition $H_{G_{13}}$ is generated as $R$-module by $|G_{13}|=96$ elements.
					\end{proof}		
			\subsubsection{\textbf{The case of} $\mathbf{G_{14}}$}
			Let $R=\ZZ[u_{s,i}^{\pm},u_{t,j}^{\pm}]_{\substack{1\leq i\leq 2 \\1\leq j\leq 4}}$ and let $H_{G_{14}}=\langle s,t\;|\; stststst=tstststs,\prod\limits_{i=1}^{2}(s-u_{s,i})=\prod\limits_{j=1}^{3}(t-u_{t,j})=0\rangle$ be the generic Hecke algebra associated to $G_{14}$. Let $u_1$ be the subalgebra of $H_{G_{14}}$ generated by $s$ and $u_2$ the subalgebra of $H_{G_{14}}$ generated by $t$. We recall that that $z:=(st)^4=(ts)^4$  generates the center of the associated complex braid group and that $|Z(G_{14})|=6$.
			We set $U=\sum\limits_{k=0}^{5}(z^ku_1u_2u_1u_2+z^ku_1tst^{-1}su_2).$
			By the definition of $U$, we have the following remark.
			
			\begin{rem}
				$Uu_2 \subset U$.
				\label{rem14}
			\end{rem}
			From now on, we will underline the elements that by definition belong to $U$.  Moreover, we will use directly remark \ref{rem14} and the fact that $u_1U\subset U$; this means that every time we have a power of $s$ in the beginning of an element or a power of $t$ at the end of it, we may ignore it. In order to remind that to the reader, we put a parenthesis around the part of the element we consider.

			Our goal is to prove that $H_{G_{14}}=U$ (theorem \ref{thm14}). Since $1\in U$, it is enough to prove that $U$ is a left-sided ideal of $H_{G_{14}}$. For this purpose, one may check that $sU$ and $tU$ are subsets of $U$. However, by the definition of $U$ and remark \ref{rem14}, we only have to prove that for every $k\in\{0,\dots,5\}$,
			$z^ktu_1u_2u_1$ and $z^ktu_1tst^{-1}s$ are subsets of $U$. In the following proposition we first prove this statement for a smaller range of the values of $k$. 
			
			\begin{prop}
				\mbox{}
				\vspace*{-\parsep}
				\vspace*{-\baselineskip}\\
				\begin{itemize}[leftmargin=0.6cm]
					\item[(i)]For every $k\in\{0,\dots,4\}$, $z^ktu_1u_2u_1\subset U$.
					\item[(ii)]For every $k\in\{0,\dots,4\}$, $z^ku_2u_1tu_1\subset U+z^{k+1}tu_1u_2u_1u_2$. Therefore, for every $k\in\{0,\dots,3\}$, $z^ku_2u_1tu_1 \subset U$.
					\item[(iii)]For every $k\in\{1,\dots,5\}$, $z^ku_2u_1t^{-1}u_1\subset U$.
					\item[(iv)]For every $k\in\{1,\dots,4\}$, $z^ku_2u_1u_2u_1\subset U+z^{k+1}tu_1u_2u_1u_2$. Therefore, for evey $k\in\{1,\dots,3\}$, $z^ku_2u_1u_2u_1\subset U$.
					
				\end{itemize}
				\label{pr14}
			\end{prop}
			
			\begin{proof}
				\mbox{}
				\vspace*{-\parsep}
				\vspace*{-\baselineskip}\\
				\begin{itemize}	[leftmargin=0.6cm]
					\item [(i)]	
					$z^ktu_1u_2u_1=z^kt(R+Rs)(R+Rt+Rt^{-1})(R+Rs)\subset \underline{z^ku_2u_1u_2}+\underline{Rz^ktst^{-1}s}+Rz^ktsts\subset U+Rz^k(ts)^4s^{-1}t^{-1}s^{-1}t^{-1}\subset U+\underline{z^{k+1}u_1u_2u_1u_2}\subset U.$
					
					\item[(ii)] We notice that $z^ku_2u_1tu_1=z^k(R+Rt+Rt^2)u_1tu_1\subset
					\underline{z^ku_1u_2u_1}+z^ktu_1u_2u_1+
					z^kt^2u_1tu_1$. However, $z^ktu_1u_2u_1\subset U$, by (i). Therefore, we have to prove that $z^kt^2u_1tu_1$ is a subset of $U$. Indeed, $z^kt^2u_1tu_1=
					z^kt^2(R+Rs)t(R+Rs)\subset U+\underline{z^ku_2u_1u_2}+Rz^kt^2sts\subset
					Rz^kt(ts)^4s^{-1}t^{-1}s^{-1}t^{-1}\subset U+
					(z^{k+1}tu_1u_2u_1)u_2$. The result follows from (i).

					\item[(iii)]We expand $u_2$ as $R+Rt+Rt^{-1}$ and we have that
					$z^ku_2u_1t^{-1}u_1\subset \underline{z^ku_1u_2u_1}+z^ktu_1t^{-1}u_1+
					z^kt^{-1}u_1t^{-1}u_1\subset U+z^kt(R+Rs)t^{-1}(R+Rs)+
					z^kt^{-1}(R+Rs^{-1})t^{-1}(R+Rs^{-1})\subset
					U+\underline{z^ku_2u_1u_2}+
					\underline{Rz^ktst^{-1}s}+
					Rz^kt^{-1}s^{-1}t^{-1}s^{-1}\subset U+
					Rz^k(st)^{-4}stst\subset U+\underline{z^{k-1}u_1u_2u_1u_2}.$
					
					\item[(iv)] The result follows from the definition of U and from (ii) and (iii), since $z^ku_2u_1u_2u_1=z^ku_2u_1(R+Rt+Rt^{-1})u_1$.
					\qedhere
				\end{itemize}	
			\end{proof}
			To make it easier for the reader to follow the calculations, from now on we will double-underline the elements as described in the above proposition (proposition \ref{pr14}) and we will use directly the fact that these elements are inside $U$.
			We can now prove the following lemmas that lead us to the main theorem of this section.
			\begin{lem}
				For every $k\in\{3,4\}$, $z^ktsu_2sts\subset
				U+z^{k+1}u_1tu_1u_2u_1u_2$.
				
				\label{lm14}
			\end{lem}
			\begin{proof}We have:
				\\
				$\hspace*{-0.2cm}\small{\begin{array}[t]{lcl}
					z^ktsu_2sts&=&z^kts(R+Rt+Rt^{-1})sts\\
					
					&\subset& \underline{\underline{z^ktu_1u_2u_1}}+Rz^k(ts)^4s^{-1}t^{-1}+z^ktst^{-1}sts\\
					&\subset&U+\underline{z^{k+1}u_1u_2}+z^kt(R+Rs^{-1})t^{-1}(R+Rs^{-1})t(R+Rs^{-1})\\
					&\subset&U+\underline{z^ku_1u_2u_1}+\underline{\underline{(z^ktu_1u_2u_1)t}}+
					Rz^kts^{-1}t^{-1}s^{-1}ts^{-1}\\
					&\subset&U+Rz^kts^{-1}t^{-1}s^{-1}(R+Rt^{-1}+Rt^{-2})s^{-1}\\
					&\subset&U+\underline{\underline{z^ktu_1u_2u_1}}+Rz^kt^2(st)^{-4}st+
					Rz^kt^2(st)^{-4                                                 }stst^{-1}s^{-1}\\
					&\subset&U+\underline{z^{k-1}u_2u_1u_2}+Rz^{k-1}t^2st(R+Rs^{-1})t^{-1}s^{-1}\\
					&\subset&U+\underline{z^{k-1}u_2}+Rz^{k-1}t^2st^2(st)^{-4}stst\\
					&\subset&U+Rz^{k-2}(R+Rt+Rt^{-1})st^2stst\\
					&\subset&U+\underline{\underline{z^{k-2}s(u_2u_1tu_1)t}}+Rz^{k-2}tst(ts)^4s^{-1}t^{-1}s^{-1}+
					Rz^{k-2}t^{-1}s(R+Rt+Rt^{-1})stst\\
					&\subset&U+Rz^{k-1}tst(R+Rs)t^{-1}s^{-1}+\underline{\underline{(z^{k-2}u_2u_1tu_1)t}}+Rz^{k-2}t^{-2}(ts)^4s^{-1}+\\&&+Rz^{k-2}t^{-1}(R+Rs^{-1})t^{-1}(R+Rs^{-1})tst\\
					&\subset&U+\underline{(z^{k-1}+z^{k-2})u_2u_1u_2}+Rz^{k-1}(ts)^4s^{-1}t^{-1}s^{-1}t^{-2}s^{-1}+
					\underline{\underline{(z^{k-2}u_2u_1tu_1)t}}+\\&&+Rz^{k-2}t^{-1}s^{-1}t^{-1}s^{-1}tst\\
					&\subset&U+z^ku_1u_2u_1u_2u_1+
					Rz^{k-2}(st)^{-4}stst^2st\\
					&\stackrel{\ref{pr14}(iv)}{\subset}&U+z^{k+1}u_1tu_1u_2u_1u_2+\underline{\underline{z^{k-3}s(tu_1u_2u_1)t}}.
					\end{array}}$
				
			\end{proof}
			\begin{lem}
				$z^kt^{-2}s^{-1}t^{-2}s^{-1}\in U$.
				\label{lem114}
			\end{lem}
			\begin{proof}We have:
				\\
				$\hspace*{-0.2cm}\small{\begin{array}[t]{lcl}
					z^5t^{-2}s^{-1}t^{-2}s^{-1}&=&z^5t^{-1}(st)^{-4}ststst^{-1}s^{-1}\\
					&\in&z^4t^{-1}stst(R+Rs^{-1})t^{-1}s^{-1}\\

					\end{array}}$\\
				
				$\hspace*{-0.55cm}\small{\begin{array}[t]{lcl}
					\phantom{z^5t^{-2}s^{-1}t^{-2}s^{-1}}
					
					&\in&\underline{z^4u_2u_1u_2}+Rz^4t^{-1}sts(R+Rt^{-1}+Rt^{-2})s^{-1}t^{-1}s^{-1}\\
					&\in&U+\underline{z^4u_2}+Rz^4t^{-1}st(R+Rs^{-1})t^{-1}s^{-1}t^{-1}s^{-1}+Rz^4t^{-1}stst^{-1}(st)^{-4}stst\\
					
					&\in&U+\underline{z^4u_2u_1}+Rz^4t^{-1}st^2(st)^{-4}st+
					Rz^3t^{-1}st(R+Rs^{-1})t^{-1}(R+Rs^{-1})tst\\
					&\in&U+\underline{\underline{(z^3u_2u_1u_2u_1)u_2}}+
					Rz^3t^{-1}sts^{-1}t^{-1}s^{-1}tst\\
					&\in&U+
					Rz^3t^{-1}(R+Rs^{-1})ts^{-1}t^{-1}s^{-1}tst\\
					&\in&U+\underline{\underline{u_1(z^3u_2u_1u_2u_1)u_2}}+Rz^3t^{-1}s^{-1}ts^{-1}t^{-1}s^{-1}tst\\
					&\in&U+
					Rz^3t^{-1}s^{-1}t^2(st)^{-4}stst^2st\\
					&\in&U+Rz^2t^{-1}s^{-1}(R+Rt+Rt^{-1})stst^2st\\
					&\in&U+\underline{z^2u_1u_2u_1u_2}+Rz^2t^{-1}s^{-2}(st)^4t^{-1}s^{-1}tst+
					z^2t^{-1}s^{-1}t^{-1}(R+Rs^{-1})tst^2st\\
					&\in&U+Rz^3t^{-1}(R+Rs^{-1})t^{-1}s^{-1}tst+\underline{z^2u_2u_1u_2}+
					Rz^2(st)^{-4}stst^2st^2st\\
					&\in&U+\underline{\underline{(z^3u_2u_1tu_1)t}}+Rz^3(st)^{-4}stst^2st+
					Rzstst^2s(R+Rt+Rt^{-1})st\\
					&\in&U+\underline{\underline{s\big((z+z^2)tu_1u_2u_1\big)t}}+Rzstst(ts)^4s^{-1}t^{-1}s^{-1}t^{-1}+
					Rzsts(R+Rt+Rt^{-1})st^{-1}st\\
					&\in&U+Rz^2stst(R+Rs)t^{-1}s^{-1}t^{-1}+\underline{\underline{s(ztu_1u_2u_1)t}}+Rz(st)^4t^{-1}s^{-1}t^{-2}st+\\&&+Rzstst^{-1}st^{-1}st\\
					&\in&U+\underline{z^2u_1}+Rz^2(st)^4t^{-1}s^{-1}t^{-2}s^{-1}t^{-1}+
					\underline{\underline{(z^2u_2u_1u_2u_1)t}}+\\&&+
					Rzst(R+Rs^{-1})t^{-1}(R+Rs^{-1})t^{-1}st\\

					&\in&U+	\underline{\underline{(z^3u_2u_1u_2u_1)t^{-1}}}+\underline{zu_1u_2u_1u_2}+
					\underline{\underline{s(z^3tu_1u_2u_1)t}}+Rzsts^{-1}t^{-1}s^{-1}t^{-1}st\\
					&\in&U+zsts^{-1}t^{-1}s^{-1}t^{-1}(R+Rs^{-1})t\\
					&\in& U+\underline{\underline{s(ztu_1u_2u_1)}}+Rzst^2(st)^{-4}st^2\\
					&\in& U+\underline{u_1u_2u_1u_2}.\phantom{==================================.}
					\qedhere
					\end{array}}$
			\end{proof}
			\begin{lem}
				For every $k\in\{0,\dots,5\}$, $z^ktu_1u_2u_1\subset U$.
				\label{cor14}
			\end{lem}
			\begin{proof}
				By proposition \ref{pr14} (i), we only have to prove that $z^5tu_1u_2u_1\subset U$. We have:
				$z^5tu_1u_2u_1=z^5t(R+Rs^{-1})(R+Rt^{-1}+Rt^{-2})u_1\subset \underline{z^5u_2u_1}+
				\underline{\underline{z^5u_2u_1t^{-1}u_1}}+z^5ts^{-1}t^{-2}u_1\subset U+z^5ts^{-1}t^{-2}(R+Rs^{-1})\subset U+\underline{z^5u_2u_1u_2}+Rz^5ts^{-1}t^{-2}s^{-1}.$
				We expand $t$ as a linear combination of 1, $t^{-1}$ and $t^{-2}$ and by the definition of $U$ and lemma \ref{lem114}, we only have to prove that $z^5t^{-1}s^{-1}t^{-2}s^{-1}\in U$. Indeed, we have:
				$z^5t^{-1}s^{-1}t^{-2}s^{-1}=z^5(st)^{-4}
				ststst^{-1}s^{-1}\in z^4stst(R+Rs^{-1})t^{-1}s^{-1}
				\subset \underline{z^4u_1u_2}+z^4stst^2(st)^{-4}stst\subset U+s(z^3tsu_2sts)t$.
				We use lemma \ref{lm14} and we have that $z^3tsu_2sts\subset U+z^4u_1tu_1u_2u_1u_2$. The result follows from proposition \ref{pr14}(i). 
			\end{proof}
			\begin{thm}
				$H_{G_{14}}=U$.
				\label{thm14}
			\end{thm}
			\begin{proof}
				As we explained in the beginning of this section, it is enough to prove that, for every $k\in\{0,\dots,5\}$, $z^ktu_1u_2u_1$ and $z^ktu_1tst^{-1}s$ are subsets of $U$. The first part is exactly what we proved in lemma \ref{cor14}. It remains to prove the second one. 
				Since $z^ktu_1tst^{-1}s=z^kt(R+Rs)tst^{-1}s$, we must prove that for every $k\in\{0,\dots,5\}$, the elements $z^kt^2st^{-1}s$ and $z^ktstst^{-1}s$ are inside $U$. 
				We distinguish the following cases:
				\begin{itemize}[leftmargin=*]
					\item \underline{The element $z^kt^2st^{-1}s$}:
					By proposition \ref{pr14} (iii), we only have to prove the case where $k=0$. \\ \\
					$\hspace*{-0.2cm}\small{\begin{array}{lcl}
						t^2st^{-1}s&\in&t^2s(R+Rt+Rt^2)s\\
						&\in& \underline{u_2u_1}+\underline{\underline{u_2u_1tu_1}}+t(ts)^4s^{-1}t^{-1}s^{-1}t^{-1}s^{-1}ts\\
						&\stackrel{\phantom{\ref{pr14}(iv)}}{\in}&U+zts^{-1}t^{-1}s^{-1}t^{-1}(R+Rs)ts\\
						&\in&U+\underline{zu_2u_1u_2}+Rzts^{-1}t^{-1}s^{-1}(R+Rt+Rt^2)sts\\
						&\in&U+\underline{zu_2}+Rzts^{-1}t^{-1}(R+Rs)tsts+Rzts^{-1}t^{-1}(R+Rs)t^2sts\\
						&\in& U+\underline{zu_2u_1}+Rzts^{-1}t^{-2}(ts)^4s^{-1}t^{-1}+
						Rzts^{-2}(st)^4t^{-1}s^{-1}t^{-1}+\\&&+
						Rzt(R+Rs)t^{-1}st^2sts\\
						&\in&U+\underline{\underline{(z^2tu_1u_2u_1)t^{-1}}}+\underline{\underline{s(zu_2u_1tu_1)}}+Rzts(R+Rt+Rt^2)st^2sts\\
						&\in&U+Rzts^2t^2sts+Rz(ts)^4s^{-1}t^{-1}s^{-1}tsts+Rztst^2st(ts)^4s^{-1}t^{-1}s^{-1}t^{-1}\\
						&\in& U+Rzt(R+Rs)t^2sts+Rz^2s^{-1}t^{-1}(R+Rs)tsts+Rz^2tst^2st(R+Rs)t^{-1}s^{-1}t^{-1}\\
						&\in& U+\underline{\underline{zu_2u_1tu_1}}+Rztst(ts)^4s^{-1}t^{-1}s^{-1}t^{-1}+
						\underline{z^2u_2u_1}+
						Rz^2s^{-1}t^{-2}(ts)^4s^{-1}t^{-1}+\\&&+\underline{z^2u_2u_1u_2}+
						Rz^2tst(ts)^4s^{-1}t^{-1}s^{-1}t^{-2}s^{-1}\\
						&\in&U+Rz^2tst(R+Rs)t^{-1}s^{-1}t^{-1}+\underline{z^3u_1u_2u_1u_2}+
						z^3tsts^{-1}t^{-1}s^{-1}(R+Rt^{-1}+Rt)s^{-1}\\
						&\in& U+\underline{z^2u_1}+Rz^2(ts)^4s^{-1}t^{-1}s^{-1}t^{-2}s^{-1}t^{-1}+
						Rz^3tsts^{-1}t^{-1}s^{-2}+
						
						Rz^3tst^2(st)^{-4}st+\\&&+
						Rz^3tst(R+Rs)t^{-1}s^{-1}ts^{-1}\\
						
						\end{array}}$\\

					$\hspace*{-0.55cm}\small{\begin{array}{lcl}
						\phantom{t^2st^{-1}s}
						
						&\in&U+\underline{\underline{s^{-1}(z^3u_2u_1u_2u_1)t^{-1}}}+Rz^3tst(R+Rs)t^{-1}s^{-2}+\underline{\underline{(z^2tu_1u_2u_1)t}}+
						\underline{\underline{z^3u_2u_1}}+\\&&+Rz^3tstst^{-1}(R+Rs)ts^{-1}\\
						&\in&U+\underline{z^3u_2u_1}+Rz^3(ts)^4s^{-1}t^{-1}s^{-1}t^{-2}s^{-2}+
						\underline{z^3u_2u_1u_2}+Rz^3(ts)^4s^{-1}t^{-1}s^{-1}t^{-2}sts^{-1}\\
						&\in&U+s^{-1}(z^4u_2u_1u_2u_1)+Rz^4s^{-1}t^{-1}s^{-1}(R+Rt^{-1}+Rt)sts^{-1}\\
						&\stackrel{\ref{pr14}(iv)}{\in}&U+s^{-1}z^5(tu_1u_2u_1)u_2+\underline{z^4u_1}+
						Rz^4(ts)^{-4}tsts^2ts^{-1}+
						Rz^4s^{-1}t^{-1}(R+Rs)tsts^{-1}\\
						&\stackrel{\ref{cor14}}{\in}&U+Rz^3tst(R+Rs)ts^{-1}+\underline{z^4u_2u_1}+
						Rz^4s^{-1}t^{-2}(ts)^4s^{-1}t^{-1}s^{-2}\\
						&\in&U+\underline{\underline{z^3tu_1u_2u_1}}+Rz^3(ts)^4s^{-1}t^{-1}s^{-2}+
						s^{-1}(z^4u_2u_1u_2u_1)\\
						&\stackrel{\ref{cor14}}{\in}&U+\underline{z^4u_1u_2u_1}+s^{-1}(z^5tu_1u_2u_1)u_2\\
						&\stackrel{\ref{cor14}}{\in}&U.

						\end{array}}$\\
					\item \underline{The element $z^ktstst^{-1}s$}: 
					For $k\in\{0,\dots,3\}$, we have $z^ktstst^{-1}s=z^k(ts)^4s^{-1}t^{-1}s^{-1}t^{-2}s\in s^{-1}(z^{k+1}u_2u_1u_2u_1)$. However, by 
					\ref{pr14}(iv) we have that $z^{k+1}u_2u_1u_2u_1\subset U+(z^{k+2}tu_1u_2
					u_1)u_2\stackrel{\ref{cor14}}{\subset}U.$ It remains to prove the case where $k\in\{4,5\}$. We have:
					$z^ktstst^{-1}s\in z^ktst(R+Rs^{-1})t^{-1}s\subset 
					\underline{z^ku_2u_1}+z^ktsts^{-1}t^{-1}(R+Rs^{-1})\subset 
					U+(Rz^ktu_1u_2u_1)t^{-1}+Rz^ktsts^{-1}t^{-1}s^{-1}$. However, by lemma \ref{cor14}
					we only have to prove that $z^ktsts^{-1}t^{-1}s^{-1}\in U$. Indeed, $z^ktsts^{-1}t^{-1}s^{-1}=
					z^ktst^2(st)^{-4}stst=(z^{k-1}tst^2sts)t
					\stackrel{\ref{lm14}}{\subset}
					U+u_1(z^ktu_1u_2u_1)u_2$. The result follows from lemma \ref{cor14}.
					\qedhere
				\end{itemize}
			\end{proof}
			
			\begin{cor}
				The BMR freeness conjecture holds for the generic Hecke algebra $H_{G_{14}}$.
			\end{cor}
			\begin{proof}
				By theorem \ref{thm14} we have that $H_{G_{14}}=U=\sum\limits_{k=0}^{5}(z^ku_1u_2+z^ku_1tsu_2+z^ku_1t^{-1}su_2+z^ku_1tst^{-1}su_2)$. The result follows from proposition \ref{BMR PROP}, since $H_{G_{14}}$ is generated as a left $u_1$-module by 72 elements and, hence, as
				$R$-module by $|G_{14}|=144$ elements (recall that $u_1$ is generated as $R$ module by 2 elements).
			\end{proof}		
		\subsubsection{\textbf{The case of} $\mathbf{G_{15}}$}
		Let $R=\ZZ[u_{s,i}^{\pm},u_{t,j}^{\pm},u_{u,l}^{\pm}]_{\substack{1\leq i,j\leq 2 \\1\leq l\leq 3}}$ and let $$H_{G_{15}}=\langle s,t,u\;|\; stu=tus,ustut=stutu,\prod\limits_{i=1}^{2}(s-u_{s,i})=\prod\limits_{j=1}^{2}(t-u_{t,j})=\prod\limits_{l=1}^{3}(u-u_{u,l})=0\rangle$$ be the generic Hecke algebra associated to $G_{15}$. Let $u_1$ be the subalgebra of $H_{G_{15}}$ generated by $s$, $u_2$ the subalgebra of $H_{G_{15}}$ generated by $t$ and $u_3$ the subalgebra of $H_{G_{15}}$ generated by $u$. We recall that $z:=stutu=ustut=tustu=tutus=utust$  generates the center of the associated complex braid group and that $|Z(G_{15})|=12$.
		We set $U=\sum\limits_{k=0}^{11}z^ku_3u_2u_1u_2u_1$. 
		By the definition of $U$ we have the following remark.
		
		\begin{rem}
			$Uu_1 \subset U$.
			\label{rem15}
		\end{rem}
		From now on, we will underline the elements that belong to $U$ by definition.
		Our goal is to prove that $H_{G_{15}}=U$ (theorem \ref{thm 15}). Since $1\in U$, it will be sufficient to prove that $U$ is a left-sided ideal of $H_{G_{15}}$. For this purpose, one must check that $sU$, $tU$ and $uU$ are subsets of $U$. The following proposition states that it is enough to prove $tU\subset U$.
		\begin{prop}
			If $tU\subset U$, then $H_{G_{15}}=U$.
			\label{prrrr1}
		\end{prop}
		\begin{proof}
			As we explained above, it is enough to prove that $sU$, $tU$ and $uU$ are subsets of $U$. However, by hypothesis and by the definition of $U$, we
			can restrict ourselves to proving that $sU\subset U$. We recall that $z=stutu$. Therefore, $s=zu^{-1}t^{-1}u^{-1}t^{-1}$ and $s^{-1}=z^{-1}tutu$. 
			We notice that $$U=
			\sum\limits_{k=0}^{10}z^ku_3u_2u_1u_2u_1+z^{11}u_3u_2u_1u_2u_1.$$
			Hence, $\small{\begin{array}[t]{lcl}sU&\subset& \sum\limits_{k=0}^{10}z^ksu_3u_2u_1u_2u_1+
				z^{11}su_3u_2u_1u_2u_1\\
				&\subset& \sum\limits_{k=0}^{10}z^{k+1}u^{-1}t^{-1}u^{-1}t^{-1}u_3u_2u_1u_2u_1+z^{11}(R+Rs^{-1})u_3u_2u_1u_2u_1\\
				&\subset& u^{-1}t^{-1}u^{-1}t^{-1}\sum\limits_{k=0}^{10}z^{k+1}u_3u_2u_1u_2u_1+z^{11}u_3u_2u_1u_2u_1+(z^{11}s^{-1})u_3u_2u_1u_2u_1\\
				&\subset&u_3u_2u_3u_2U+
				z^{10}tutz^ku_3u_2u_1u_2u_1\\
				&\subset&U+u_3u_2u_3u_2U+
				tut(z^{10}z^ku_3u_2u_1u_2u_1)\\&\subset& U+u_3u_2u_3u_2U.
				\end{array}}$\\\\
			Since $tU\subset U$ we have $u_2U\subset U$ (recall that $u_2=R+Rt$) and, by the definition of $U$, we also have $u_3U\subset U$. The result then is obvious.
		\end{proof}
		
		A first step to prove our main theorem is analogous to \ref{rem15} (see proposition \ref{prrr15}). For this purpose, we first prove some preliminary results.

		\begin{lem}
			\mbox{}
			\vspace*{-\parsep}
			\vspace*{-\baselineskip}\\
			\begin{itemize}[leftmargin=0.6cm]
				\item[(i)] For every $k\in\{0,...,11\}$, $z^ku_3u_1u^{-1}u_2\subset U$.
				\item[(ii)] For every $k\in\{3,...,11\}$, $z^ku_3u_1u^{-2}t^{-1}\subset U$.
				\item[(iii)] For every $k\in\{3,...,11\}$, $z^ku_3u_1u_3t^{-1}\subset U$.
			\end{itemize}
			\label{sutt}
		\end{lem}
		\begin{proof}
			Since $u_3=R+Ru^{-1}+Ru^{-2}$, (iii) follows from (i) and (ii) and from the definition of $U$. For (i) we have: $z^ku_3u_1u^{-1}u_2=z^ku_3(R+Rs^{-1})u^{-1}u_2\subset \underline{z^ku_3u_2}+z^ku_3(s^{-1}u^{-1}t^{-1})u_2\subset U+z^ku_3(u^{-1}t^{-1}s^{-1})u_2\subset U+\underline{z^ku_3u_2u_1u_2}\subset U.$
			It remains to prove (ii). We have:
			\\\\$\hspace*{-0.2cm}\small{\begin{array}{lcl}
				z^ku_3u_1u^{-2}t^{-1}&=&z^ku_3(R+Rs^{-1})u^{-2}t^{-1}\\
				&\subset&\underline{z^ku_3t^{-1}}+z^ku_3(s^{-1}u^{-1}t^{-1}u^{-1}t^{-1})tutu^{-1}t^{-1}\\
				&\subset&U+z^{k-1}u_3tu(R+Rt^{-1})u^{-1}t^{-1}\\
				&\subset&U+\underline{z^{k-1}u_3}+z^{k-1}u_3tu^2(u^{-1}t^{-1}u^{-1}t^{-1}s^{-1})s\\
				&\subset&U+z^{k-2}u_3t(R+Ru+Ru^{-1})s\\
				&\subset&U+\underline{z^{k-2}u_3ts}+z^{k-2}u_3(tustu)u^{-1}t^{-1}+z^{k-2}u_3(R+Rt^{-1})u^{-1}s\\
				&\subset&U+\underline{z^{k-1}u_3t^{-1}}+\underline{z^{k-2}u_3s}+z^{k-2}u_3(u^{-1}t^{-1}u^{-1}t^{-1}s^{-1})st\\
				&\subset&U+\underline{z^{k-3}u_3st}.\phantom{====================================..}\qedhere
				\end{array}}$
			
			\qedhere
		\end{proof}
		
		\begin{lem}
			\mbox{}
			\vspace*{-\parsep}
			\vspace*{-\baselineskip}\\
			\begin{itemize}[leftmargin=0.8cm]
				\item[(i)] For every $k\in\{0,...,10\}$, $z^ku_3u_2uu_2\subset U$.
				\item[(ii)] For every $k\in\{1,...,11\}$, $z^ku_3u_2u^{-1}u_2\subset U$.
				\item[(iii)] For every $k\in\{1,...,10\}$, $z^ku_3u_2u_3u_2\subset U$.
			\end{itemize}
			\label{tutt}
		\end{lem}
		\begin{proof}
			Since $u_3=R+Ru^{-1}+Ru^{-2}$, (iii) follows from (i) and (ii) and from the definition of $U$. For (i) we have: $z^ku_3u_2uu_2=z^ku_3(R+Rt)uu_2\subset \underline{z^ku_3u_2}+z
			^ku_3(utust)t^{-1}s^{-1}u_2\subset U+\underline{z^{k+1}u_3t^{-1}s^{-1}u_2}\subset U.$
			For (ii), we use similar kind of calculations: $z^ku_3u_2u^{-1}u_2=z^ku_3(R+Rt^{-1})u^{-1}u_2\subset \underline{z^ku_3u_2}+z
			^ku_3(u^{-1}t^{-1}u^{-1}t^{-1}s^{-1})su_2\subset U+\underline{z^{k-1}u_3su_2}.$
		\end{proof}
		\begin{lem}
			\mbox{}
			\vspace*{-\parsep}
			\vspace*{-\baselineskip}\\
			\begin{itemize}[leftmargin=0.6cm]
				\item[(i)] For every $k\in\{0,...,10\}$, $z^ku_3u_1tu\subset U$.
				\item[(ii)] For every $k\in\{0,...,9\}$, $z^ku_3u_1tu^2\subset U$.
				\item[(iii)] For every $k\in\{0,...,9\}$, $z^ku_3u_1tu_3\subset U$.
			\end{itemize}
			\label{stuu}
		\end{lem}
		\begin{proof}
			Since $u_3=R+Ru+Ru^{2}$, (iii) follows from (i) and (ii) and from the definition of $U$. For (i) we have: $z^ku_3u_1tu=z^ku_3(R+Rs)tu\subset z^ku_3tu+z^ku_3(stutu)u^{-1}t^{-1}\stackrel{\ref{tutt}(i)}{\subset}U+\underline{z^{k+1}u_3t^{-1}}\subset U.$ Similarly, for (ii): $z^ku_3u_1tu^2=z^ku_3(R+Rs)tu^2\subset z^ku_3tu^2+z^ku_3(stutu)u^{-1}t^{-1}u\stackrel{\ref{tutt}(iii)}{\subset}U+z^{k+1}u_3t^{-1}u\stackrel{\ref{tutt}(i)}{\subset}U.$
		\end{proof}
		
		\begin{lem}
			\mbox{}
			\vspace*{-\parsep}
			\vspace*{-\baselineskip}\\
			\begin{itemize}[leftmargin=0.6cm]
				\item[(i)] For every $k\in\{0,...,10\}$, $z^ku_3u_2uu_1\subset U$.
				\item[(ii)] For every $k\in\{1,...,11\}$, $z^ku_3u_2u^{-1}u_1\subset U$.
				\item[(iii)] For every $k\in\{1,...,10\}$, $z^ku_3u_2u_3u_1\subset U$.
			\end{itemize}
			\label{tuss}
		\end{lem}
		\begin{proof}
			Since $u_3=R+Ru+Ru^{-1}$, (iii) follows from (i) and (ii) and from the definition of $U$. For (i) we have: $z^ku_3u_2uu_1=z^ku_3(R+Rt)u(R+Rs)
			\subset \underline{z^ku_3u_1}+z^ktu+z^k(tus)\stackrel{\ref{tutt}(i)}{\subset}U+z^kstu\stackrel{\ref{stuu}(i)}{\subset}U.$ Similarly, for (ii): $z^ku_3u_2u^{-1}u_1=z^ku_3(R+Rt^{-1})u^{-1}u_1\subset \underline{z^ku_3u_1}+z^ku_3(u^{-1}t^{-1}u^{-1}t^{-1}s^{-1})stu_1\subset U+\underline{z^{k-1}u_3stu_1}.$
		\end{proof}

		\begin{lem}
			\mbox{}
			\vspace*{-\parsep}
			\vspace*{-\baselineskip}\\
			\begin{itemize}[leftmargin=0.6cm]
				\item[(i)] For every $k\in\{0,...,8\}$, $z^ku_3u_1uu_1\subset U$.
				\item[(ii)] For every $k\in\{0,...,11\}$, $z^ku_3u_1u^{-1}u_1\subset U$.
				\item[(iii)] For every $k\in\{0,...,8\}$, $z^ku_3u_1u_3u_1\subset U$.
			\end{itemize}
			\label{suu}
		\end{lem}
		\begin{proof} 
			By remark \ref{rem15} we can ignore the $u_1$ in the end.
			Moreover, since $u_3=R+Ru+Ru^{-1}$, (iii) follows from (i) and (ii). However, $z^ku_3u_1u^{-1}\subset z^ku_3u_1u^{-1}u_2$ and, hence, (ii) follows from lemma \ref{sutt} (i). Therefore, it will be sufficient to prove (i). We have:\\ \\
			$\hspace*{-0.2cm}\small{\begin{array}{lcl}
				z^ku_3u_1u&=&z^ku_3(R+Rs)u\\
				&\subset& \underline{z^ku_3}+z^ku_3(ustut)t^{-1}u^{-1}t^{-1}u\\
				&\subset& U+z^{k+1}u_3(R+Rt)u^{-1}(R+Rt)u\\
				&\subset& U+z^{k+1}u_3u_2u+\underline{z^{k+1}u_3}+\underline{z^{k+1}u_3t}+z^{k+1}u_3tu^{-1}tu\\
				&\stackrel{\ref{tutt}(i)}{\subset}&U+z^{k+1}u_3t(R+Ru+Ru^2)tu\\
				&\subset&U+z^{k+1}u_3u_2u+z^{k+1}u_3(tutus)s^{-1}+z^{k+1}u_3(utust)t^{-1}s^{-1}(utust)t^{-1}s^{-1}\\
				&\stackrel{\ref{tutt}(i)}{\subset}&U+\underline{z^{k+2}u_3s^{-1}}+\underline{z^{k+3}u_3t^{-1}s^{-1}t^{-1}s^{-1}}.\phantom{========================}\qedhere
				\end{array}}$
			
		\end{proof}
		To make it easier for the reader to follow the calculations, from now on we will double-underline the elements described in the above lemmas (lemmas \ref{sutt} - \ref{suu}) and we will use directly the fact that these elements are inside $U$.
		
		\begin{prop}
			\mbox{}
			\vspace*{-\parsep}
			\vspace*{-\baselineskip}\\
			\begin{itemize}[leftmargin=0.6cm]
				\item[(i)] For every $k\in\{2,\dots,11\}$, $z^ks^{-1}u^{-2}\in U+z^{k-3}u_3^{\times}stst$.
				\item[(ii)] For every $k\in\{0,\dots,11\}$, $z^ku_3u_1u_2u_1u_2\subset U$.
				\item[(iii)]For every $k\in\{0,\dots,11\}$, $z^ku_3u_1u_3\subset U$.
			\end{itemize}
			\label{ststt}
		\end{prop}
		\begin{proof}\mbox{}
			\vspace*{-\parsep}
			\vspace*{-\baselineskip}\\
			\begin{itemize}[leftmargin=0.6cm]
				\item[(i)]$\hspace*{-0.2cm}\small{\begin{array}[t]{lcl}
					z^ks^{-1}u^{-2}&=&z^k(s^{-1}u^{-1}t^{-1}u^{-1}t^{-1})tutu^{-1}\\
					&\in& z^{k-1}(R+R^{\times}t^{-1})u(R+R^{\times}t^{-1})u^{-1}\\
					&\in& \underline{z^{k-1}u_3}+\underline{z^{k-1}u_3t^{-1}}+\underline{\underline{z^{k-1}u_3u_2u^{-1}u_2}}+R^{\times}z^{k-1}t^{-1}u^2(u^{-1}t^{-1}u^{-1}t^{-1}s^{-1})st\\
					&\in& U+R^{\times}z^{k-2}t^{-1}(R+Ru+R^{\times}u^{-1})st\\
					&\in& U+\underline{z^{k-2}u_3t^{-1}st}+Rz^{k-2}t^{-1}(ustut)t^{-1}u^{-1}+R^{\times}z^{k-2}u(u^{-1}t^{-1}u^{-1}t^{-1}s^{-1})stst\\
					&\in& U+\underline{\underline{z^{k-1}u_3u_2u^{-1}u_2}}+z^{k-3}u_3^{\times}stst\\ &\subset&
					U+z^{k-3}u_3^{\times}stst.
					\end{array}}$
				\item[(ii)] For $k\in\{0,...,5\}$, we have:
				$z^ku_3u_1u_2u_1u_2=z^ku_3(R+Rs)(R+Rt)(R+Rs)(R+Rt)\subset  \underline{z^ku_3u_1u_2u_1}+\underline{z^ku_3u_2u_1u_2}+z^ku_3stst$. Therefore, by (i) we have that $z^ku_3u_1u_2u_1u_2\subset U+z^{k+3}u_3^{\times}s^{-1}u^{-2}$. The result follows from 
				\ref{suu}(iii). It remains to prove the case where $k\in\{6,\dots,11\}$. We have:
				\\\\
				$\hspace*{-0.2cm}\small{\begin{array}[t]{lcl}	z^ku_3u_1u_2u_1u_2&=&z^ku_3(R+Rs^{-1})(R+Rt^{-1})(R+Rs^{-1})(R+Rt^{-1})\\
					&\subset& \underline{z^ku_3u_1u_2u_1}+\underline{z^ku_3u_2u_1u_2}+z^ku_3s^{-1}t^{-1}s^{-1}t^{-1}\\
					&\subset&z^ku_3s^{-1}(t^{-1}s^{-1}u^{-1}t^{-1}u^{-1})utut^{-1}\\
					&\subset& z^{k-1}u_3s^{-1}(R+Ru^{-1}+Ru^{-2})tut^{-1}\\
					&\subset& z^{k-1}u_3(R+Rs)tut^{-1}+z^{k-1}u_3s^{-1}u^{-1}(R+Rt^{-1})ut^{-1}+\\&&+
					z^{k-1}u_3s^{-1}u^{-2}(R+Rt^{-1})ut^{-1}\\
					&\subset&\underline{\underline{z^{k-1}u_3u_2uu_2}}+z^{k-1}u_3(ustut)t^{-2}+
					\underline{z^{k-1}u_3s^{-1}t^{-1}}+\\&&+z^{k-1}u_3(s^{-1}u^{-1}t^{-1}u^{-1}t^{-1})tu^2t^{-1}+\underline{\underline{z^{k-1}u_3u_1u_3t^{-1}}}+\\&&+
					z^{k-1}u_3s^{-1}u^{-2}t^{-1}(R+Ru^{-1}+Ru^{-2})t^{-1}\\
					\end{array}}$\\

				$\hspace*{-0.4cm}\small{\begin{array}[t]{lcl}	\phantom{z^ku_3u_1u_2u_1u_2}

					&\subset&U+\underline{z^kt^{-2}}+\underline{\underline{z^{k-2}u_3u^2u_3u_2}}+z^{k-1}u_3s^{-1}u^{-2}t^{-2}+\\&&+z^{k-1}u_3(s^{-1}u^{-1}t^{-1})t(u^{-1}t^{-1}u^{-1}t^{-1}s^{-1})s+\\&&+z^{k-1}u_3(s^{-1}u^{-1}t^{-1}u^{-1}t^{-1})tut(u^{-1}t^{-1}u^{-1}t^{-1}s^{-1})stu^{-1}t^{-1}\\
					&\subset& U+z^{k-1}u_3s^{-1}u^{-2}(R+Rt^{-1})+\underline{z^{k-2}u_3t^{-1}s^{-1}ts}+\\&&+
					z^{k-3}u_3tu(R+Rt^{-1})stu^{-1}t^{-1}\\
					&\subset&U+z^{k-1}u_3s^{-1}u^{-2}+\underline{\underline{z^{k-1}u_3u_1u_3t^{-1}}}+z^{k-3}u_3(utust)u^{-1}t^{-1}+\\&&+z^{k-3}u_3tut^{-1}(R+Rs^{-1})tu^{-1}t^{-1}\\
					&\subset& U+z^{k-1}u_3s^{-1}u^{-2}
					+\underline{(z^{k-2}+z^{k-3})u_3u_2}+z^{k-3}u_3tut^{-1}s^{-1}tu^{-1}t^{-1}.
					\end{array}}$
				\\\\
				We notice now that $z^{k-3}u_3tut^{-1}s^{-1}tu^{-1}t^{-1}$ is a subset of $U$: \\\\
				$\hspace*{-0.2cm}\small{\begin{array}{lcl}
					z^{k-3}u_3tut^{-1}s^{-1}tu^{-1}t^{-1}	&\subset&z^{k-3}u_3tut^{-1}s^{-1}(R+Rt^{-1})u^{-1}t^{-1}\\
					&\subset& z^{k-3}tu^2(u^{-1}t^{-1}s^{-1}u^{-1}t^{-1})+
					\\&&+z^{k-3}u_3tu^2(u^{-1}t^{-1}s^{-1}u^{-1}t^{-1})tu^2(u^{-1}t^{-1}u^{-1}t^{-1}s^{-1})s\\
					&\subset& \underline{\underline{z^{k-4}u_3u_2u_3u_2}}+z^{k-5}u_3
					t(R+Ru+Ru^{-1})tu^2s\\
					&\subset&U+\underline{\underline{z^{k-5}u_3u_2u_3u_1}}+
					z^{k-5}u_3(tutus)s^{-1}us+\\&&+z^{k-5}u_3tu^{-1}t(R+Ru+Ru^{-1})s\\
					&\subset&U+\underline{\underline{z^{k-4}u_3u_1uu_1}}+z^{k-5}u_3tu^{-1}(R+Rt^{-1})(R+Rs^{-1})+\\&&+z^{k-5}u_3tu^{-2}(utust)t^{-1}+z^{k-5}u_3tu^{-1}(R+Rt^{-1})u^{-1}s\\
					&\subset&U+\underline{\underline{(z^{k-5}+z^{k-4})u_3u_2u_3u_2}}+\underline{\underline{z^{k-5}u_3u_2u_3u_1}}
					+z^{k-5}u_3tu^{-1}t^{-1}s^{-1}+\\&&+
					z^{k-5}u_3t(u^{-1}t^{-1}u^{-1}t^{-1}s^{-1})sts\\

					&\subset&U+z^{k-5}u_3t^2u(u^{-1}t^{-1}u^{-1}t^{-1}s^{-1})+\underline{z^{k-6}u_3tsts}\\
					&\subset&U+\underline{\underline{z^{k-6}u_3u_2uu_2}}.
					\end{array}}$\\
				Hence, \begin{equation}
				z^ku_3u_1u_2u_1u_2\subset U+z^{k-1}u_3s^{-1}u^{-2},\; k\in\{6,\dots,11\}.
				\label{pp}
				\end{equation} 
				For $k\in\{6,...,9\}$ we rewrite (\ref{pp}) and we have $z^ku_3u_1u_2u_1u_2\subset U+z^{k-1}u_3u_1u_3u_1$. Therefore, by lemma \ref{suu}(iii) we have $z^ku_3u_1u_2u_1u_2\subset U$. For $k\in\{10,11\}$ we use (i) and (\ref{pp}) becomes $z^ku_3u_1u_2u_1u_2\subset U+z^{k-4}u_3stst$. However, since $k-4\in\{6,7\}$, we can apply (\ref{pp})  and we have that $z^{k-4}u_3stst\subset U+z^{k-5}u_3s^{-1}u^{-2}$. The result follows from lemma \ref{suu}(iii).\\

				\item [(iii)] By lemma \ref{suu} (iii), it is enough to prove that for $k\in\{9,10,11\}$, $z^ku_1u_3\subset U$. We expand $u_1$ as $R+Rs^{-1}$ and $u_3$ as $R+Ru^{-1}+Ru^{-2}$ and we have 
				$z^ku_1u_3\subset \underline{z^ku_3}+z^ku_1u^{-1}+z^ku_1u^{-2}$. Hence, by lemma  \ref{suu}(ii) we only have to prove that
				$z^ku_3s^{-1}u^{-2}\subset U$ . However, by (i) we have $z^ku_3s^{-1}u^{-2}\subset U+z^{k-3}u_3stst$ and the result  follows directly from (ii).
				\qedhere
			\end{itemize}
			
		\end{proof}
		\begin{lem}
			\mbox{}
			\vspace*{-\parsep}
			\vspace*{-\baselineskip}\\
			\begin{itemize}[leftmargin=0.6cm]
				\item[(i)] For every $k\in\{3,\dots,8\}$, $z^ku_3tu^{-1}u_1u\subset U$.
				\item[(ii)] For every $k\in\{3,4\}$, $z^ku_3tu^{-1}u_1u^2\subset U$.
				\item[(iii)] For every $k\in\{5,\dots,8\}$, $z^ku_3tu^{-1}u_1u^{-1}\subset U$.
				\item[(iv)] For every $k\in\{3,\dots,8\}$, $z^ku_3tu^{-1}u_1u_3\subset U$.
			\end{itemize}
			\label{stuuuu}
		\end{lem}
		\begin{proof}
			\mbox{}
			\vspace*{-\parsep}
			\vspace*{-\baselineskip}\\
			\begin{itemize}[leftmargin=0.6cm]
				\item[(i)]$\hspace*{-0.2cm}\small{\begin{array}[t]{lcl}
					z^ku_3tu^{-1}u_1u^{\phantom{2}}&\subset&z^{k}u_3t(R+Ru+Ru^2)(R+Rs)u\\
					&\subset& U+\underline{\underline{z^{k}u_3u_1tu_3}}+z^{k}u_3tsu+z^{k}u_3(tus)u+z^{k}u_3tu^2su\\ 
					&\subset&U+z^{k}u_3(R+Rt^{-1})(R+Rs^{-1})u+z^{k}u_3stu+z^{k}u_3tu(ustut)t^{-1}u^{-1}t^{-1}u\\ 
					&\subset&U+\underline{\underline{z^{k}u_3u_1u_3u_1}}+\underline{\underline{z^{k}u_3u_2u_3u_1}}+
					z^{k}u_3(t^{-1}s^{-1}u^{-1}t^{-1}u^{-1})utu^2+\underline{\underline{z^{k}u_3u_1tu_3}}+\\&&+
					z^{k+1}u_3tu(R+Rt)u^{-1}t^{-1}u\\ 		&\subset&U+\underline{\underline{z^{k-1}u_3u_2u_3u_1}}+\underline{z^{k+1}u_3}+
					z^{k+1}u_3tut(R+Ru+Ru^2)t^{-1}u\\ 
					&\subset&U+
					
					\underline{\underline{z^{k+1}u_3u_2u_3u_2}}+
					z^{k+1}u_3(tutus)s^{-1}t^{-1}u+z^{k+1}u_3(tutus)s^{-1}ut^{-1}u\\
					
					&\subset&U+z^{k+2}u_3s^{-1}(R+Rt)u+z^{k+2}u_3s^{-1}u(R+Rt)u\\ 
					
					&\subset&U+z^{k+2}u_3u_1u_3+\underline{\underline{z^{k+2}u_3u_1tu}}+z^{k+2}u_3s^{-1}(utust)t^{-1}s^{-1}\\ 
					&\stackrel{\ref{ststt}(iii)}{\subset}&U+
					\underline{z^{k+3}u_3u_1u_2u_1}.
					\end{array}}$\\
				
				\item[(ii)]$\hspace*{-0.2cm}\small{\begin{array}[t]{lcl}
					z^ku_3tu^{-1}u_1u^2&\subset&z^{k}u_3t(R+Ru+Ru^2)(R+Rs)u^2\\
					&\stackrel{\phantom{\ref{ststt}(iii)}}{\subset}& U+\underline{\underline{z^{k}u_3u_1tu_3}}+z^{k}u_3tsu^2+z^{k}u_3(tus)u^2+z^{k}u_3tu^2su^2\\ 
					&\subset&U+z^{k}u_3(R+Rt^{-1})(R+Rs^{-1})u^2+z^{k}u_3stu^2+z^{k}u_3tu(ustut)t^{-1}u^{-1}t^{-1}u^2\\ 
					&\subset&U+\underline{\underline{z^{k}u_3u_1u_3u_1}}+\underline{\underline{z^{k}u_3u_2u_3u_1}}+
					z^{k}u_3(t^{-1}s^{-1}u^{-1}t^{-1}u^{-1})utu^3+\underline{\underline{z^{k}u_3u_1tu_3}}+\\&&+
					z^{k+1}u_3tu(R+Rt)u^{-1}t^{-1}u^2\\ 
					&\subset&U+\underline{\underline{z^{k-1}u_3u_2u_3u_1}}+\underline{z^{k+1}u_3}+
					z^{k+1}u_3tut(R+Ru+Ru^2)t^{-1}u^2\\ 
					
					&\subset&U+
					
					\underline{\underline{z^{k+1}u_3u_2u_3u_2}}+
					z^{k+1}u_3(tutus)s^{-1}t^{-1}u^2+z^{k+1}u_3(tutus)s^{-1}ut^{-1}u^2\\
					&\subset&U+z^{k+2}u_3s^{-1}(R+Rt)u^2+z^{k+2}u_3s^{-1}u(R+Rt)u^2\\ 
					&\subset&U+z^{k+2}u_3u_1u_3+\underline{\underline{z^{k+2}u_3u_1tu_3}}+z^{k+2}u_3s^{-1}(utust)t^{-1}s^{-1}u\\ 
					&\stackrel{\ref{ststt}(iii)}{\subset}&U+
					z^{k+3}u_3(R+Rs)t^{-1}s^{-1}u\\
					&\subset&z^{k+3}(u^{-1}t^{-1}s^{-1}u^{-1}t^{-1})tu^2+z^{k+3}u_3s(R+Rt)(R+Rs)u\\

					&\subset&U+\underline{\underline{z^{k+2}u_3u_1tu_3}}+\underline{\underline{z^{k+3}u_3u_1u_3u_1}}+\underline{\underline{z^{k+3}u_3u_1tu_3}}+z^{k+3}u_3stsu\\
					&\subset&U+z^{k+3}u_3(ustut)t^{-1}u^{-1}su\\
					&\subset&U+z^{k+4}u_3(R+Rt)u^{-1}su\\
					&\stackrel{\phantom{\ref{ststt}(iii)}}{\subset}&U+z^{k+4}u_3u_1u_3+z^{k+4}tu^{-1}u_1u.
					\end{array}}$
				\\\\
				The result follows from proposition \ref{ststt}(iii) and from (i).\\
				
				\item[(iii)]$\small{\begin{array}[t]{lcl}
					z^ku_3tu^{-1}u_1u^{-1}&\subset&z^ku_3(R+Rt^{-1})u^{-1}(R+Rs^{-1})u^{-1}\\
					&\subset&\underline{\underline{z^ku_3u_1u_3u_1}}+\underline{\underline{z^ku_3u_2u_3u_2}}+z^ku_3t^{-1}u^{-1}s^{-1}u^{-1}\\
					&\subset&U+z^ku_3(u^{-1}t^{-1}u^{-1}t^{-1}s^{-1})sts^{-1}u^{-1}\\
					&\subset&U+z^{k-1}u_3s(R+Rt^{-1})s^{-1}u^{-1}\\
					&\subset&\underline{\underline{z^{k-1}u_3u_1u_3u_1}}+z^{k-1}u_3s(t^{-1}s^{-1}u^{-1}t^{-1}u^{-1})ut\\
					&\subset&U+z^{k-2}u_3su(R+Rt^{-1})\\
					&\subset&\underline{\underline{z^{k-2}u_3u_1u_3u_1}}+\underline{\underline{z^{k-2}u_3u_1u_3t^{-1}}}.
					
					\end{array}}$\\
				\item[(iv)] For $k\in\{3,4\}$ we have $z^ku_3tu^{-1}u_1u_3\subset z^ku_3tu^{-1}u_1(R+Ru+Ru^2) \stackrel{(i),(ii)}{\subset}U+ \underline{\underline{z^ku_3u_2u^{-1}u_1}}\subset U$.
				Similarly, for $k\in\{5,\dots,8\}$ we have $z^ku_3tu^{-1}u_1u_3\subset z^ku_3tu^{-1}u_1(R+Ru+Ru^{-1}) \stackrel{(i),(iii)}{\subset}U+ \underline{\underline{z^ku_3u_2u^{-1}u_1}}$.
				\qedhere
			\end{itemize}
		\end{proof}
		
		\begin{prop}$Uu_2\subset U$.
			\label{prrr15}
		\end{prop}
		
		\begin{proof}
			By the definition of $U$ and the fact that $u_2=R+Rt$, we have to prove that $z^ku_3u_2u_1u_2u_1t\subset U$, for every $k\in\{0,...,11\}$. If we expand $u_1$ as $R+Rs$ and $u_2$ as $R+Rt$ we notice that $z^ku_3u_2u_1u_2u_1t\subset z^ku_3u_2u_1u_2u_1+z^ku_3u_1u_2u_1u_2+z^ku_3tstst$. Therefore, by the definition of $U$ and by proposition \ref{ststt}(ii), 
			we only have to prove  $z^ku_3tstst\subset U$, for every $k\in\{0,\dots,11\}$. We distinguish the following cases:
			\begin{itemize}[leftmargin=*]
				\item \underline{$k\in\{0,\dots,5\}$}:
				\\
				$\hspace*{-0.2cm}\small{\begin{array}[t]{lcl}
					z^ku_3tstst&=&z^ku_3tst(stutu)u^{-1}t^{-1}u^{-1}\\
					&\subset&z^{k+1}u_3tst(R+Ru+Ru^2)t^{-1}u^{-1}\\
					&\subset&z^{k+1}u_3tsu^{-1}+z^{k+1}u_3t(stutu)u^{-1}t^{-2}u^{-1}+z^{k+1}u_3tstu^2
					t^{-1}u^{-1}\\
					&\subset& z^{k+1}u_3(R+Rt^{-1})(R+Rs^{-1})u^{-1}+z^{k+2}u_3tu^{-1}(R+Rt^{-1})u^{-1}+
					z^{k+1}u_3tstu^2(R+Rt)u^{-1}\\
					&\subset&\underline{\underline{z^{k+1}u_3u_1u_3u_1}}+\underline{\underline{(z^{k+1}+z^{k+2})u_3u_2u_3u_1}}+
					
					z^{k+1}u_3(t^{-1}s^{-1}u^{-1}t^{-1}u^{-1})ut+\\&&+
					z^{k+2}u_3t(u^{-1}t^{-1}u^{-1}t^{-1}s^{-1})st+z^{k+1}u_3t(stutu)u^{-1}t^{-1}+
					z^{k+1}u_3t(stutu)u^{-1}t^{-1}utu^{-1}\\
					&\subset&U+\underline{(z^k+z^{k+1})u_3u_2u_1u_2}+\underline{\underline{z^{k+2}u_3u_2u^{-1}u_2}}+z^{k+2}u_3tu^{-1}(R+Rt)utu^{-1}\\
					&\subset& U+\underline{\underline{z^{k+2}u_3u_2u^{-1}u_2}}+z^{k+2}u_3tu^{-1}(tutus)s^{-1}u^{-2}\\
					&\subset&U+z^{k+3}u_3tu^{-1}u_1u_3.
					\end{array}}$
				\\\\
				The result follows from lemma \ref{stuuuu}(iii).\\
				\item 	\underline{$k\in\{6,\dots,11\}$}:
				\\
				$\hspace*{-0.2cm}\small{\begin{array}[t]{lcl}
					z^ku_3v_8t&=&z^ku_3tstst\\
					&\subset&z^ku_3(R+Rt^{-1})(R+Rs^{-1})(R+Rt^{-1})(R+Rs^{-1})(R+Rt^{-1})\\

					\end{array}}$\\
				
				$\hspace*{-0.6cm}\small{\begin{array}[t]{lcl}
					\phantom{z^ku_3v_8t}
					
					&\subset&z^ku_3u_1u_2u_1u_2+z^kt^{-1}s^{-1}t^{-1}s^{-1}\\
					&\stackrel{\ref{ststt}(ii)}{\subset}& U+z^ku_3(t^{-1}s^{-1}u^{-1}t^{-1}u^{-1})utu^2tu(u^{-1}t^{-1}u^{-1}t^{-1}s^{-1})t^{-1}\\
					&\subset&U+z^{k-2}u_3t(R+Ru+Ru^{-1})tut^{-1}\\
					&\subset&U+\underline{\underline{z^{k-2}u_3t^2ut^{-1}}}+z^{k-2}u_3(tutus)s^{-1}t^{-1}+
					z^{k-2}u_3tu^{-1}(R+Rt^{-1})ut^{-1}\\
					&\subset&U+\underline{z^{k-1}u_3s^{-1}t^{-1}+z^{k-2}u_3}+z^{k-2}u_3(R+Rt^{-1})u^{-1}t^{-1}ut^{-1}\\
					&\subset&U+\underline{\underline{z^{k-2}u_3t^{-1}ut^{-1}}}+z^{k-2}u_3(u^{-1}t^{-1}u^{-1}t^{-1}s^{-1})su^{-1}t^{-1}\\
					&\subset&U+\underline{\underline{z^{k-3}u_3su^{-1}t^{-1}}}.\phantom{================================..}	\qedhere
					\end{array}}$
			\end{itemize}
		\end{proof}
		We can now prove the main theorem of this section.
		\begin{thm} $H_{G_{15}}=U$.
			
			\label{thm 15}
		\end{thm}
		\begin{proof}
			By proposition \ref{prrrr1} it is enough to prove that $tU\subset U$. By remark \ref{rem15} and proposition \ref{prrr15}, we only have to prove that $z^ktu_3\subset U$. By lemma \ref{tutt} (iii), we only have to check the cases where $k\in\{0,11\}$. We have:
			\begin{itemize}[leftmargin=*]
				\item \underline{$k=0$}:\\
				$\hspace*{-0.2cm}\small{\begin{array}[t]{lcl}
					tu_3&=&t(R+Ru+Ru^2)\\
					&\subset&\underline{t}+\underline{\underline{tu}}+tu^2\\
					&\subset& U+s^{-1}(stutu)u^{-1}t^{-1}u\\
					&\subset& U+zs^{-1}u^{-1}(R+Rt)u\\
					&\subset& U+\underline{zs}+zs^{-1}u^{-2}(utust)t^{-1}s^{-1}\\
					&\subset&U+zu_1u_3t^{-1}s^{-1}\\
					&\stackrel{\ref{ststt}(iii)}{\subset}&U+Uu_2u_1.
					\end{array}}$\\
				
				\item \underline{$k=11$}:\\
				$\hspace*{-0.2cm}\small{\begin{array}[t]{lcl}
					z^{11}tu_3&\subset&z^{11}(R+Rt^{-1})(R+Ru^{-1}+Ru^{-2})\\
					&\subset&\underline{z^{11}u_3}+\underline{\underline{z^{11}t^{-1}u^{-1}}}+z^{11}t^{-1}u^{-2}\\
					&\subset& U+z^{11}u(u^{-1}t^{-1}u^{-1}t^{-1}s^{-1})stu^{-1}\\
					&\subset&U+z^{10}u_3s(R+Rt^{-1})u^{-1}\\
					&\subset&U+\underline{\underline{z^{10}u_3su^{-1}}}+z^{10}u_3su(u^{-1}t^{-1}u^{-1}t^{-1}s^{-1})st\\
					&\subset&U+z^9u_3sust\\
					&\stackrel{\ref{ststt}(iii)}{\subset}&U+Uu_1u_2.
					\end{array}}$\\
			\end{itemize}
			The result follows from remark \ref{rem15} and proposition \ref{prrr15}.
		\end{proof}
		
		\begin{cor}
			The BMR freeness conjecture holds for the generic Hecke algebra $H_{G_{15}}$.
		\end{cor}
		\begin{proof}
			By theorem \ref{thm 15} we have that $H_{G_{15}}=U=\sum\limits_{k=0}^{11}z^k(u_3+u_3s+u_3t+u_3ts+u_3st+u_3tst+u_3sts+u_3tsts)$. The result follows from proposition \ref{BMR PROP}, since $H_{G_{15}}$ is generated as left $u_3$-module by 96 elements and, hence, as
			$R$-module by $|G_{15}|=288$ elements (recall that $u_3$ is generated as $R$-module by 3 elements).
		\end{proof}

	\end{document}